\documentclass[11pt,twoside]{article}
\usepackage{amsmath}
\usepackage{amssymb}
\usepackage{amsthm}
\usepackage{mathrsfs}
\usepackage{simplewick}
\usepackage{times}
\usepackage{color}
\usepackage{hep}
\usepackage{leftidx}
\usepackage{graphicx}
\usepackage{float}
\usepackage{enumerate}
\usepackage{titletoc}
\usepackage{epstopdf}
\usepackage{ulem}
\usepackage{appendix}
\usepackage{enumitem}
\usepackage{hyperref}
\usepackage{booktabs}

\allowdisplaybreaks

\def\cQ{\mathcal Q}

\def\cX{\mathcal X}
\def\cY{\mathcal Y}

\def\na{\nabla}
\def\N{\mathop{\mathbb N\kern 0pt}\nolimits}
\def\Z{\mathop{\mathbb Z\kern 0pt}\nolimits}
\def\Q{\mathop{\mathbb Q\kern 0pt}\nolimits}
\def\R{\mathop{\mathbb R\kern 0pt}\nolimits}
\def\T{\mathop{\mathbb T\kern 0pt}\nolimits}
\def\C{\mathop{\mathbb C\kern 0pt}\nolimits}

\def\ds{\displaystyle}
\def\f{\frac}

\def\p{\partial}

\def\ve{\varepsilon}

\def\ls{\lesssim}

\newcommand{\w}[1]{\langle {#1} \rangle}

\newcommand{\pk}{\dot{P}_k}
\newcommand{\Pk}{P_k}

\definecolor{mygreen}{RGB}{50,145,80}

\newcommand{\Ltw}[1]{\|#1\|_{L^2(w)}}
\newcommand{\Atwo}[1]{[#1]_{A_2}}
\newcommand{\Rj}[1]{R_{#1}}

\DeclareMathOperator{\Div}{div}

\DeclareMathOperator{\supp}{supp}
\DeclareMathOperator{\rot}{rot}
\newcommand{\Sph}{\mathbb{S}^2}

\hypersetup{colorlinks=true,linkcolor=blue,citecolor=red,urlcolor=cyan}

\theoremstyle{plain}
\newtheorem{theorem}{Theorem}[section]

\newtheorem{lemma}[theorem]{Lemma}

\theoremstyle{definition}
\newtheorem{definition}[theorem]{Definition}
\theoremstyle{assumption}

\newtheorem{remark}{Remark}[section]

\numberwithin{equation}{section}

\title{Global classical solutions to 3D  irrotational compressible Euler equations of Chaplygin gases with weakly decaying initial data}

\begin{document}
	\author{
		Gao Mu$^{1}$, \quad Yin Huicheng$^{2,}$
		\footnote{Gao Mu (\texttt{876117149@qq.com}, \texttt{mu\_gao@smail.nju.edu.cn}) is supported by the National key research and development program of China (No.2024YFA1013301).
			Yin Huicheng (\texttt{huicheng@nju.edu.cn}, \texttt{05407@njnu.edu.cn}) is supported by the NSFC (No.12331007, No.12101304).}
		\\[0.5cm]
		\small
		$^{1}$ School of Mathematics, Nanjing University, Nanjing 210009, China\\
		\small
		$^{2}$ School of Mathematical Sciences and IMS, Nanjing Normal University, Nanjing 210023, China
	}
	\date{}
	\maketitle
	
	\thispagestyle{empty}

	\begin{abstract}
We are concerned with the global classical solution problem of 3D compressible isentropic Euler equations of Chaplygin gases
	\begin{equation*}
		\left\{
		\begin{aligned}
			&\p_t\rho+\Div(\rho v)=0,\\
			&\p_t(\rho v)+\Div(\rho v \otimes v)+\nabla p=0,\\
&\rho(0,x)=\bar\rho+\ve\rho_0(x), v(0,x)=\ve v_0(x),
		\end{aligned}
		\right.
	\end{equation*}
	where $\bar\rho>0$
is a constant, $\ve>0$ is small, the state equation is $p=p(\rho)=P_0-\frac{D}{\rho}$ with $P_0$ and $D$ being
some positive constants. For the  3D compressible Euler equations of Chaplygin gases,
which are a prototype of multidimensional nonlinear symmetric hyperbolic systems with totally linearly degenerate eigenvalues, there is
a basic conjecture imposed by A. Majda: it typically has a global classical solution $(\rho, v)$ with $(\rho-\bar\rho, v)\in C([0,\infty), H^s(\Bbb R^3))
\cap C^1([0,\infty), H^{s-1}(\Bbb R^3))$ when
$(\rho_0, v_0)\in H^s(\Bbb R^3)$ with $s>\f52$ unless $(\rho, v)$  itself blows up in finite
time. In this paper, under the assumptions that for any fixed constant $\mu$ with $0<\mu<1/2$, integer $N\ge 15$, $\rot v_0(x)\equiv 0$ and
		\[
		\|(\rho_0, v_0)\|_{H^{N}(\mathbb{R}^3)}+\sum_{|a|\leq 13}
		\|\langle x\rangle^{1+\mu}
		\nabla^a(\rho_0,v_0)\|_{L^2(\mathbb{R}^3)} \le 1,
		\]
we show that the classical solution $(\rho, v)$ exists globally. Our main ingredients include: establishing  a series of new
decay estimates of energy bounds, weighted  pointwise space-time $L^\infty$-$L^2$ estimates and weighted Strichartz-type estimates
for the 3D potential flow equation of Chaplygin gases.
		
		\vskip 0.2 true cm
		\noindent

\qquad\quad strong Huygens' principle, weighted space-time estimate, global existence
		
		\vskip 0.2 true cm
		\noindent
	\end{abstract}
	
	\vskip 0.2 true cm
	
	\addtocontents{toc}{\protect\thispagestyle{empty}}
	\tableofcontents

	\section{Introduction}

	\subsection{Main results and remarks}

	The 3D compressible isentropic Euler equations are
	\begin{equation}\label{Euler}
		\left\{
		\begin{aligned}
			&\p_t\rho+\Div(\rho v)=0,\\
			&\p_t(\rho v)+\Div(\rho v \otimes v)+\nabla p=0,\\
		\end{aligned}
		\right.
	\end{equation}
	where $(t,x)\in [0,+\infty)\times\mathbb{R}^3$, $t=x_0$, $x=(x_1, x_2, x_3)$, $\p_k=\p_{x_k}$
($0\le k\le 3$), $\nabla=(\p_1,\p_2, \p_3)$, $\p=(\p_0, \p_1, \p_2, \p_3)$, $\rho$, $v=(v_1,v_2,v_3)$ and
	$p$ stand for the density, velocity and  pressure, respectively.
	The equation of state for Chaplygin gases (see \cite{CF}) is
	\begin{equation}\label{pressure2}
		p(\rho)=P_0-\frac{D}{\rho},
	\end{equation}
	where $P_0$ and $D$ are some positive constants.

Suppose that \eqref{Euler} is equipped with the following small perturbed initial data:
\begin{equation}\label{irrot:condition}
\begin{split}
&(\rho,v)(0,x)=(\bar\rho+\ve\rho_0(x),\ve v_0(x)),\\
\end{split}
\end{equation}
where $\bar\rho>0$ is a constant, $\ve>0$ is small, $v_0(x)=(v_{1,0}(x), v_{2,0}(x), v_{3,0}(x))$ and $\bar\rho+\ve\rho_0(x)>0$.

It is pointed out that  for Chaplygin gases, \eqref{Euler} is a prototype of multidimensional nonlinear symmetric hyperbolic
systems with totally linearly degenerate eigenvalues (see  Page 89 of \cite{Majda}). So far there has been
a basic and unsolved conjecture (Page 89 of \cite{Majda}):

\vskip 0.1 true cm

{\bf Conjecture.} \eqref{Euler} with \eqref{pressure2}-\eqref{irrot:condition}  has a global classical solution
 $(\rho, v)$ with $(\rho-\bar\rho, v)\in C([0,\infty),$ $H^s(\Bbb R^3))\cap C^1([0,\infty), H^{s-1}(\Bbb R^3))$
 when  $(\rho_0(x), v_0(x))\in H^s(\Bbb R^3)$ with $s>\f52$ unless the solution $(\rho,v)$ itself blows up in finite
time.

\vskip 0.1 true cm

By our knowledge, this conjecture has not been solved yet for any
small perturbed initial data in $H^s(\Bbb R^3)$. As illustrated in Page 89 of \cite{Majda}, the above conjecture is mainly of
mathematical interest but its resolution would elucidate both the nonlinear nature
of the conditions requiring linear degeneracy of each wave field and also might
isolate the fashion in which the shock wave formation arises in quasilinear hyperbolic
systems. In this paper, we are concerned with the
global classical solution problem of \eqref{Euler} with irrotational and weakly decaying
$(\rho_0(x), v_0(x))$. In this case, it holds that
\begin{equation}\label{irrot:condition-1}
		\begin{split}
		&\p_kv_{j,0}=\p_jv_{k,0},\quad 1\le k<j\le3.
		\end{split}
	\end{equation}

Without loss of generality, $\bar\rho$ is
normalized so that the sound speed becomes
\begin{equation}\label{YHC-0}
\begin{split}
c(\bar\rho)=\sqrt{p'(\bar\rho)}=1,
\end{split}
\end{equation}
namely, $\f{D}{{\bar\rho}^2}=1$ holds.

In addition, we assume that
$(\rho_0(x), v_0(x))$ fulfills
\begin{equation}\label{YHC-1}
\begin{split}
\|(\rho_0,v_0)\|_{H^N(\mathbb{R}^3)}
+\sum_{|a|\leq 13}\|\langle x\rangle^{1+\mu}
\nabla^a(\rho_0,v_0)\|_{L^2(\mathbb{R}^3)}
\leq 1,
\end{split}
\end{equation}
where $0<\mu<1/2$ is some fixed constant, integer $N\ge 15$, $\langle x\rangle=(1+|x|^2)^{\f12}$ and $|x|=\sqrt{x_1^2+x_2^2+x_3^2}$.
From now on, we set $A=A(t,x)=1+t+|x|$ and $B=B(t,x)=1+\bigl|t-|x|\bigr|$. Let $C$ be the generic positive constant depending on $\mu$.

Our main result can be stated as
\begin{theorem}\label{thm:main-global-1}
For \eqref{Euler} with \eqref{pressure2}-\eqref{YHC-1},
there exists a positive constant $\varepsilon_0$ depending on $\mu$ such that for
$\varepsilon\in(0,\varepsilon_0)$, \eqref{Euler} has a global classical
solution $(\rho, v)$ with $(\rho-\bar\rho, v)\in C([0,\infty);H^{N}(\mathbb R^3))\cap C^1([0,\infty);H^{N-1}(\mathbb R^3))$
and $\rho(t,x)>\f{\bar\rho}{2}$ for all $(t,x)\in [0,\infty)\times\mathbb R^3$.
Moreover, the following estimates hold:
\begin{equation}\label{YHCCC-01}
	\sup_{t\geq 0}
	(1+t)^{-\frac{\mu}{100}}
	\|(\rho-\bar\rho,v)(t)\|_{H^N(\mathbb{R}^3)}
	\leq C\varepsilon,
\end{equation}
\begin{equation}\label{YHCCC-03}
	\sum_{|a|\leq 7}
	\big\|
	A B^{\frac{\mu}{4}}
	\nabla^a(\rho-\bar\rho,v)
	\big\|_{L^\infty([0,\infty)\times\mathbb{R}^3)}
	\leq C\varepsilon.
\end{equation}
\end{theorem}

\begin{remark}\label{YHC-90}
	We can use the following condition instead of \eqref{YHC-1}
	\begin{equation}\label{YHC-91}
		\|(\rho_0,v_0)\|_{H^N(\mathbb R^3)}
		+\||x|^{1+\mu}(\rho_0,v_0)\|_{L^2(\mathbb R^3)}
		\leq 1,
	\end{equation}
	where
\begin{equation}\label{YHC-999}
N>\frac{13(1+\mu)}{\mu}.
\end{equation}
In this case, Theorem \ref{thm:main-global-1} still holds.
	In fact, it follows from the weighted Gagliardo-Nirenberg inequality that there exists some constant $\mu'$
with $0<\mu'<\mu$ such that
	\begin{equation}\label{WY-1}
		\sum_{|a|\leq 13}
		\big\|
		\langle x\rangle^{1+\mu'}
		\nabla^a(\rho_0,v_0)
		\big\|_{L^2(\mathbb R^3)}\leq C,
	\end{equation}
which implies that a condition analogous to \eqref{YHC-1} holds.
		\end{remark}

\begin{remark}\label{YHC-90E}
Consider the 3D irrotational gravity water wave equation
\begin{equation}\label{GWW}
\left\{
\begin{aligned}
&\p_th=G(h)\phi,\\
&\p_t\phi=-h-\frac12|\nabla\phi|^2
+\frac{(G(h)\phi+\nabla h\cdot\nabla\phi)^2}{2(1+|\nabla h|^2)},\\
&h(0,x)=h_0(x), \phi(0,x)=\phi_0(x),
\end{aligned}
\right.
\end{equation}
where $x\in\R^2$, $\nabla=(\p_1,\p_2)$, $h=h(t,x)$, $\phi=\phi(t,x)=\Phi(x,h(t,x))$, $\Phi(x,z)$
is the velocity potential, and $G(h)\phi=\sqrt{1+|\na h|^2}\p_n\Phi$
is the Dirichlet-Neumann operator associated to the domain $\Omega(t)=\{(x,z)\in\R^2\times\R: z<h(t,x)\}$.
Set $\|v\|_Z=\|\w{x}^\mu (1-\Delta)^4v\|_{L^2}$ with $\mu\in (0,1)$ and $U=h+i|\na|^{\f12}(\phi-T_{\p_z\Phi}h|_{z=h(t,x)})$
with $i=\sqrt{-1}$. In \cite{Zheng22}, the author shows that when
\begin{equation}\label{YHC-9100}
N\ge\max\{\frac{33}{\mu(1-\mu)},\f{8}{\mu^2}\}
\end{equation}
and
\begin{equation*}\label{YHC-9100A}
\|h_0\|_{H^{N+\frac12}}+\||\nabla|^{\f12}\phi_0\|_{H^N}
+\|h_0\|_Z+\||\nabla|^{\frac12}\phi_0\|_Z\le\ve,
\end{equation*}
\eqref{GWW} has a global solution $(h, \phi)$ such that $U\in C([0, +\infty), H^N(\R^2))$.
Note that \eqref{GWW} can be changed into a nonlinear Klein-Gordon equation of $U$.
Comparing \eqref{YHC-999} and \eqref{YHC-9100}, one knows that both $\ds\lim_{\mu\to 0+}N=+\infty$ hold.

\end{remark}

In order to prove Theorem \ref{thm:main-global-1}, we now give the reformulation on \eqref{Euler}
under the irrotationality of velocity $v$.
By \eqref{irrot:condition-1} and \eqref{Euler}, one has $\rot v(t,x)\equiv 0$ for all
$(t,x)\in [0, T]\times\Bbb R^3$ as long as $(\rho-\bar\rho, v)\in C([0, T],$ $H^s(\Bbb R^3))\cap C^1([0,T], H^{s-1}(\Bbb R^3))$
with $T>0$ and $s>\f52$. Then we can introduce a potential function $\phi(t,x)$ such that $v=\nabla\phi$ and $\ds\lim_{|x|\to\infty}\phi=0$.
In this case,
it follows from the second equation in \eqref{Euler} that the
	Bernoulli's law holds:
\begin{equation}\label{YHC-2}
\begin{split}
\p_t\phi+\frac12|\nabla\phi|^2+h(\rho)=0,
\end{split}
\end{equation}
where the enthalpy $h(\rho)$ satisfies $h'(\rho)=\frac{p'(\rho)}\rho$ and $h(\bar\rho)=0$.

By \eqref{pressure2}, \eqref{YHC-0} and \eqref{YHC-2}, one has that from the first equation in \eqref{Euler}
that $\phi$ fulfills the following second order quasilinear wave equation
	\begin{equation}\label{potentia1}
		\Box\phi=2\p_t\phi\Delta\phi-2\sum_{k=1}^3\p_k\phi\p_{tk}^2\phi
		+|\nabla\phi|^2\Delta\phi-\sum_{k,j=1}^3\p_k\phi\p_j\phi\p_{kj}^2\phi
	\end{equation}
with the initial data $(\phi,\p_t\phi)|_{t=0}=\ve(\phi_0,\phi_1)$ as
\begin{equation}\label{YHC-3}
\begin{split}
&\phi_{0}(x)=\int_{-\infty}^{x_1}v_{1,0}(s,x_2,x_3)ds,\\
&\phi_{1}(x)=-\frac12\ve(|v_{1,0}|^2+|v_{2,0}|^2+|v_{3,0}|^2)-\f{h(\bar\rho+\ve\rho_0)}{\ve},
\end{split}
\end{equation}
where $\Box=\p_t^2-\Delta$ and $\Delta=\p_1^2+\p_2^2+\p_3^2$. Under the assumption \eqref{YHC-1}, one has
\begin{equation}\label{YHC-4}
\begin{split}
\|\phi_0\|_{H^{N+1}(\mathbb{R}^3)}+\|\phi_1\|_{H^{N}(\mathbb{R}^3)} + \sum_{|a|\le 13} \|\langle x\rangle^{1+\mu} \nabla^a(\na\phi_0,\phi_1)\|_{L^2(\mathbb{R}^3)} \le C_0,
\end{split}
\end{equation}
where $C_0>0$ is some fixed constant. Indeed, by $v_0=\nabla\phi_0$, we have
$\Delta \phi_0=\Div v_0$. Hence, it follows from $\ds\lim_{|x|\to\infty}\phi_0=0$,
Hardy's inequality and \eqref{YHC-1} that
$$\|\phi_0(x)\|_{L_x^2}=\|\hat\phi_0(\xi)\|_{L_{\xi}^2}=\bigg\|\ds\frac{\sum_{j=1}^3\xi_j\hat v_{0,j}(\xi)}{|\xi|^2}\bigg\|_{L^2_{\xi}}
\le\bigg\|\f{\hat v_0(\xi)}{|\xi|}\bigg\|_{L^2_{\xi}}\ls \|\nabla_{\xi}\hat v_0(\xi)\|_{L^2_{\xi}}\le\||x|v_0(x)\|_{L_x^2}\le C.$$
In addition, $\p\phi_0\in H^{N}(\Bbb R^3)$ obviously holds. Together with \eqref{YHC-1} and \eqref{YHC-3}, \eqref{YHC-4} is shown.
	
Next, we focus on the equivalent 3D Cauchy problem:
\begin{equation}\label{eq:Chaplygin-Cauchy}
\left\{
\begin{aligned}
&\Box\phi=2\partial_t\phi\,\Delta\phi-
2\sum_{k=1}^3\partial_k\phi\,\partial^2_{tk}\phi+
|\nabla\phi|^2\Delta\phi-\sum_{k,j=1}^3\partial_k\phi\,\partial_j\phi\,\partial^2_{kj}\phi,\\
&(\phi, \partial_t\phi)(0,x)=\ve(\phi_0(x),\phi_1(x)).
\end{aligned}
\right.
\end{equation}

\begin{theorem}\label{thm:main-global}
Under the condition \eqref{YHC-4},
there exists a positive constant $\varepsilon_0$ depending on $\mu$ such that for
$\varepsilon\in(0,\varepsilon_0)$, \eqref{eq:Chaplygin-Cauchy} has a global classical
solution $\phi\in C([0,\infty);H^{N+1}(\mathbb R^3))\cap C^1([0,\infty);H^N(\mathbb R^3))$.
Moreover, the following estimates hold:
		\begin{equation}\label{YHCCC-1}
			\sup_{t\ge 0}(1+t)^{-\f\mu{100}}\|\partial\phi(t)\|_{H^N(\mathbb R^3)}
			\le C\varepsilon,
		\end{equation}
		
		\begin{equation}\label{YHCCC-2}
			\sup_{t\ge 0}(1+t)^{-\f\mu{100}}\sum_{\Gamma}\|\partial\Gamma\phi(t)\|_{H^{11}(\mathbb R^3)}
			\le C\varepsilon,
		\end{equation}
		
		\begin{equation}\label{YHCCC-3}
			\bigl\|AB^{\frac{\mu}{4}}\phi\bigr\|_{L^\infty([0,\infty)\times\mathbb R^3)}
			+\sum_{|a|\le 7}\bigl\|AB^{\frac{\mu}{4}}\nabla^a\partial\phi\bigr\|_{L^\infty([0,\infty)\times\mathbb R^3)}
			\le C\varepsilon,
		\end{equation}
		
		\begin{equation}\label{YHCCC-4}
			\sum_{\Gamma}\bigl\|A^{\frac{\mu}{2}}\Gamma\phi\bigr\|_{L^\infty([0,\infty)\times\mathbb R^3)}
			+\sum_{\Gamma}\sum_{|a|\le 6}\bigl\|A^{\frac{\mu}{2}}\partial\nabla^a\Gamma\phi\bigr\|_{L^\infty([0,\infty)\times\mathbb R^3)}
			\le C\varepsilon,
		\end{equation}
		
		\begin{equation}\label{YHCCC-5}
			\sum_{\Gamma}\bigl\|A^{\frac{\mu}{3}}\Gamma\phi\bigr\|_{L^2([0,\infty);L^\infty(\mathbb R^3))}
			+\sum_{\Gamma}\sum_{|a|\le 6}\bigl\|A^{\frac{\mu}{3}}\partial\nabla^a\Gamma\phi\bigr\|_{L^2([0,\infty);L^\infty(\mathbb R^3))}
			\le C\varepsilon,
		\end{equation}
	where and below $\Gamma\in
	\{L_j,\Omega_{kl}: 1\le j\le3, 1\le k<l\le3\}$ with $L_j=x_j\p_t+t\p_j$ and $\Omega_{kl}=x_k\p_l-x_l\p_k$.
	\end{theorem}

\begin{remark}\label{Rem-1}	
When $(\rho_0(x), v_0(x))\in C_0^{\infty}(\Bbb R^3)$ and $\rot v_0(x)\equiv 0$,
the global existence of smooth solution to \eqref{Euler} with \eqref{pressure2}-\eqref{irrot:condition}
can be easily derived due to the null form of nonlinearity in \eqref{potentia1} (see \cite{Christodoulou} and \cite{Klainerman2}).
For the corresponding 2D case of \eqref{Euler}
with \eqref{pressure2}-\eqref{irrot:condition}, when $(\rho_0(x), v_0(x))\in C_0^{\infty}(\Bbb R^2)$ and $\rot v_0(x)\equiv 0$,
the smooth solution $(\rho, v)$ also exists globally in terms of \cite{Alinhac1}. In addition, for 2D compressible full Euler equations
of Chaplygin gases, when
the initial and rotational axisymmetric perturbation of a rest state is small, the authors in \cite{HouYin3-0}-\cite{HouYin3-1}
and \cite{Wei}
have proved the global existence of smooth solution $(\rho, v)$.
\end{remark}

\begin{remark}\label{Rem-3}
When $(\rho_0(x), v_0(x))\in C_0^{\infty}(\Bbb R^d)$ with $d\ge 2$ and $\rot v_0(x)\not\equiv 0$, so far
the {\bf Conjecture} is still completely open
except the global solvability on the  small symmetric solution of  \eqref{Euler} with \eqref{pressure2}-\eqref{irrot:condition} (see
\cite{HouYin3-0}-\cite{HouYin3-1} and \cite{Wei}).
\end{remark}

\begin{remark}\label{Rem-3-0}
		We point out that the argument developed in this paper is not restricted to
		\eqref{potentia1}. It can be adapted to more general
		quasilinear wave systems in \cite{HouTaoYin2}:
		\begin{equation}\label{QWE}
			\left\{
			\begin{aligned}
				&\Box u^\iota=\sum_{j,k=1}^m\sum_{|a|+|b|\le1}C^{ab}_{\iota jk}Q_0(\p^au^j,\p^bu^k)
				+\sum_{j,k,l=1}^m \sum_{\alpha,\beta,\mu=0}^3N^{\alpha\beta\mu}_{\iota jkl}\p^2_{\alpha\beta}u^j\p_{\mu}u^k u^l\\
				&\qquad\quad+\sum_{j,k,l=1}^m  \sum_{\alpha,\beta=0}^3N^{\alpha\beta}_{\iota jkl}\p_{\alpha}u^j\p_{\beta}u^ku^l
+\sum_{j,k,l=1}^m\sum_{\alpha,\beta,\mu,\nu=0}^3N_{\iota jkl}^{\alpha\beta\mu\nu}\p^2_{\alpha\beta}u^j\p_{\mu}u^k\p_{\nu}u^l,\\
				&(u,\p_tu)(0,x)=\ve (u_{(0)},u_{(1)})(x),
			\end{aligned}
			\right.
		\end{equation}
		where $\iota =1,\cdots,m$, $u=(u^1,\ldots,u^m)$, $(t,x)\in[0,+\infty)\times\mathbb R^3$,
		$
		Q_0(f,g)=\p_t f\,\p_t g-\ds\sum_{j=1}^3\p_j f\,\p_j g,
		$
		and $u_{(r)}=(u^1_{(r)},\ldots,u^m_{(r)})$ for $r=0,1$. We rewrite 		
		the first term on the right-hand side of \eqref{QWE} as
		\begin{equation}\label{NL:rewrite}
			\sum_{|a|+|b|\le1}
			C^{ab}_{\iota jk}Q_0(\p^a u^j,\p^b u^k)
			=
			\sum_{\alpha,\beta=0}^3
			Q_{\iota jk}^{\alpha\beta}
			\p_\alpha u^j\p_\beta u^k
			+
			\sum_{\alpha,\beta,\mu=0}^3
			Q_{\iota jk}^{\alpha\beta\mu}
			\p^2_{\alpha\beta}u^j\p_\mu u^k .
		\end{equation}
		It is easy to know that the nonlinearity in \eqref{NL:rewrite} satisfies  the null condition, namely,
		for any \((\xi_0,\xi_1,\xi_2,\xi_3)\in\{\pm1\}\times\mathbb S^2\), it holds
		\begin{equation}\label{NC:rewrite}
			\sum_{\alpha,\beta=0}^3
			Q_{\iota jk}^{\alpha\beta}\xi_\alpha\xi_\beta
			\equiv0,
			\qquad
			\sum_{\alpha,\beta,\mu=0}^3
			Q_{\iota jk}^{\alpha\beta\mu}
			\xi_\alpha\xi_\beta\xi_\mu
			\equiv0 .
		\end{equation}
		In \cite{HouTaoYin2}, the authors introduce the following good unknown		
		\[ V^\iota=u^\iota-\frac12\sum_{j,k=1}^m\sum_{|a|+|b|\le1}C^{ab}_{\iota jk}\p^au^j\p^bu^k.\]
		Then it follows from the nonlinear wave system in \eqref{QWE} and direct computation that
		\begin{equation}
			\label{eq:sketch-Vi-equation}
			\Box V^\iota
			=\mathcal R^\iota_3(u)+\mathcal R^\iota_4(u)+\mathcal R^\iota_5(u),\qquad \iota= 1,\cdots,m,
		\end{equation}
		where $\mathcal R^\iota_3,\mathcal R^\iota_4$ and $\mathcal R^\iota_5$ are  cubic,
		quartic and quintic  nonlinearity, respectively. Moreover,  \(\mathcal R^\iota_k\) for $3\le k\le 5$
satisfy the corresponding null condition. Therefore, \eqref{eq:sketch-Vi-equation} admits a complete analogous structure of the crucial equation \eqref{eq:box-V-schematic-null} below. From \eqref{eq:sketch-Vi-equation} and \eqref{QWE},
by the same argument in the paper, the conclusion similar to Theorem \ref{thm:main-global} will hold
under the assumption of $(u_{(0)},u_{(1)})$ as in \eqref{YHC-4}.
\end{remark}

\begin{remark}\label{Rem-2}	
For the 3D quadratic quasilinear wave system
\begin{equation}\label{eq:wave}
\left\{
\begin{aligned}
\square u= &G(\p u,\p^2 u), \qquad\quad (t,x)\in [0,\infty)\times\R^3,\\
(u,\p_tu)&(0,x)=\ve (u_{0},u_{1})(x),\qquad x\in\R^3,
\end{aligned}
\right.
\end{equation}
where $(u_0, u_1)(x)\in (H^{N+1}, H^N)(\R^3)$
with  $N\in\Bbb \Z_+$ being sufficiently large, and the quadratic nonlinearity $G(\p u,\p^2 u)$ satisfies the null condition,
when
\begin{equation}\label{YHCC-00}
\begin{split}
\|x(|\nabla| u_0, u_1)\|_{H^2}+\||\nabla|(|x|^2 (|\nabla|u_0, u_1))\|_{H^1} +
\|(|\nabla|u_0, u_1)\|_{H^{N}}\le 1,
\end{split}
\end{equation}
the authors in \cite{Pusateri} have established the global classical solution $u\in C([0,\infty);H^{N+1}(\mathbb R^3))\cap C^1([0,\infty);$
$H^N(\mathbb R^3))$
by the space-time resonance method. It is pointed out that
the weights $x$ and $|x|^2$ in \eqref{YHCC-00} play essential roles in the proof procedure on the global existence of
$u$ (see the norm in (3.6) of \cite{Pusateri}). However, for the lower weighted Sobolev space in \eqref{YHC-4}
with the weight $\w{x}^{1+\mu}$ ($\mu>0$ is any small fixed constant), it seems difficult for us to apply the corresponding methods and results
in \cite{Pusateri} directly. In the present paper, we will
establish a new series of weighted \(L^\infty\)-\(L^2\) and
Strichartz-type estimates instead of the \(L^\infty\) and weighted
\(H^k\) framework in (3.5)-(3.6) of \cite{Pusateri} to obtain the uniform bounds
\eqref{YHCCC-1}-\eqref{YHCCC-5}.
\end{remark}

\begin{remark}\label{Rem-4}	
For the 3D cubic quasilinear wave system with multiple speeds
\begin{equation}\label{eq:wave-A}
\left\{
\begin{aligned}
\square_{c_k} u^k=&G^k(u,\p u,\p^2 u),\qquad k=1,\cdots,m, \\
(u,\p_tu)&(0,x)=\ve(u_{0},u_{1})(x),
\end{aligned}
\right.
\end{equation}
where $\square_{c_k}=\p_t^2-c_k^2\Delta$ ($c_k\not=0$), $G^k(u,\p u,\p^2 u)=O(|u|^3+|\p u|^3+|\p^2 u|^3)$,
the authors in \cite{GaoLiYin2026} have shown that there is a constant $\ve_0>0$ such that for $\ve\in(0,\ve_0)$,
$N\ge 6$, $\mu\in (0,1)$, and $(u_{0},u_{1})$ satisfies
\begin{equation}\label{initial:data2A}
\|u_{0}\|_{H^{N+1}(\R^3)}+\|u_1\|_{H^N(\R^3)}+\sum_{|a|\le 5}\|\langle x\rangle^{\mu}\p^a_x(\nabla u _0,u_1)\|_{L^2(\R^3)}\le 1,
\end{equation}
\eqref{eq:wave-A} admits a global solution $u\in C([0, \infty);H^{N+1}(\R^3))\cap C^1([0, \infty);H^N(\R^3))$
with the following estimates
\begin{equation}\label{YHCCC-1A}
\|\p u\|_{L^\infty([0,\infty);H^N(\R^3))}\le C\varepsilon,
\end{equation}
\begin{equation}\label{YHCCC-2A}
(1+t)^{\mu^-}\big(\|u(t,\cdot)\|_{L^\infty(\R^3)}+\|\p u(t,\cdot)\|_{L^\infty(\R^3)}\big)\le C\varepsilon,
\end{equation}
\begin{equation}\label{YHCCC-3A}
\|(1+t+|x|)^{\f{1}{2}\mu^-}u\|_{L^{2}([0,\infty); L^\infty(\R^3))}
+\sum_{|b|\le 1}\|(1+t+|x|)^{\f{1}{2}\mu^-}\p_x^b\p u\|_{L^{2}([0,\infty); L^\infty(\R^3))}\le C\varepsilon,
\end{equation}
where $\mu^-$ denotes any fixed positive constant less than $\mu$. Compared with the cubic nonlinearity in \eqref{eq:wave-A},
where the related weight  $\langle x\rangle^{\mu}$ is imposed for the initial data $(u_{0},u_{1})$ in \eqref{initial:data2A},
the required weight $\langle x\rangle^{1+\mu}$ in \eqref{YHC-4}
seems to be rather reasonable for the quadratic nonlinearity in \eqref{eq:Chaplygin-Cauchy} with null form.
\end{remark}

\begin{remark}\label{rmk:1.4-1}
Consider the Cauchy problem of the 3D quasilinear Klein-Gordon equation
\begin{equation*}\label{KG-01}
\left\{
\begin{aligned}
&\Box u+u=F(u,\p u,\p^2u),\quad(t,x)\in[0,\infty)\times\R^3,\\
&(u,\p_tu)(0,x)=\ve(u_0,u_1)(x),
\end{aligned}
\right.
\end{equation*}
where $\ve>0$ is small, $(u_0,u_1)\in (H^{s+1}, H^s)$ with $s>0$ being a suitably large constant,
and the smooth nonlinearity $F(u,\p u,\p^2u)$ is quadratic
and is linear in $\p^2u$. The authors in  \cite{HouYin2} have proved the global classical solution $u\in C([0,\infty); H^{s+1})\cap  C^1([0,\infty);H^{s})$.
With respect to more results for the long time existence of small solutions to 2D and 1D Klein-Gordon equations with slowly decaying
initial data, one can see \cite{Delort97}, \cite{DelortFang}, \cite{HouTaoYin} and \cite{Stingo18}.
It is pointed out that the crucial analyses for the Klein-Gordon equation are not suitable for
the wave equation since the Fourier integral
operators corresponding to the solution expression of linear wave equation contain the singular phase functions $\pm|\xi|$ near $\xi=0$,
so far there are few systematic long time existence results for the nonlinear wave equations with the small
slowly decaying initial data.
\end{remark}

	\subsection{Previous results}
	
	We now briefly recall some basic results related to our work.
	Consider the Cauchy problem of the nonlinear wave equation
	\begin{equation}\label{eq:1.9}
		\begin{cases}
			\square u = F(\partial u,\partial^2 u), \quad (t,x) \in [0,\infty)\times\mathbb{R}^d,\\
			(u,\partial_t u)(0,x)=\varepsilon (u_0, u_1)(x),
		\end{cases}
	\end{equation}
	where $F(\partial u,\partial^2 u)=O(|\partial u|^2+|\partial^2 u|^2)$, $d\geq 2$, $\varepsilon>0$ is sufficiently small.
	When the smooth functions $(u_0,u_1)(x)$ are compactly supported or decay rapidly at infinity, the following systematic
global existence results have been established.
	\vskip 0.1 true cm
	{\bf $\bullet$} For $d \geq 4$,  \eqref{eq:1.9}
	admits a global small data smooth solution $u$ (see \cite{Hormander}, \cite{KlainermanPonce}, \cite{LiChen}).
	\vskip 0.1 true cm
	{\bf $\bullet$} For $d = 3$, if  $F(\partial u,\partial^2 u)$ satisfies the null condition,
	then the global existence of small data solution $u$ to \eqref{eq:1.9} has been shown in \cite{Christodoulou} and \cite{Klainerman2}.
	\vskip 0.1 true cm
	{\bf $\bullet$} For $d=2$, when $F(\partial u,\partial^2 u)$ satisfies both the first and second null conditions, the author in \cite{Alinhac1} proved that \eqref{eq:1.9} has a global smooth solution $u$.
	\vskip 0.1 true cm
	
	When the initial data are not compactly supported but belong to
	higher order spatially weighted Sobolev spaces (the weight $\w{x}^\mu$ with $\mu> 2$), a lot of global and long time
	existence results have also been obtained; see, for example,
	\cite{Asa,CaiLei,HouYin1}.

For the weakly decaying initial data $(u_0, u_1)\in (H^{s+1}, H^{s})$ problem ($s>\f{d}{2}$ is suitably large)
\begin{equation}\label{YHC-9A}
\begin{cases}
&\square u
=G(\partial u,\partial^2 u)=O(|\partial u|^p+|\partial^2 u|^p),\quad p\in\Bbb N,\quad p\ge 2,\\
&(u,\partial_t u)(0,x) = (\varepsilon u_0,\varepsilon u_1)(x),\qquad x\in \Bbb R^d,
\end{cases}
\end{equation}
when some extra regularity assumptions on $(u_0, u_1)$ or nonlinearity structure are imposed, there are some interesting
long time existence works which can be roughly listed  as follows (see the table below)	

	\begin{table}[H]
		\centering
		\caption{Long-time existence results for nonlinear wave equations with weakly decaying data}
		\label{tab:Strichartz-previous-results}
		\renewcommand{\arraystretch}{1.355}
		\vskip 0.2 true cm
		\begin{tabular}{c|c|c|p{4.6cm}|p{2.5cm}|c}
			\hline
			\(d\)
			&
			{\bf Equation type}
			&
			\({\bf p}\)
			&
			{\bf Additional assumption on the initial data}
			&
			{\bf Conclusion}
			&
			{\bf Reference}
			\\
			\hline
			
			\(2\)
			&
			Semilinear
			&
			\(3\)
			&
			$
			(u_0,u_1)
			\in
			H_{\theta}^{{\frac32}^+,{\frac12}^+}
			\times
			H_{\theta}^{{\frac12}^+,{\frac12}^+}
			$
			&
			$
			T_\varepsilon\ge e^{c\varepsilon^{-2}}
			$
			&
			\cite{FangWang2013}
			\\
			\hline
			\(2\)
			&
			Semilinear
			&
			\(4\)
			&
			$
			(u_0,u_1)
			\in
			H_{\theta}^{\frac53,{\frac12}^+}
			\times
			H_{\theta}^{\frac23,{\frac12}^+}
			$
			&
			$T_\varepsilon=\infty$
			&
			\cite{FangWang2013}
			\\
			\hline
			
			\(2\)
			&
			Semilinear
			&
			\(5\)
			&
			None
			&
			$T_\varepsilon=\infty$
			&
			\cite{FangWang2013}
			\\
			\hline
			
			\(3\)
			&
			Quasilinear
			&
			\(2\), null form
			&
			$\bigl\||\nabla|\bigl(|x|^2(|\nabla|u_0,u_1)\bigr) \bigr\|_{H^1}
            +\bigl\| x(|\nabla|u_0,u_1) \bigr\|_{H^2}\le 1$
			&
			$T_\varepsilon=\infty$
			&
			\cite{Pusateri}
			\\
			\hline

			\(3\)
			&
			Quasilinear
			&
			\(3\)
			&
			None
			&
			$
			T_\varepsilon\ge e^{c\varepsilon^{-2}}
			$
			&
			\cite{GaoLiYin2026}
			\\
			\hline

			\(3\)
			&
			Semilinear
			&
			\(3\)
			&
				$
				(u_0,u_1)
				\in
					H_{\theta}^{2,0^+}
				\times
				H_{\theta}^{1,0^+}
			$
			&
			$T_\varepsilon=\infty$
			&
			\cite{MachiharaNakamuraNakanishiOzawa}
			\\
			\hline
			
			\(3\)
			&
			Quasilinear
			&
			\(3\)
			&
			$
			\big\|
			|x|^\mu
			(\nabla u_0,u_1)
			\big\|_{L^2}
			\le 1
			$
			&
			$T_\varepsilon=\infty$
			&
			\cite{GaoLiYin2026}
			\\
			\hline
			
			\(3\)
			&
			Quasilinear
			&
			\(4\)
			&
			None
			&
			$T_\varepsilon=\infty$ and \quad scattering
			&
			As in \cite{FangWang2013,HouYin3}
			\\
			\hline
			
			\(\ge4\)
			&
			Quasilinear
			&
			\(3\)
			&
			None
			&
			$T_\varepsilon=\infty$ and \quad scattering
			&
			\cite{HouYin3}
			\\
			\hline
			
			\(4\)
			&
			Semilinear
			&
			\(2\)
			&
			$
			(u_0,u_1)
			\in
			\dot B_{2,1,\Omega}^{1}
			\times
			\dot B_{2,1,\Omega}^{0}
			$
			&
			$T_\varepsilon=\infty$ and \quad scattering
			&
			\cite{Sterbenz2}
			
			\\
			\hline
				\(\ge 6\)
			&
			Semilinear
			&
			\(2\)
			&
			$
			(u_0,u_1)
			\in
			\dot{B}_{2,1}^{s_c}
			\times
			\dot{B}_{2,1}^{s_c-1}
			$
			&
			$T_\varepsilon=\infty$ and \quad scattering
			&
			\cite{Sterbenz1}
			\\
			\hline
			
		\end{tabular}
	\end{table}
	
	Here $\|f\|_{H_\theta^{s,b}}
	=\|f\|_{H^s}+
	\big\|
	(1-\Delta_{\mathbb{S}^{d-1}})^{\frac b2}f
	\big\|_{H^s}$ and $\|f\|_{\dot B_{2,1,\Omega}^{s}}
	=\|f\|_{\dot B_{2,1}^{s}}+
	\ds\sum_{1\le k<j\le4}
	\|\Omega_{kj}f\|_{\dot B_{2,1}^{s}}$.
	
	It should be emphasized that the results summarized in
	Table~\ref{tab:Strichartz-previous-results} are presented only in a simplified
	form, while the nonlinearities considered in the original references may be generally broader.

	\subsection{Sketch of the proof}
	
	We now give an outline on the proof of Theorem~\ref{thm:main-global}. When
	$(u_0,u_1)(x)$ decay suitably fast at infinity, the proofs on the global existence results mentioned above  essentially
depend on the standard energy estimates for wave
	equations, the smallness of $\|u(t,\cdot)\|_{H^{\lfloor \frac{d}{2} \rfloor+2}(\mathbb{R}^d)}$ and the following Klainerman-Sobolev inequality in \cite{Klainerman3}:
	\begin{equation}\label{eq:1.14}
		|\partial u(t,x)| \leq \frac{C}{(1+|t-r|)^{\frac{1}{2}}(1+t)^{\frac{d-1}{2}}}
		\sum_{|I| \leq \lfloor \frac{d}{2} \rfloor+1} \|Z^I \partial u(t,x)\|_{L_x^2(\mathbb{R}^d)},
	\end{equation}
	where $Z \in \{\partial_t, x_k\partial_j-x_j\partial_k, 1\leq k<j \leq d, x_l\partial_t+t\partial_l, 1\leq l\leq d, t\partial_t + \sum\limits_{k=1}^d x_k\partial_k\}$ and $r=|x|$.
	In order to apply the vector field method and \eqref{eq:1.14} to investigate problem \eqref{eq:1.9}, the weighted norm
	\begin{equation}\label{YHCCC-20}
		\begin{split}
			\|(1+|x|)^2(\nabla u_0,u_1)\|_{L^2(\mathbb{R}^d)}<\infty
		\end{split}
	\end{equation}
	should be satisfied.  However, the
	initial data in \eqref{YHC-4} contain only the spatial weight $\langle x\rangle^{1+\mu}$
($0<\mu<1/2$). Therefore the classical vector field method cannot be applied directly since
 the stronger spatial weighted Sobolev norms are required. For the same reason, as explained in Remark~\ref{Rem-2},
 the space-time resonance method in \cite{Pusateri} is not directly applicable either.
 To overcome these crucial difficulties, we will develop a new bootstrap framework.
 Next we  explain how to introduce the bootstrap assumptions.

 $\bullet$ At first,  for the highest order derivatives of $\phi$, as in \cite{Alinhac1}, the following slow time increasing bound is assumed
		\begin{equation}
		\label{eq:sketch-ordinary-energy-bootstrap}
		\sup_{0\le t\le T}
		(1+t)^{-\eta}
		\|\partial\phi(t)\|_{H^N_x}
		\le
		\varepsilon_1,
	\end{equation}
where $\eta\in(0,\mu/100]$ is any fixed small constant. It is pointed out that for the
quasilinear wave equations with quadratic null form nonlinearities, the highest order energy estimate
\eqref{eq:sketch-ordinary-energy-bootstrap} is closed through utilizing the basic Klainerman-Sobolev inequality \eqref{eq:1.14}
to show
	\begin{equation}\label{decay-1}
	\|\partial\phi(t)\|_{W^{2,\infty}(\R^3)}
	\lesssim
	\varepsilon(1+t)^{-1}.
	\end{equation}
	However, the usual 	Klainerman-Sobolev inequality  is not available for our problem \eqref{eq:Chaplygin-Cauchy}.
	Inspired by \cite{HouTaoYin2}, we introduce a modified good unknown
	\begin{equation}\label{eq:sketch-V-definition}
		V=\phi-\phi\partial_t\phi.
	\end{equation}
	In this case, one has  from the equation \eqref{potentia1} that
\begin{equation}\label{eq:sketch-V-equation}
	\Box V= R(\phi)=\mathcal R_3(\p^{\le3}\phi)+\mathcal R_4(\p^{\le3}\phi)+\mathcal R_5(\p^{\le3}\phi),
	\end{equation}
	where $\mathcal R_3$, $\mathcal R_4$ and $\mathcal R_5$ are the cubic,
quartic and quintic nonlinearities, respectively. Note that the right hand side in \eqref{eq:sketch-V-equation} does not contain
the troublesome quadratic nonlinearity, and $\mathcal R_3$ still admits the null structure which will play a crucial role in the related analysis
later. Nevertheless,
due to the different weights in our bootstrap assumptions from those in \cite{GaoLiYin2026} (see (6.1)-(6.4) in Section 6 of \cite{GaoLiYin2026})
and the different nonlinear structure in the right-hand side of \eqref{eq:sketch-V-equation} containing the loss of derivatives of $\phi$ from the cubic nonlinearity in (1.1) of \cite{GaoLiYin2026},
the method in \cite{GaoLiYin2026} can not be utilized to derive the sharp pointwise decay \eqref{decay-1}.
To overcome this difficulty, our idea is to study the following 3D linear wave equation
\begin{equation}\label{YHC-5}
\Box w=F, \quad (w,\partial_t w)(0,x)=(w_0,w_1)(x)
\end{equation}
so that the weighted terms $AB^\mu w$ and $AB^\mu\partial w$ can be estimated as (see Lemma \ref{lem:weighted-Linfty-L2} below)
\begin{equation}\label{YHC-6}
\begin{aligned}
\|A B^\mu w\|_{L^\infty([0,T]\times\mathbb R^3)}
&\lesssim\sum_{|a|\le 2}\big[\left\|\langle x\rangle^{1+\mu}
\nabla^a \p w(0)
\right\|_{L^2(\R^3)}+\left\|
A^{1+\mu}\nabla^a F\right\|_{L^1([0,T];L^2_x)}\big],
\end{aligned}
\end{equation}
\begin{equation}\label{YHC-7}
\begin{aligned}
\sum_{|a|\le 7}\|A B^\mu \nabla^a\partial w\|_{L^\infty([0,T]\times\mathbb R^3)}
&\lesssim\sum_{|a|\le 10}
\big[\left\|\langle x\rangle^{1+\mu}\nabla^a \p w(0)
\right\|_{L^2(\R^3)}+\left\|A^{1+\mu}\nabla^a F\right\|_{L^1([0,T];L^2_x)}\big],
\end{aligned}
\end{equation}
where $T>0$ is any fixed time. Note that in the proof of \eqref{YHC-6} and \eqref{YHC-7}, we have used
the strong Huygens principle and taking some delicate integral estimates on shifted spheres.
From \eqref{YHC-6} and \eqref{YHC-7}, we are motivated to make the following pointwise bootstrap assumptions:
	\begin{equation}\label{eq:sketch-bootstrap-pointwise-phi}
		\big\|AB^{\frac{\mu}{4}}\phi\big\|_{L^\infty([0,T]\times\mathbb R^3)}
		+\sum_{|a|\le 7}\big\|
		AB^{\frac{\mu}{4}}\nabla^a\partial\phi
		\big\|_{L^\infty([0,T]\times\mathbb R^3)}
		\le
		\varepsilon_1.
	\end{equation}
	Obviously, \eqref{eq:sketch-bootstrap-pointwise-phi} implies that \(\partial\phi\) fulfills \eqref{decay-1}
and further \(\partial\phi\) admits a mild additional
	\(B^{-\mu/4}\) space-time decay.

 $\bullet$ Secondly, in view of the equation \eqref{eq:sketch-V-equation}, \eqref{eq:sketch-bootstrap-pointwise-phi} alone is not sufficient
 	to close the estimates for $V$ and $\phi$ in other left bootstrap
 	assumptions. Thanks to the null structure
 	of $\mathcal R_3$, $\mathcal R_3$ can be controlled by
 	$\bar\partial\phi$ or $\bar\partial\partial\phi$, where $\bar\partial\in\big\{\partial_t+\partial_r,\,
		\frac1r\Omega_{12},\,\frac1r\Omega_{13},\,	\frac1r\Omega_{23}\big\}$ is the good
 	derivative (see \cite{Alinhac:book}).
Note that
	\begin{equation}
		\label{eq:sketch-good-derivative-Gamma}
		|\bar\partial f(t,x)|
		\lesssim
		\frac1{A(t,x)}
		\sum_{\Gamma}|\Gamma f(t,x)|
		+
		\frac{B(t,x)}{A(t,x)}
		|\partial f(t,x)|,
	\end{equation}
	where $\Gamma\in\{L_j,\Omega_{kl}: 1\le j\le3, 1\le k<l\le3\}$.
	Meanwhile, the coefficients of \(\Gamma\) naturally contain first-order spacetime weights.
	Since only the space weight $\langle x\rangle^{1+\mu}$ is imposed in \eqref{YHC-1},
this implies that $\Gamma$ acts on $\phi$ at most once for the $L_x^2$ energy estimate of $\Gamma\phi$.
This motivates us  to make such bootstrap assumptions,
	\begin{equation}
		\label{eq:sketch-bootstrap-energy-Gamma}
		\sup_{0\le t\le T}
		(1+t)^{-\eta}
		\sum_{\Gamma}
		\|\partial\Gamma\phi(t)\|_{H^{11}(\mathbb R^3)}
		\le
		\varepsilon_1
	\end{equation}
	and
	\begin{equation}
		\label{eq:sketch-bootstrap-pointwise-Gamma}
		\sum_{\Gamma}
		\big\|
		A^{\frac{\mu}{2}}
		\Gamma\phi
		\big\|_{L^\infty([0,T]\times\mathbb R^3)}+
		\sum_{\Gamma}
		\sum_{|a|\le 6}
		\big\|
		A^{\frac{\mu}{2}}
		\partial\nabla^a\Gamma\phi
		\big\|_{L^\infty([0,T]\times\mathbb R^3)}
		\le
		\varepsilon_1.
	\end{equation}
	The weight $A^{\frac{\mu}{2}}$ in \eqref{eq:sketch-bootstrap-pointwise-Gamma} is carefully
	selected to guarantee the closure of the bootstrap assumption
	\eqref{eq:sketch-bootstrap-pointwise-phi}.

${\bf \bullet}$ Thirdly, in the closing procedure of
	\eqref{eq:sketch-bootstrap-pointwise-Gamma}, we will come across the cubic terms in \eqref{eq:sketch-V-equation}.
As treated in \cite{GaoLiYin2026}, it is necessary to derive the weighted \(L^2_tL^\infty_x\) Strichartz
	estimates for $\Gamma\phi$ since the extra time decay rate can be supplied.
	This motivates us to introduce the following Strichartz-type bootstrap assumptions,
	\begin{equation}\label{eq:sketch-bootstrap-Strichartz-Gamma}
		\sum_{\Gamma}
		\big\|A^{\frac{\mu}{3}}
		\Gamma\phi\big\|_{L^2([0,T];L^\infty(\mathbb R^3))}
		+\sum_{\Gamma}\sum_{|a|\le6}\big\|A^{\frac{\mu}{3}}\partial\nabla^a\Gamma\phi
		\big\|_{L^2([0,T];L^\infty(\mathbb R^3))}
		\le\varepsilon_1.
	\end{equation}
	Note that the choice of the weight $A^{\frac{\mu}{3}}$ comes from the constraint requirement in the weighted
Strichartz-type estimate for 3D linear wave equation (see Lemma \ref{lem:weak-weighted-L2t-Linfty-L2} below).
	
Based on the above explanations on the introduction of bootstrap assumptions, we now illustrate how to
close the bootstrap assumptions.

\vskip 0.1 true cm

{\bf Step 1. Closing the highest order energy bootstrap assumption \eqref{eq:sketch-ordinary-energy-bootstrap}}

\vskip 0.1 true cm

Rewriting the equation \eqref{potentia1} in the symmetric quasilinear form as
	$$\Box\phi=\sum_{\alpha,\beta=0}^3
	\mathcal Q^{\alpha\beta}(\partial\phi)
	\partial_{\alpha\beta}^2\phi
	\quad\text{with $\mathcal Q^{\alpha\beta}
	=\mathcal Q^{\beta\alpha}$}.$$
	It follows from the standard energy method that
	\begin{equation}
		\label{eq:sketch-energy-integral}
		\|\partial\phi(t)\|_{H^N_x}^2
		\lesssim
		\|\partial\phi(0)\|_{H^N_x}^2
		+
		\int_0^t
		\big(
		\|\partial\phi(\tau)\|_{B^{1+\delta}_{\infty,1}}
		+
		\|\partial\phi(\tau)\|_{B^{1+\delta}_{\infty,1}}^2
		\big)
		\|\partial\phi(\tau)\|_{H^N_x}^2\,d\tau,
	\end{equation}
where the definition of Besov space $B^{1+\delta}_{\infty,1}$ sees Definition \ref {YHC-30} below.

	In addition, the pointwise bootstrap assumption \eqref{eq:sketch-bootstrap-pointwise-phi} implies
	\[
	\|\partial\phi(\tau)\|_{B^{1+\delta}_{\infty,1}}
	+
	\|\partial\phi(\tau)\|_{B^{1+\delta}_{\infty,1}}^2
	\lesssim
	\varepsilon_1(1+\tau)^{-1}.
	\]
This, together with \eqref{eq:sketch-energy-integral}, yields \eqref{eq:sketch-ordinary-energy-bootstrap}.

\vskip 0.1 true cm
	
{\bf Step 2. Closing the bootstrap assumption \eqref{eq:sketch-bootstrap-energy-Gamma} of $\Gamma\phi$}

\vskip 0.1 true cm
	Applying \(\Gamma\) to \eqref{eq:sketch-V-equation} gives
	\[
	\Box\Gamma V=\Gamma\mathcal R(\phi).
	\]
	By the standard energy estimate, one has
	\begin{equation}\label{YHC-8}
		\begin{aligned}
			\sum_{\Gamma}
			\|\p\Gamma V(t)\|_{H^{11}_x}
			\ls
			&
			\sum_{\Gamma}
			\|\p\Gamma V(0)\|_{H^{11}_x}
			+
			\int_0^t
			\sum_{\Gamma}
			\|\nabla^{\le 11}\Gamma\mathcal R(\phi)(\tau)\|_{L^2_x}
			\,d\tau
			\\
			\ls
			&
			\sum_{\Gamma}
			\|\p\Gamma V(0)\|_{H^{11}_x}
			+
			\sum_{k=3}^{5}\int_0^t
			\|A(\tau,\cdot)\p\nabla^{\le 11}\mathcal R_k(\phi)(\tau)\|_{L^2_x}
			\,d\tau
			\\
			\lesssim&
			\varepsilon+\varepsilon_1^3(1+t)^\eta,
		\end{aligned}
	\end{equation}
	where the term \(A(\tau,\cdot)\p\nabla^{\le 11}\mathcal R(\phi)\) in \eqref{YHC-8} is estimated by
	using \eqref{eq:sketch-ordinary-energy-bootstrap} and \eqref{eq:sketch-bootstrap-pointwise-phi}
	with a weighted Moser's estimate.
	Due to $\Gamma V=\Gamma\phi-\Gamma(\phi\partial_t\phi)$,
then \eqref{eq:sketch-bootstrap-energy-Gamma} can be closed by smallness
of the quadratic error $\Gamma(\phi\partial_t\phi)$ in the bootstrap assumptions.
	
\vskip 0.1 true cm
	
{\bf Step 3. Closing the pointwise bootstrap assumption \eqref{eq:sketch-bootstrap-pointwise-phi}}

\vskip 0.1 true cm
Applying the crucial weighted $L^{\infty}$-$L^2$ estimate (see Lemma~\ref{lem:weighted-Linfty-L2} below) to \(V\)
and to its spatial derivatives yields that
\begin{equation}\label{YHC-9}
	\begin{aligned}
		\|AB^{\frac{\mu}{4}}V\|_{L^\infty_{t,x}}
		+\sum_{|a|\le 7}\big\|AB^{\frac{\mu}{4}}
		\nabla^a\p V\big\|_{L^\infty_{t,x}}
		\ls \text{initial data}
		+\sum_{|a|\le 10}\big\|A^{1+\frac{\mu}{4}}\nabla^a\mathcal R(\phi)\big\|_{L^1_tL^2_x}.
	\end{aligned}
\end{equation}

To treat $\big\|A^{1+\frac{\mu}{4}}\nabla^a\mathcal R_3\big\|_{L^1_tL^2_x}$ in \eqref{YHC-9}, we decompose
$\Bbb R^3$ into $\mathcal E_t\cup \mathcal I_t$ with
\begin{equation*}
	\mathcal E_t=\big\{x\in\R^3:r\ge \f{1+t}{2}\big\}, \quad \mathcal I_t=\big\{x\in\R^3:r\le \f{1+t}{2}\big\}.
\end{equation*}
Note that $r^{-1}\ls A^{-1}$ in $\mathcal{E}_t$ and $B \sim A$ in $\mathcal{I}_t$.
Then $\big\|A^{1+\frac{\mu}{4}}\nabla^a\mathcal R_3\big\|_{L^1_tL^2_x}$ can be treated
by the null condition structure of $\mathcal R_3$, \eqref{eq:sketch-good-derivative-Gamma} and
the bootstrap assumption \eqref{eq:sketch-ordinary-energy-bootstrap}, \eqref{eq:sketch-bootstrap-pointwise-phi}, \eqref{eq:sketch-bootstrap-energy-Gamma} and \eqref{eq:sketch-bootstrap-pointwise-Gamma}.

In addition, $\big\|A^{1+\frac{\mu}{4}}\nabla^a\mathcal R_4\big\|_{L^1_tL^2_x}$
and $\big\|A^{1+\frac{\mu}{4}}\nabla^a\mathcal R_5\big\|_{L^1_tL^2_x}$ in \eqref{YHC-9}
can be handled directly. Therefore, by \(V=\phi-\phi\partial_t\phi\), we arrive at
\begin{equation}\label{YHC-10}
	\begin{aligned}
\big\|
AB^{\frac{\mu}{4}}\phi
\big\|_{L^\infty_{t,x}}
+\sum_{|a|\le 7}
\big\|AB^{\frac{\mu}{4}}\nabla^a\partial\phi
\big\|_{L^\infty_{t,x}}
\lesssim
\varepsilon+\varepsilon_1^2+\varepsilon_1^3.
\end{aligned}
\end{equation}

\vskip 0.1 true cm
	
{\bf Step 4.  Closing the pointwise bootstrap assumption \eqref{eq:sketch-bootstrap-pointwise-Gamma}}

\vskip 0.1 true cm

Applying the weighted \(L^\infty\)-\(L^2\) estimate (see Lemma \ref{lem:weak-weighted-Linfty-L2} below)
to \(\Gamma V\) yields that
		\begin{equation}\label{1.29}
		\begin{aligned}
			&\sum_{\Gamma}\big\|A^{\frac{\mu}{2}}\Gamma V
			\big\|_{L^\infty_{t,x}}
			+\sum_{\Gamma}\sum_{|a|\le 6}
			\big\|A^{\frac{\mu}{2}}\p\nabla^a\Gamma V\big\|_{L^\infty_{t,x}}
			\ls \text{initial data}
			+\sum_{\Gamma}
			\sum_{|a|\le 9}
			\big\|
			A^{\frac{\mu}{2}}
			\nabla^a\Gamma\mathcal R(\phi)
			\big\|_{L^1_tL^2_x}.
		\end{aligned}
	\end{equation}
	Note that $\mathcal R_3$ is composed by $\partial^{\le2}\phi\cdot
	\partial^{\le2}\phi\cdot
	\partial^{\le2}\partial\phi$.
	It follows from direct computation on $\nabla^a\Gamma\mathcal R_3$, \eqref{eq:sketch-ordinary-energy-bootstrap}, \eqref{eq:sketch-bootstrap-pointwise-phi}, \eqref{eq:sketch-bootstrap-energy-Gamma}
and \eqref{eq:sketch-bootstrap-Strichartz-Gamma} that

\begin{equation}\label{YHC-11}
	\begin{aligned}
\sum_{\Gamma}
			\sum_{|a|\le 9}
			\big\|
			A^{\frac{\mu}{2}}
			\nabla^a\Gamma\mathcal R(\phi)
			\big\|_{L^1_tL^2_x}	
			\lesssim
	\varepsilon_1^3.
	\end{aligned}
	\end{equation}
Then it holds that
\begin{equation}\label{YHC-12}
	\begin{aligned}
	\sum_{\Gamma}
	\big\|
	A^{\frac{\mu}{2}}\Gamma\phi
	\big\|_{L^\infty_{t,x}}
	+\sum_{\Gamma}
	\sum_{|a|\le 6}
	\big\|
	A^{\frac{\mu}{2}}\partial\nabla^a\Gamma\phi
	\big\|_{L^\infty_{t,x}}
	\ls
	\varepsilon+\varepsilon_1^2+\varepsilon_1^3.
	\end{aligned}
	\end{equation}

\vskip 0.1 true cm

{\bf Step 5. Closing the Strichartz-type bootstrap assumption \eqref{eq:sketch-bootstrap-Strichartz-Gamma}}

 \vskip 0.1 true cm

 It follows from a weighted \(L^2_tL^\infty_x\)-\(L^2\) estimate for $V$ (see Lemma \ref{lem:weak-weighted-L2t-Linfty-L2} below) and analogous argument for \eqref{1.29}
that
\begin{equation}\label{YHC-13}
\begin{aligned}
\sum_{\Gamma}
\big\|A^{\frac{\mu}{3}}\Gamma\phi
\big\|_{L^2_tL^\infty_x}
+\sum_{\Gamma}\sum_{|a|\le 6}
\big\|	A^{\frac{\mu}{3}}\partial\nabla^a\Gamma\phi
\big\|_{L^2_tL^\infty_x}
\lesssim
\varepsilon+\varepsilon_1^2+\varepsilon_1^3.
\end{aligned}
\end{equation}

	Therefore, all bootstrap assumptions are closed by taking
	\(\varepsilon_1=C_0\varepsilon\) with \(C_0\) being suitably large.  By the standard continuity argument,
the global existence in  Theorem~\ref{thm:main-global} is derived. Meanwhile, \eqref{YHCCC-1}-\eqref{YHCCC-5} follow from
the bootstrap bounds. Based on this,  Theorems~\ref{thm:main-global-1}-\ref{thm:main-global}
 can be obtained.

	This paper is organized as follows. In
	Section~\ref{sec:Preliminaries}, we will list the notations for Littlewood-Paley decomposition and prove some basic
	estimates for 3D linear wave equations. These estimates include the weighted
	\(L^\infty\)-\(L^2\) estimates, the \(AB^\mu\)-weighted
	\(L^\infty\)-\(L^2\) estimates, and the weighted Strichartz estimates.
In addition, the properties of both the null condition and the good derivatives are illustrated.
	In Section \ref{sec:CVP-bootstrap-assumptions}, the bootstrap assumptions of $\phi$ are given. Meanwhile,
we derive some estimates for the higher time derivatives of $\phi$.
This will allow us to reduce the
	pointwise arguments on $\phi$ to the case for the spatial derivatives of $\phi$.
	In Section \ref{sec:Chaplygin-energy-estimates}, we complete the closure of the bootstrap
assumptions for the ordinary higher order derivatives of  $\phi$
by applying the standard energy method to \eqref{eq:Chaplygin-Cauchy}.
	In Sections \ref{sec:equation-for-V}-\ref{sec:weighted-Strichartz-estimates}, all remaining bootstrap
arguments are shown.
	The proofs of Theorems \ref{thm:main-global}-\ref{thm:main-global-1}
	are finished  by the continuity argument in Section \ref{sec:Proof-Theorem}.
	In the appendix, Section \ref{section a} includes the definition and properties of $A_2$ weight and some related weighted inequalities;
Section \ref{sec:b} collects several technical lemmas used in this paper.

	\subsection{Notations}\label{subsec:notations}
	
	\begin{description}
		\item [$\blacktriangleright $]
		$\partial=(\partial_t,\nabla)=(\partial_t,\p_x)= (\partial_t,\partial_{x_1},\partial_{x_2},\partial_{x_3})=(\partial_0,\partial_1,\partial_2,\partial_3)$.
		
		\item [$\blacktriangleright $]
		$\nabla^a=\partial_x^a=\partial_1^{a_1}\partial_2^{a_2}\partial_3^{a_3}$ for
		$a=(a_1,a_2,a_3)\in\mathbb N_0^3$, and $|a|=a_1+a_2+a_3$.
		
		\item [$\blacktriangleright $]
		For a space-time multi-index $\alpha=(\alpha_0,\alpha_1,\alpha_2,\alpha_3)\in\mathbb N_0^4$,
		$\partial^\alpha=\partial_t^{\alpha_0}\partial_1^{\alpha_1}
		\partial_2^{\alpha_2}\partial_3^{\alpha_3}$.
		
		\item [$\blacktriangleright $]
		$\p \phi(0) = \ve(\phi_1,\nabla \phi_0).$
		
		\item [$\blacktriangleright $]
		$\Box=\partial_t^2-\Delta$, $\Delta=\ds\sum_{k=1}^3\partial_k^2$.
		
		\item [$\blacktriangleright $]
		$r=|x|=\sqrt{x_1^2+x_2^2+x_3^2}$, $\omega=x/r$ for $r>0$.
		
		\item [$\blacktriangleright $]
		$\langle x\rangle=(1+|x|^2)^{1/2}$.
		
		\item [$\blacktriangleright $]
		$A=A(t,x)=1+t+|x|$ and
		$B=B(t,x)=1+|t-|x||$.
		
		\item [$\blacktriangleright $]
		$Q_0(f,g)=\partial_t f\,\partial_t g-\nabla f\cdot\nabla g$.

		\item [$\blacktriangleright $]
		$\Omega_{kl}=x_k\partial_l-x_l\partial_k$ for $1\le k<l\le3$.
		
		\item [$\blacktriangleright $]
		$L_j=t\partial_j+x_j\partial_t$ for $1\le j\le3$.
		
		\item [$\blacktriangleright $]
		$\Gamma\in\{L_j,\Omega_{kl}: 1\le j\le3, 1\le k<l\le3\}$.
		
		\item [$\blacktriangleright $]
		$\Gamma^\ell$ means the identity operator when $\ell=0$, and one of
		$L_k,\Omega_{kj}$ when $\ell=1$.
		
		\item [$\blacktriangleright $]
		$
			\bar\partial
			\in
			\big\{
			\partial_t+\partial_r,\,
			\frac1r\Omega_{12},\,
			\frac1r\Omega_{13},\,
			\frac1r\Omega_{23}
			\big\}.
		$
		
		\item [$\blacktriangleright $]For $f = (f^1,\cdots,f^m)$, $\p f=   (\p f^1,\cdots,\p f^m)$, $\ds|\p^{\le j}f|=\big(\sum_{0\le|\alpha|
			\le j}|\p^\alpha f|^2\big)^\frac12$.

		\item [$\blacktriangleright $]
		$
			|\bar\partial f|
			=
			|\partial_t f+\partial_r f|
			+\ds\sum_{1\le k<j\le3}
			\big|\frac1r\Omega_{kj}f\big|.
		$

		\item [$\blacktriangleright $]
		For $1\le p<\infty$,
		\begin{equation*}
			\|F\|_{L^p_tL^q_x}
			=
			\big(
			\int_0^T\|F(t,\cdot)\|_{L^q_x}^p\,dt
			\big)^{1/p}.
		\end{equation*}
		
		\item [$\blacktriangleright $]
		$\|F\|_{L^\infty_{t,x}}=\|F\|_{L^\infty([0,T]\times\mathbb R^3)}$.
		
		\item [$\blacktriangleright $]
		If $f=(f^1,\cdots,f^m)$ and $\|\cdot\|$ is a norm, then
		$\|f\|=\ds\sum_{k=1}^m\|f^k\|$.
		
		\item [$\blacktriangleright $] $\|f\|_{L^p(\omega)}=\|f\|_{L^p(\R^d,\omega(x)dx)} = (\int_{\R^d} |f(x)|^p\omega(x)dx)^{1/p}$ for $p \in [1,\infty)$ and $\omega(x)\ge 0$.
		
		\item [$\blacktriangleright $]$\Pk$ (resp. $\pk$) is the nonhomogeneous (resp. homogeneous) Littlewood-Paley projection onto frequency $2^k$ (see the details in \eqref{YHCCC-001}).
		\item [$\blacktriangleright $] \begin{equation*}
			\begin{split}
				\cY_k&=\cY_k^1\cup\cY_k^2,\\
				\cY_k^1&=\{(k_1,k_2,k_3)\in\Z^3: |\max_{l=1,2,3}\{k_l\}-k|\le4,k_1,k_2,k_3\ge-1\},\\
				\cY_k^2&=\{(k_1,k_2,k_3)\in\Z^3: \max_{l=1,2,3}\{k_l\}\ge k+4,
				\max_{l=1,2,3}\{k_l\}-{\rm med}_{l=1,2,3}\{k_l\}\le4,k_1,k_2,k_3\ge-1\}.
			\end{split}
		\end{equation*}
		
		\item [$\blacktriangleright $]$R=(R_1,R_2,R_3)=\ds\f{\nabla}{|\nabla|}$ is the Riesz transformation.
		
		\item [$\blacktriangleright $]
		$
			\mathcal{S}_{\rho}
			=
			\ds\frac{\sin(\rho|\nabla|)}{|\nabla|},
			\qquad
			\mathcal{T}_{\rho}
			=
			\cos(\rho|\nabla|).
		$

		\item [$\blacktriangleright $]
			For non-negative quantities $f$ and $g$, the notation $f\ls g$ means
			$f\le Cg$, while $f\sim g$ means
			$C_1g\le f\le C_2g$, where $C,C_1,C_2>0$ are generic constants independent of $\ve$ but depending on $\mu$.

	\end{description}

	\section{Preliminaries}\label{sec:Preliminaries}
	
	\subsection{Notations for dyadic decomposition}\label{subsec:Notations-for-dyadic-decomposition}
	For function $f$ on $\mathbb{R}^3$, define its Fourier transformation as
	\begin{align*}
		\hat{f}(\xi)=\mathscr{F}f(\xi)=\int_{\mathbb{R}^3} e^{-i x\cdot\xi}f(x)dx\quad\text{with $i=\sqrt{-1}$ and $\xi\in\Bbb R^3$}.
	\end{align*}
	Choose a smooth cutoff function $\psi\colon\mathbb{R}\rightarrow[0,1]$, which equals 1 on $[-5/4,5/4]$ and vanishes outside
	$[-8/5,8/5]$. Set
	\begin{align*}
		&\dot{\psi}_{k}(x)=\psi(|x|/2^{k})-\psi(|x|/2^{k-1}),\quad k\in\mathbb{Z},\\
		&\psi_{k}(x)=\dot{\psi}_k(x),\quad k\in\mathbb{N}_0,\\
		&\psi_{-1}(x)=1-\sum_{k\geq 0}\psi_{k}(x)=\psi(2|x|),\\
		&\dot{\psi}_{I}=\sum_{k\in I\cap\mathbb{Z}}\dot{\psi}_{k}, \quad {\psi}_{I}=\sum_{k\in I\cap\mathbb{Z}\cap[-1,+\infty)}{\psi}_{k},
	\end{align*}
	where $I$ is any subset of $\mathbb{R}$. Note that $\psi_0=1$ on $\{x\in\R^3:|x|\in[4/5,5/4]\}$ and $\supp\psi_0\subseteq\{x\in\R^3:|x|\in [5/8,8/5]\}$. Let $\pk$ be the homogeneous
	Littlewood-Paley projection onto frequency $2^k$:
	\begin{equation*}
		\mathscr{F}(\pk f)(\xi)=\dot{\psi}_{k}(\xi)\mathscr{F}f(\xi),\quad k\in\mathbb{Z}.
	\end{equation*}
	In addition, for any subset $I$, define
	\begin{align*}
		\dot{P_I}f&=\sum_{k\in I\cap\mathbb{Z}}\pk f,\quad
		\dot{P}_{[[k]]}f=\sum_{l\in[k-1,k+1]\cap\mathbb{Z}}\dot{P}_l f.
	\end{align*}
	Let $\Pk$ be the nonhomogeneous
	Littlewood-Paley projection onto frequency $2^k$:
	\begin{equation}\label{YHCCC-001}
		\begin{split}
			\mathscr{F}(\Pk f)(\xi)={\psi}_{k}(\xi)\mathscr{F}f(\xi),\quad k\in\mathbb{Z}\cap[-1,+\infty).
		\end{split}
	\end{equation}
	For any subset $I$, set
	\begin{align*}
		&P_I f=\sum_{k\in I\cap\mathbb{Z}\cap\lbrack-1,+\infty)}\Pk f,\quad
		P_{[[k]]}f= \dot{P}_{[[k]]}f\quad\text{for $k\geq 0$},\quad
		P_{[[-1]]}f= (P_{-1}+P_{0})f.
	\end{align*}
	It is easy to check that $P_{[[k]]}P_k=P_k$ and $\dot{P}_{[[k]]}\dot{P}_k=\dot{P}_k$.
	Let $\psi_{[[k]]}$ and $\dot{\psi}_{[[k]]}$ denote the symbols of $P_{[[k]]}$ and $\dot{P}_{[[k]]}$, respectively. Note that $\psi_{[[k]]} = \dot{\psi}_{[[k]]}$ for $k\in\N_0$.
	
	\begin{definition}
		Denote by $\mathcal{S}_h'(\mathbb{R}^3)$ the space of tempered distributions $u\in\mathcal{S}'(\R^3)$ such that
		\[
		\lim_{\lambda \to \infty} \left\| \theta(\lambda D) u \right\|_{L^\infty} = 0 \quad \text{for any } \theta\in  C_0^{\infty}(\mathbb{R}^3).
		\]
	\end{definition}
	For $f\in\mathcal{S}'(\R^3)$, $f=\ds\sum_{k\ge-1}P_kf$ in $\mathcal{S}'(\R^3)$ holds;
	for $f\in\mathcal{S}'_h(\R^3)$, $f=\ds\sum_{k\in\Z}\pk f$ in $\mathcal{S}'(\R^3)$ holds.
	\begin{lemma}
		Let $u \in \mathcal{S}'(\R^3)$. If $P_{-1}u\in L^p(\R^3)$ for $p\in[1,\infty)$ or $\hat{u}\in L^1_{loc}(\R^3)$, then $u \in \mathcal{S}'_h(\R^3)$.
	\end{lemma}
	\begin{proof}
		Since its proof can be found in \cite[Chapter 1]{Chemin}, we omit the details here.
	\end{proof}
	\begin{definition}\label{YHC-30}
		Let \( s \in \mathbb{R} \) and \( 1 \leq p, r \leq \infty \). The nonhomogeneous Besov space \( B^s_{p,r}(\R^3) \) consists of all  \( u\in \mathcal{S}' \) with $\|u\|_{B^s_{p,r}(\R^3)}= \big\| \big( 2^{ks} \|P_k u\|_{L^p(\R^3)} \big)_{k \in \mathbb{Z}\cap[- 1,+\infty)} \big\|_{\ell^r(\mathbb{Z}\cap[- 1,+\infty))} < \infty$.  The homogeneous Besov space \( \dot{B}^s_{p,r}(\R^3)\) consists of \( u\in \mathcal{S}'_h \) with $\|u\|_{\dot{B}^s_{p,r}(\R^3)}= \big\|\big( 2^{ks} \|\pk  u\|_{L^p(\R^3)}\big)_{k \in \mathbb{Z}}\big\|_{\ell^r(\mathbb{Z})}<\infty$.
	\end{definition}
	Note that $H^s(\R^3)=B^s_{2,2}(\R^3)$. For $s>0,\ p\in[1,\infty),\ r\in[1,\infty]$, one has $B^s_{p,r}(\R^3) = \dot{B}^s_{p,r}(\R^3)\cap L^p(\R^3)$.
	In addition, we set
	\begin{equation*}
		\begin{split}
			\cX_k&=\cX_k^1\cup\cX_k^2,\\
			\cX_k^1&=\{(k_1,k_2)\in\Z^2: |\max_{l=1,2}\{k_l\}-k|\le4,k_1,k_2\ge-1\},\\
			\cX_k^2&=\{(k_1,k_2)\in\Z^2: \max_{l=1,2}\{k_l\}\ge k+4,
			\max_{l=1,2}\{k_l\}-\min_{l=1,2}\{k_l\}\le4,k_1,k_2\ge-1\};\\
			\cY_k&=\cY_k^1\cup\cY_k^2,\\
			\cY_k^1&=\{(k_1,k_2,k_3)\in\Z^3: |\max_{l=1,2,3}\{k_l\}-k|\le4, k_1,k_2,k_3\ge-1\},\\
			\cY_k^2&=\{(k_1,k_2,k_3)\in\Z^3: \max_{l=1,2,3}\{k_l\}\ge k+4,
			\max_{i=1,2,3}\{k_i\}-{\rm med}_{l=1,2,3}\{k_l\}\le4, k_1,k_2,k_3\ge-1\}.
		\end{split}
	\end{equation*}
	As in \cite[page 799]{IP13}, if $P_k\big(\prod_{\iota=1}^2 P_{k_\iota}f_\iota\big)\neq0$,
	then $(k_1,k_2)\in\cX_k$; if $P_k\big(\prod_{\iota=1}^3 P_{k_\iota}f_\iota\big)\neq0$, then $(k_1,k_2,k_3)\in\cY_k$.
	Obviously, one has $2^k\ls 2^{\max\{k_1,k_2,k_3\}}$ for $(k_1,k_2,k_3)\in\cY_k$.

	\subsection{Weighted $L^\infty$-$L^2$ estimates with the weight $A^{\mu}$}\label{subsec:weighted-linear-Linfty-L2-estimate-1}
			
In this subsection, motivated by Section 3 and especially Corollary~3.8 of \cite{GaoLiYin2026}, we establish some weighted \(L^\infty\)-\(L^2\) estimate for the 3D linear wave equation.  The precise statement on the  weighted $L^\infty$-$L^2$ estimates
will be given in Lemma~\ref{lem:weak-weighted-Linfty-L2} below.

At first, we give an integral estimate on shifted spheres.
			\begin{lemma}\label{lem:weak-geom-sphere}
				For \(0<\mu<1\), \(0\le s\le t\), \(\rho=t-s\ge0\) and
				\(x\in\mathbb R^3\), it holds that
				\begin{equation}\label{eq:weak-geom-sphere-surface}
				\int_{|y-x|=\rho}
				(1+s+|y|)^{-2\mu}\,dS_y
				\ls
				\rho^2 A^{-2\mu}(t,x).
				\end{equation}
			\end{lemma}
			
			\begin{proof}
				Denote $r = |x|$. For any radial function \(h=h(|y|)\), one has
that for  \(r>0\) and \(\rho>0\),
				\begin{equation}\label{eq:sphere-one-dimensional-reduction}
					\begin{aligned}
						\int_{|y-x|=\rho} h(|y|)\,dS_y
						&=
						\rho^2\int_{\mathbb S^2} h(|x+\rho\omega|)\,d\omega  \\
						&=
						2\pi\rho^2
						\int_{-1}^1
						h\!\big((r^2+\rho^2+2r\rho\eta)^{1/2}\big)\,d\eta  \\
						&=
						2\pi\rho^2
						\int_{|r-\rho|}^{r+\rho}
						h(\lambda)\frac{\lambda}{r\rho}\,d\lambda  \\
						&=
						\frac{2\pi\rho}{r}
						\int_{|r-\rho|}^{r+\rho} h(\lambda)\lambda\,d\lambda .
					\end{aligned}
				\end{equation}
				
				We now prove \eqref{eq:weak-geom-sphere-surface}. When \(\rho=0\), \eqref{eq:weak-geom-sphere-surface}
obviously holds. When \(r=0\), one has
				\[
				\int_{|y|=\rho}
				(1+s+|y|)^{-2\mu}\,dS_y
				=
				4\pi\rho^2(1+s+\rho)^{-2\mu}.
				\]
				Due to $A(t,0)=1+t=1+s+\rho$,
				we have
				\[
				\int_{|y|=\rho}
				(1+s+|y|)^{-2\mu}\,dS_y
				\lesssim
				\rho^2 A^{-2\mu}(t,0).
				\]
				
Next,  \(r>0\) and \(\rho>0\) are assumed. Set $h(\lambda)=(1+s+\lambda)^{-2\mu}$ in \eqref{eq:sphere-one-dimensional-reduction}.
				Then
				\begin{equation}
					\label{eq:weak-I-definition}
					\begin{aligned}
						I=
						\int_{|y-x|=\rho}
						(1+s+|y|)^{-2\mu}\,dS_y
						=
						\frac{2\pi\rho}{r}
						\int_{|r-\rho|}^{r+\rho}
						\lambda(1+s+\lambda)^{-2\mu}\,d\lambda .
					\end{aligned}
				\end{equation}
				Let $m_0=\min\{r,\rho\}$ and $M_0=\max\{r,\rho\}$.
				Note that $A(t,x)=1+s+\rho+r$.

				\vskip 0.1 true cm
				\noindent
				\textbf{Case 1. \(m_0\le A(t,x)/8\)}
\vskip 0.1 true cm
				For $\lambda\in[|r-\rho|,r+\rho]=[M_0-m_0,M_0+m_0]$, one has
				\[
				1+s+\lambda
				\ge
				1+s+M_0-m_0
				=
				A(t,x)-2m_0
				\ge
				\frac34A(t,x).
				\]
				Therefore, by \(M_0m_0=r\rho\), we arrive at
				\begin{equation}
					\label{eq:weak-geom-small-min-surface}
					\begin{aligned}
						I
						\lesssim
						\frac{\rho}{r}
						A^{-2\mu}(t,x)
						\int_{M_0-m_0}^{M_0+m_0}
						\lambda\,d\lambda
						=
						\frac{\rho}{r}
						A^{-2\mu}(t,x)
						\cdot 2M_0m_0
						=
						2\rho^2A^{-2\mu}(t,x).
					\end{aligned}
				\end{equation}
				
				\smallskip
				\noindent
				\textbf{Case 2. \(m_0>A(t,x)/8\)}
\vskip 0.1 true cm

			In this case, one has
				$
				r\rho=M_0m_0\ge m_0^2\ge \frac{A^2(t,x)}{64}.
				$
				Moreover, \(r+\rho\le A(t,x)\) holds.  Due to \(0<\mu<1\) and \(A(t,x)\ge1\), we have
				\[
				\int_0^{A(t,x)}
				\lambda(1+s+\lambda)^{-2\mu}\,d\lambda
				\le
				\int_0^{A(t,x)}
				\lambda(1+\lambda)^{-2\mu}\,d\lambda
				\lesssim
				1+\int_1^{A(t,x)} \lambda^{1-2\mu}\,d\lambda
				\ls
				A^{2-2\mu}(t,x).
				\]
				Thus, it follows from \eqref{eq:weak-I-definition} that
				\begin{equation}
					\label{eq:weak-geom-large-min-surface}
					\begin{aligned}
						I
						\lesssim
						\frac{\rho}{r}
						\int_0^{A(t,x)}
						\lambda(1+s+\lambda)^{-2\mu}\,d\lambda
						\ls
						\frac{\rho}{r}A^{2-2\mu}(t,x).
					\end{aligned}
				\end{equation}
				This, together with  \(r\rho\gtrsim A^2(t,x)\), yields
				\begin{equation}\label{eq:weak-geom-large-min-surface-1}
\begin{aligned}
				I\ls\frac{\rho}{r}A^{2-2\mu}(t,x)
				=
				\rho A^{-2\mu}(t,x)\frac{A^2(t,x)}{r}
				\ls
				\rho^2A^{-2\mu}(t,x).
				\end{aligned}
				\end{equation}
				Combining the two cases leads to \eqref{eq:weak-geom-sphere-surface}.
			\end{proof}

			Next, we establish some weighted $L^\infty$-$L^2$ estimates with the weight $A^{\mu}$ for the 3D wave propagators.

			\begin{lemma}\label{lem:weak-weighted-wave-propagators}
				Let \(0<\mu<1\), \(T>0\) and \(s\in[0,T]\).
				For \(\rho\ge0\), define
				\begin{equation*}
					\mathcal S_\rho
					=
					\frac{\sin(\rho|\nabla|)}{|\nabla|},
					\qquad
					\mathcal T_\rho
					=
					\cos(\rho|\nabla|).
				\end{equation*}
				Then
				\begin{equation}
					\label{eq:weak-S-basic}
					\left\|
					A^\mu(t,x)\mathcal S_{t-s}f
					\right\|_{L^\infty([s,T]\times\R^3)}
					\ls
					\sum_{|a|\le2}
					\left\|
					A^\mu(s,\cdot)\nabla^a f
					\right\|_{L^2_x},
				\end{equation}
				\begin{equation}
					\label{eq:weak-T-basic}
					\left\|
					A^\mu(t,x)\mathcal T_{t-s}f
					\right\|_{L^\infty([s,T]\times\R^3)}
					\ls
					\sum_{|a|\le3}
					\left\|
					A^\mu(s,\cdot)\nabla^a f
					\right\|_{L^2_x}.
				\end{equation}
			\end{lemma}
			
			\begin{proof}
				Set $\widetilde A(t,x)=1+t+\langle x\rangle$.
				Obviously, $A(t,x)\sim \widetilde A(t,x)$ holds. One can use $\widetilde A(t,x)$
instead of $A(t,x)$ in \eqref{eq:weak-S-basic} and \eqref{eq:weak-T-basic}.
However, for simplicity and without confusion, the notation \(A\) is still applied.
				
				Note that
				\begin{equation}\label{eq:weak-S-sphere}
					\mathcal{S}_{\rho} f(x)
					=
					\frac{\sin(\rho|\nabla|)}{|\nabla|}f(x)
					=
					\frac{\rho}{4\pi}
					\int_{\mathbb S^2}f(x+\rho\omega)\,d\omega,
				\end{equation}
				\begin{equation}\label{eq:weak-T-sphere}
					\mathcal{T}_\rho f(x)
					=
					\cos(\rho|\nabla|)f(x)
					=
					\frac{1}{4\pi}
					\int_{\mathbb S^2}f(x+\rho\omega)\,d\omega
					+
					\frac{\rho}{4\pi}
					\int_{\mathbb S^2}
					\omega\cdot\nabla f(x+\rho\omega)\,d\omega.
				\end{equation}
				
				At first, we prove \eqref{eq:weak-S-basic}. Fix \(t\in[s,T]\) and set
				$
				\rho=t-s.
				$
				By \eqref{eq:weak-S-sphere} and Cauchy's inequality, one has
				\begin{equation}
					\label{eq:weak-S-CS}
					\begin{aligned}
						|\mathcal S_\rho f(x)|
						\lesssim
						\rho
						\big(
						\int_{\mathbb S^2}
						A^{-2\mu}(s,x+\rho\omega)\,d\omega
						\big)^{1/2}
						\big(
						\int_{\mathbb S^2}
						A^{2\mu}(s,x+\rho\omega)
						|f(x+\rho\omega)|^2\,d\omega
						\big)^{1/2}.
					\end{aligned}
				\end{equation}
				By Lemma~\ref{lem:weak-geom-sphere}, it holds
				\begin{equation}
					\label{eq:weak-geom-used}
					A^\mu(t,x)
					\left(
					\int_{\mathbb S^2}
					A^{-2\mu}(s,x+\rho\omega)\,d\omega
					\right)^{1/2}
					\ls 1.
				\end{equation}
				Indeed, for \(\rho>0\), we have that by Lemma \ref{lem:weak-geom-sphere},
				\[
				\int_{\mathbb S^2}
				A^{-2\mu}(s,x+\rho\omega)\,d\omega
				=
				\rho^{-2}
				\int_{|y-x|=\rho}
				(1+s+|y|)^{-2\mu}\,dS_y\ls
				A^{-2\mu}(t,x),
				\]
				and for \(\rho=0\),  \eqref{eq:weak-geom-used} is obvious.
				
				In addition, it follows from Lemma \ref{lem:sphere-trace-H32} that
				\begin{equation}
					\label{eq:weak-weighted-sphere-trace}
					\begin{aligned}
						\rho
						\big(
						\int_{\mathbb S^2}
						A^{2\mu}(s,x+\rho\omega)
						|f(x+\rho\omega)|^2\,d\omega
						\big)^{1/2}
						\lesssim
						\|A^\mu(s,\cdot)f\|_{H^2_x}.
					\end{aligned}
				\end{equation}
				Due to
				\[
				|\nabla^\gamma A^\mu(s,x)|
				\lesssim
				A^\mu(s,x)
				\quad\text{for $|\gamma|\le2$},
				\]
				one has
				\begin{equation}
					\label{eq:weak-weighted-H2-product}
					\|A^\mu(s,\cdot)f\|_{H^2_x}
					\ls
					\sum_{|a|\le2}
					\|A^\mu(s,\cdot)\nabla^a f\|_{L^2_x}.
				\end{equation}
				Combining \eqref{eq:weak-S-CS}-\eqref{eq:weak-weighted-H2-product} yields
				\eqref{eq:weak-S-basic}.
				
				For \eqref{eq:weak-T-basic}, the first term on the right hand side of
				\eqref{eq:weak-T-sphere} can be directly treated by Cauchy's inequality and the trace estimate
				in Lemma~\ref{lem:sphere-trace-H32} as above. The second term may be estimated as
				for \(\mathcal S_\rho(\nabla f)\) since the angular factor \(\omega\) is bounded.
				Consequently,
				\[
				A^\mu(t,x)|\mathcal T_\rho f(x)|
				\ls
				\sum_{|a|\le3}
				\|A^\mu(s,\cdot)\nabla^a f\|_{L^2_x},
				\]
				which gives \eqref{eq:weak-T-basic}.
				
			\end{proof}

			We now derive the weighted $L^\infty$-$L^2$ estimates with the weight $A^\mu$ for 3D linear wave equation.

			\begin{lemma}[{\bf Weighted \(L^\infty\)-\(L^2\) estimates with the weight $A^\mu$}]
				\label{lem:weak-weighted-Linfty-L2}
				Assume \(0<\mu<1\) and $m\in\N_0$. Let \(w\) solve
				\begin{equation*}
					\Box w=F,\quad
					(w,\partial_t w)(0,x)=(w_0,w_1)(x),
				\end{equation*}
				where $(t,x)\in [0,T]\times\mathbb R^3$ and $w_0\in L^2(\R^3)$.
				Then it holds that
				\begin{equation}
					\label{eq:weak-weighted-Linfty-L2-w}
					\begin{aligned}
						\|A^\mu w\|_{L^\infty([0,T]\times\mathbb R^3)}
						\ls
						\sum_{|a|\le2}
						\big[
						\left\|
						\langle x\rangle^\mu\nabla^a \partial w(0)
						\right\|_{L^2_x}
						+
						\left\|
						A^\mu\nabla^aF
						\right\|_{L^1([0,T];L^2_x)}
						\big],
					\end{aligned}
				\end{equation}
				\begin{equation}
					\label{eq:weak-weighted-Linfty-L2-dw}
					\begin{aligned}
						\sum_{|a|\le m}
						\|A^\mu\nabla^a\partial w\|_{L^\infty([0,T]\times\mathbb R^3)}
						&\ls
						\sum_{|a|\le m+3}
						\big[
						\left\|
						\langle x\rangle^\mu
						\nabla^a\partial w(0)
						\right\|_{L^2_x}
						+
						\left\|
						A^\mu\nabla^aF
						\right\|_{L^1([0,T];L^2_x)}
						\big].
					\end{aligned}
				\end{equation}
			\end{lemma}
			
			\begin{proof}
				Note that
				\begin{equation}\label{YHC-36}
					w=w_{\mathrm{hom}}+w_F,
				\end{equation}
				where $w_F(t,x)=\int_0^t
					\mathcal S_{t-s}F(s,\cdot)(x)\,ds$
				and $w_{\mathrm{hom}}(t)
					=\mathcal T_t w_0+\mathcal S_t w_1$.

				We first prove \eqref{eq:weak-weighted-Linfty-L2-w}. It follows from \eqref{eq:weak-S-basic}
                and Minkowski's inequality that
				\begin{equation}
					\label{eq:weak-Duhamel-w}
					\begin{aligned}
						\|A^\mu w_F\|_{L^\infty_{t,x}}
						&\le
						\int_0^T
						\left\|
						A^\mu(t,x)\mathcal S_{t-s}F(s,\cdot)
						\right\|_{L^\infty([s,T]\times\mathbb R^3)}
						\,ds
						\\
						&\ls
						\sum_{|a|\le2}
						\int_0^T
						\|A^\mu(s,\cdot)\nabla^aF(s,\cdot)\|_{L^2_x}
						\,ds
						\\
						&=
						\sum_{|a|\le2}
						\|A^\mu\nabla^aF\|_{L^1_tL^2_x}.
					\end{aligned}
				\end{equation}
				Applying \eqref{eq:weak-S-basic} with \(s=0\) gives
				\begin{equation}
					\label{eq:weak-St-w1-w}
					\|A^\mu \mathcal S_t w_1\|_{L^\infty_{t,x}}
					\ls
					\sum_{|a|\le2}
					\|\langle x\rangle^\mu\nabla^a w_1\|_{L^2_x}.
				\end{equation}
				
				It remains to estimate \(\mathcal T_t w_0\). At first, we derive a weighted pointwise
				bound for $w_0$. For $f\in S'_h(\R^3)$, it follows from Bernstein's inequality that
				\begin{equation}\label{ineq:homogeneous-Sobolev-inequality}
					\|f\|_{L^\infty(\R^3)}
					\ls
					\sum_{k\in\Z}2^{\f32 k}\|\pk f\|_{L^2(\R^3)}
					\ls
					\sum_{k\in\Z}2^{\f12 k}\|\pk\nabla f\|_{L^2(\R^3)}
					\ls
					\|\nabla f\|_{H^{1}(\R^3)}.
				\end{equation}
				In addition, due to $\w{x}^{\mu} w_0 \in S'_h(\R^3)$ (see Lemma \ref{lem:weighted-data-in-Sh}),
                one has
				\begin{equation}
					\label{eq:weak-weighted-Linfty-initial-displacement}
					\begin{aligned}
						\|\langle x\rangle^\mu w_0\|_{L^\infty_x}
						&\lesssim
						\left\|
						\nabla\left(\langle x\rangle^\mu w_0\right)
						\right\|_{H^1_x}
						\ls
						\sum_{1\le |a|\le2}
						\|\langle x\rangle^\mu\nabla^a w_0\|_{L^2_x}
						+
						\|\langle x\rangle^{\mu-1}w_0\|_{L^2_x}.
					\end{aligned}
				\end{equation}
				Because of $\mu-1>-\frac32$,
				applying Lemma~\ref{lem:weighted-Hardy-initial-displacement}
				with parameter \(\mu-1\) yields
				\begin{equation}
					\label{eq:weak-Hardy-w0}
					\|\langle x\rangle^{\mu-1}w_0\|_{L^2_x}
					\ls
					\|\langle x\rangle^\mu\nabla w_0\|_{L^2_x}.
				\end{equation}
				Combining \eqref{eq:weak-weighted-Linfty-initial-displacement} and
				\eqref{eq:weak-Hardy-w0}, we arrive at
				\begin{equation}\label{eq:weak-weighted-Linfty-initial-displacement-final}
					\|\langle x\rangle^\mu w_0\|_{L^\infty_x}
					\ls
					\sum_{1\le |a|\le2}
					\|\langle x\rangle^\mu\nabla^a w_0\|_{L^2_x}.
				\end{equation}
				Note that by the Kirchhoff formula,
				\begin{equation}
					\label{eq:weak-Tt-w0-Kirchhoff}
					\mathcal T_t w_0(x)
					=
					\frac1{4\pi}
					\int_{\mathbb S^2}w_0(x+t\omega)\,d\omega
					+
					\frac{t}{4\pi}
					\int_{\mathbb S^2}\omega\cdot\nabla w_0(x+t\omega)\,d\omega.
				\end{equation}
				For the first term in the right-hand side of
				\eqref{eq:weak-Tt-w0-Kirchhoff}, by
				\eqref{eq:weak-weighted-Linfty-initial-displacement-final},
				Cauchy's inequality and Lemma~\ref{lem:weak-geom-sphere} with \(s=0\),
				we arrive at
				\begin{equation}
					\label{eq:weak-initial-displacement-spherical-average}
					\begin{aligned}
						&
						A^\mu(t,x)
						\big|
						\int_{\mathbb S^2}
						w_0(x+t\omega)\,d\omega
						\big|
						\\
						&\le
						A^\mu(t,x)
						\|\langle y\rangle^\mu w_0\|_{L^\infty_y}
						\int_{\mathbb S^2}
						\langle x+t\omega\rangle^{-\mu}\,d\omega
						\\
						&\lesssim
						A^\mu(t,x)
						\|\langle y\rangle^\mu w_0\|_{L^\infty_y}
						|\Sph|^{\f12}
						\left(
						\int_{\mathbb S^2}
						A^{-2\mu}(0,x+t\omega)\,d\omega
						\right)^{1/2}
						\\
						&\ls
						\|\langle y\rangle^\mu w_0\|_{L^\infty_y}
						\ls
						\sum_{1\le |a|\le2}
						\|\langle x\rangle^\mu\nabla^a w_0\|_{L^2_x}.
					\end{aligned}
				\end{equation}
				For the second term in the right-hand side of
				\eqref{eq:weak-Tt-w0-Kirchhoff}, as in the proof of
				\eqref{eq:weak-S-basic}, one has
				\begin{equation}
					\label{eq:weak-initial-displacement-gradient-average}
					\big\|
					A^\mu
					t
					\int_{\mathbb S^2}
					\omega\cdot\nabla w_0(x+t\omega)\,d\omega
					\big\|_{L^\infty_{t,x}}
					\ls
					\sum_{1\le |a|\le3}
					\|\langle x\rangle^\mu\nabla^a w_0\|_{L^2_x}.
				\end{equation}
				By \eqref{eq:weak-initial-displacement-spherical-average} and
				\eqref{eq:weak-initial-displacement-gradient-average}, we obtain
				\begin{equation}
					\label{eq:weak-Tt-w0-no-zero-order}
					\|A^\mu\mathcal T_t w_0\|_{L^\infty_{t,x}}
					\ls
					\sum_{1\le |a|\le3}
					\|\langle x\rangle^\mu\nabla^a w_0\|_{L^2_x}.
				\end{equation}
				Collecting \eqref{eq:weak-Duhamel-w}, \eqref{eq:weak-St-w1-w} and
				\eqref{eq:weak-Tt-w0-no-zero-order} yields
				\eqref{eq:weak-weighted-Linfty-L2-w}.
				
				We next prove \eqref{eq:weak-weighted-Linfty-L2-dw}. Note that
				\begin{equation*}
					\nabla w_F(t)
					=\int_0^t
					\mathcal S_{t-s}\nabla F(s)\,ds
				\quad\text{and}\quad
				\partial_t w_F(t)
					=\int_0^t \mathcal T_{t-s}F(s)\,ds.
				\end{equation*}
				Thus, for \(|a|\le m\), Lemma~\ref{lem:weak-weighted-wave-propagators} implies
				\begin{equation}
					\label{eq:weak-Duhamel-dw}
					\begin{aligned}
						\|A^\mu\nabla^a\partial w_F\|_{L^\infty_{t,x}}
						\ls
						\sum_{|b|\le |a|+3}
						\|A^\mu\nabla^bF\|_{L^1_tL^2_x}.
					\end{aligned}
				\end{equation}
				On the other hand, it holds $\nabla w_{\mathrm{hom}}
					=\mathcal T_t\nabla w_0+\mathcal S_t\nabla w_1$.
				Therefore, by Lemma~\ref{lem:weak-weighted-wave-propagators},
				\begin{equation}
					\label{eq:weak-homogeneous-space-dw}
					\begin{aligned}
						\|A^\mu\nabla^a\nabla w_{\mathrm{hom}}\|_{L^\infty_{t,x}}
						&\ls
						\sum_{|b|\le |a|+3}
						\|\langle x\rangle^\mu\nabla^b\nabla w_0\|_{L^2_x}
						+
						\sum_{|b|\le |a|+2}
						\|\langle x\rangle^\mu\nabla^b\nabla w_1\|_{L^2_x}
						\\
						&\ls
						\sum_{|b|\le |a|+3}
						\|\langle x\rangle^\mu\nabla^b\partial w(0)\|_{L^2_x}.
					\end{aligned}
				\end{equation}
				
				Due to
				\begin{equation}
					\label{eq:weak-homogeneous-time-identity}
					\begin{aligned}
						\partial_t w_{\mathrm{hom}}(t)
						=
						\partial_t\mathcal T_t w_0
						+
						\partial_t\mathcal S_t w_1
						=
						\mathcal S_t\Delta w_0+\mathcal T_t w_1,
					\end{aligned}
				\end{equation}
				one has
				\begin{equation}
					\label{eq:weak-homogeneous-time-dw}
					\begin{aligned}
						\|A^\mu\nabla^a\partial_t w_{\mathrm{hom}}\|_{L^\infty_{t,x}}
						&\ls
						\sum_{|b|\le |a|+2}
						\|\langle x\rangle^\mu\nabla^b\Delta w_0\|_{L^2_x}
						+
						\sum_{|b|\le |a|+3}
						\|\langle x\rangle^\mu\nabla^b w_1\|_{L^2_x}
						\\
						&\ls
						\sum_{|b|\le |a|+3}
						\|\langle x\rangle^\mu\nabla^b\partial w(0)\|_{L^2_x}.
					\end{aligned}
				\end{equation}
				
				Combining \eqref{eq:weak-Duhamel-dw},
				\eqref{eq:weak-homogeneous-space-dw} and
				\eqref{eq:weak-homogeneous-time-dw}, we obtain that for \(|a|\le m\),
				\begin{equation}\label{YHC-14}
					\|A^\mu\nabla^a\partial w\|_{L^\infty_{t,x}}
					\ls
					\sum_{|b|\le |a|+3}
					\big[
					\|\langle x\rangle^\mu\nabla^b\partial w(0)\|_{L^2_x}
					+\|A^\mu\nabla^bF\|_{L^1_tL^2_x}
					\big].
				\end{equation}
				Summing over \(|a|\le m\) for \eqref{YHC-14} yields \eqref{eq:weak-weighted-Linfty-L2-dw}.
			\end{proof}

	\subsection{Weighted $L^\infty$-$L^2$ estimates with the weight $AB^{\mu}$}\label{subsec:weighted-linear-Linfty-L2-estimate-2}
	
	In this subsection, we continue to establish the weighted \(L^\infty\)-\(L^2\) estimates with the weight \(AB^\mu\),
 which will be used later in Section \ref{sec:weighted-pointwise-estimates-I}. The introduction of the weight
function $B=1+\bigl|t-|x|\bigr|$ is motivated by the Klainerman-Sobolev inequality (see \eqref{eq:1.14}). Although the related proofs in this subsection are somewhat similar to those in Subsection \ref{subsec:weighted-linear-Linfty-L2-estimate-1}, we still provide the
 full details for the reader's convenience.

	At first, we give an integral estimate on shifted spheres as follows.
	
	\begin{lemma}\label{lem:geom-sphere}
		Let $\mu>0$, $0\le s\le t$, $\rho=t-s\ge 0$ and $x\in\mathbb R^3$.
		Then
		\begin{equation}\label{eq:geom-sphere-equiv}
			\int_{|y-x|=\rho}
			(1+s+|y|)^{-2-2\mu}\,dS_y
			\ls
			\rho^2 A^{-2}(t,x)B^{-2\mu}(t,x).
		\end{equation}
	\end{lemma}
	
	\begin{proof}
		Set \(r=|x|\). Substituting $h(\lambda)=(1+s+\lambda)^{-2-2\mu}$ into
		\eqref{eq:sphere-one-dimensional-reduction} yields that for $r>0$,
		\begin{equation}\label{eq:I-definition}
			I=
			\int_{|y-x|=\rho}
			(1+s+|y|)^{-2-2\mu}\,dS_y
			=
			\frac{2\pi\rho}{r}
			\int_{|r-\rho|}^{r+\rho}
			\lambda(1+s+\lambda)^{-2-2\mu}\,d\lambda .
		\end{equation}
		We next prove
		\begin{equation}\label{eq:I-goal}
			I\ls \rho^2 A^{-2}(t,x)B^{-2\mu}(t,x).
		\end{equation}

		\smallskip
		\noindent
		\textbf{Case 1. $r\le \rho$}
\vskip 0.1 true cm

		Note that $B(t,x)=1+|s+\rho-r|=1+s+\rho-r$
		and $A(t,x)=B(t,x)+2r$.
		
		If $r\le A(t,x)/4$, then $B(t,x)\ge A(t,x)/2$. We can obtain
		\begin{equation*}
			\begin{aligned}
				\int_{|r-\rho|}^{r+\rho} \lambda(1+s+\lambda)^{-2-2\mu}\,d\lambda
				\le
				\int_{\rho-r}^{\rho+r} 2\rho(1+s+\rho-r)^{-2-2\mu}\,d\lambda
				\le
				4 r\rho\,B^{-2-2\mu}.
			\end{aligned}
		\end{equation*}
		This, together with $A/2 \le B$, yields
		$I\ls \frac{\rho}{r} r\rho A^{-2}B^{-2\mu}
			\ls \rho^2 A^{-2}B^{-2\mu}$.

		If $r>A(t,x)/4$, then $r\rho/A^2(t,x)\gtrsim 1$. Moreover, it holds that
		\begin{equation*}
			\begin{aligned}
				\int_{|r-\rho|}^{r+\rho} \lambda(1+s+\lambda)^{-2-2\mu}\,d\lambda
				&\le
				\int_{\rho-r}^\infty (1+s+\lambda)^{-1-2\mu}\,d\lambda  \ls
				(1+s+\rho-r)^{-2\mu}
				\ls
				B^{-2\mu}.
			\end{aligned}
		\end{equation*}
		Due to $r\rho/A^2\gtrsim1$, one has
		$I\ls \frac{\rho}{r}B^{-2\mu}
			\ls \rho^2 A^{-2}B^{-2\mu}$.

		\vskip 0.1 true cm
		\noindent
		\textbf{Case 2. $r\ge \rho$}
\vskip 0.1 true cm

		Note that $A(t,x)=1+s+r+\rho$ and $B(t,x)=1+|s+\rho-r|$.

		If $\rho\le A(t,x)/4$, then $1+s+r-\rho=A(t,x)-2\rho\ge A(t,x)/2$. We have
		\begin{equation*}
			\begin{aligned}
				\int_{|r-\rho|}^{r+\rho} \lambda(1+s+\lambda)^{-2-2\mu}\,d\lambda
				\le
				\int_{r-\rho}^{r+\rho} 2r(1+s+r-\rho)^{-2-2\mu}\,d\lambda
				\ls
				r\rho A^{-2-2\mu}(t,x).
			\end{aligned}
		\end{equation*}
		By $B(t,x)\le A(t,x)$, one has
		$I\ls \frac{\rho}{r} r\rho A^{-2-2\mu}(t,x)
			\ls \rho^2 A^{-2}B^{-2\mu}(t,x)$.

		If $\rho>A(t,x)/4$, then $r\rho/A^2(t,x)\gtrsim1$.
		Moreover, by $1+s+r-\rho\ge B(t,x) $, one has
		\begin{equation*}
			\begin{aligned}
				\int_{|r-\rho|}^{r+\rho} \lambda(1+s+\lambda)^{-2-2\mu}\,d\lambda
				&\le
				\int_{r-\rho}^\infty (1+s+\lambda)^{-1-2\mu}\,d\lambda\ls
				(1+s+r-\rho)^{-2\mu}
				\ls B^{-2\mu}(t,x).
			\end{aligned}
		\end{equation*}
		Then $I\ls \frac{\rho}{r} r\rho A^{-2}(t,x) B^{-2\mu}(t,x)
			\ls \rho^2 A^{-2}(t,x)B^{-2\mu}(t,x)$.

		Collecting two cases above yields \eqref{eq:I-goal} when $r>0$.
		
		It remains to consider the case $r=0$. At this time, we have $|y|=\rho$ and
		\begin{equation*}
			\begin{aligned}
				I
				&=
				4\pi\rho^2(1+s+\rho)^{-2-2\mu}
				=
				4\pi\rho^2(1+t)^{-2-2\mu}.
			\end{aligned}
		\end{equation*}
		It follows from $A=B=1+t$ that
		$I\ls\rho^2 A^{-2}(t,x)B^{-2\mu}(t,x)$.
		Thus, \eqref{eq:geom-sphere-equiv} holds.
	\end{proof}

	We next establish the weighted $L^\infty$-$L^2$ estimates with the weight $AB^\mu$ for the 3D wave propagators.

	\begin{lemma}\label{lem:weighted-wave-propagators}
		Assume $\mu>0$, $T>0$ and $s\in[0,T]$.
		For $\rho\ge0$, define
		\begin{equation*}
			\mathcal{S}_{\rho}
			=\frac{\sin(\rho|\nabla|)}{|\nabla|},
			\qquad
			\mathcal{T}_{\rho}
			=\cos(\rho|\nabla|).
		\end{equation*}
		Then it holds that
		\begin{equation}
			\label{eq:Srho-basic-w}
			\left\|
			A(t,x)B^\mu(t,x)
			\mathcal{S}_{t-s}f
			\right\|_{L^\infty([s,T]\times\R^3)}
			\ls
			\sum_{|a|\le2}
			\left\|
			A^{1+\mu}(s,\cdot)\nabla^a f
			\right\|_{L^2_x},
		\end{equation}
		\begin{equation}
			\label{eq:dt-Srho-basic-w}
			\left\|
			A(t,x)B^\mu(t,x)
			\mathcal{T}_{t-s}f
			\right\|_{L^\infty([s,T]\times\R^3)}
			\lesssim
			\sum_{|a|\le3}
			\left\|
			A^{1+\mu}(s,\cdot)\nabla^a f
			\right\|_{L^2_x}.
		\end{equation}
	\end{lemma}
	
	\begin{proof}
		Without loss of generality, $A(t,x)=1+t+\w{x}$ is assumed.

		We first prove \eqref{eq:Srho-basic-w}. Fix $t\in[s,T]$ and set $\rho=t-s$. By
		\eqref{eq:weak-S-sphere}, we have
		\begin{equation*}
			\begin{aligned}
				|\mathcal{S}_{\rho} f(x)|
				&\ls\rho
				\big(
				\int_{\mathbb S^2}
				A^{-2-2\mu}(s,x+\rho\omega)\,d\omega
				\big)^{1/2}
				\big(
				\int_{\mathbb S^2}
				A^{2+2\mu}(s,x+\rho\omega)
				|f(x+\rho\omega)|^2\,d\omega
				\big)^{1/2}.
			\end{aligned}
		\end{equation*}
		Lemma \ref{lem:geom-sphere} gives
		\begin{equation*}
			A(t,x)B^\mu(t,x)
			\big(
			\int_{\mathbb S^2}
			A^{-2-2\mu}(s,x+\rho\omega)\,d\omega
			\big)^{1/2}
			\ls 1,
		\end{equation*}
		where $1+s+\langle y\rangle\sim1+s+|y|$ is used. By Lemma \ref{lem:sphere-trace-H32}, one has
		\begin{equation*}
			\rho
			\big(
			\int_{\mathbb S^2}
			A^{2+2\mu}(s,x+\rho\omega)
			|f(x+\rho\omega)|^2\,d\omega
			\big)^{1/2}
			\lesssim
			\|A^{1+\mu}(s,\cdot)f\|_{H^2_x}.
		\end{equation*}
		Due to $|\nabla^\gamma A^{1+\mu}(s,x)|
			\lesssim
			A^{1+\mu}(s,x)$
			for $|\gamma|\le 2$,
		we can obtain
		\begin{equation}\label{eq:weighted-H2-product-w}
			\|A^{1+\mu}(s,\cdot)f\|_{H^2_x}
			\ls
			\sum_{|a|\le2}
			\|A^{1+\mu}(s,\cdot)\nabla^a f\|_{L^2_x}.
		\end{equation}
		Therefore, \eqref{eq:Srho-basic-w} is proved. For
		\eqref{eq:dt-Srho-basic-w}, the first term in the right-hand side of
		\eqref{eq:weak-T-sphere} can be treated as in the above. The second term in \eqref{eq:weak-T-sphere}
is also treated analogously to 	$\mathcal{S}_{\rho}(\nabla f)$. Then
		\eqref{eq:dt-Srho-basic-w} is proved.
	\end{proof}
	
	We now give the weighted $L^\infty$-$L^2$ estimates with the weight $A B^\mu$ for 3D linear wave equation.
	
		\begin{lemma}[{\bf Weighted $L^\infty$-$L^2$ estimates with the weight $A B^\mu$}]\label{lem:weighted-Linfty-L2}
		Assume $\mu>0$ and $m\in\N_0$. Let $w$ solve
		\begin{equation*}
			\Box w=F,\quad
			(w,\partial_t w)(0,x)=(w_0,w_1)(x),
		\end{equation*}
		where $(t,x)\in [0,T]\times\mathbb R^3$ and $w_0\in L^2(\R^3)$.
		Then one has
		\begin{equation}\label{eq:weighted-Linfty-L2-w}
				\begin{aligned}
					\|A B^\mu w\|_{L^\infty([0,T]\times\mathbb R^3)}
					&\lesssim
					\sum_{|a|\le 2}
					\big[
					\left\|
					\langle x\rangle^{1+\mu}
					\nabla^a \p w(0)
					\right\|_{L^2(\R^3)}
					+
					\left\|
					A^{1+\mu}\nabla^a F
					\right\|_{L^1([0,T];L^2_x)}
					\big],
				\end{aligned}
		\end{equation}
		\begin{equation}\label{eq:weighted-Linfty-L2-dw}
			\begin{aligned}
				\sum_{|a|\le m}
				\|A B^\mu \nabla^a\partial w\|_{L^\infty([0,T]\times\mathbb R^3)}
				&\lesssim
				\sum_{|a|\le m+3}
				\big[
				\left\|
				\langle x\rangle^{1+\mu}
				\nabla^a \p w(0)
				\right\|_{L^2(\R^3)}
				+
				\left\|
				A^{1+\mu}\nabla^a F
				\right\|_{L^1([0,T];L^2_x)}
				\big].
			\end{aligned}
		\end{equation}
	\end{lemma}
	
	\begin{proof}
		We first prove \eqref{eq:weighted-Linfty-L2-w}. As in Lemma \ref{lem:weak-weighted-Linfty-L2}, decompose
		$w=w_{\mathrm{hom}}+w_F$.
		
		It follows from  \eqref{eq:Srho-basic-w} that
		\begin{equation}\label{eq:wF-zero-order-w}
			\|A B^\mu w_F\|_{L^\infty_{t,x}}
			\ls
			\sum_{|a|\le2}
			\|A^{1+\mu}\nabla^a F\|_{L^1_tL^2_x}.
		\end{equation}
		
		For the term $\mathcal{S}_{t} w_1$, by \eqref{eq:Srho-basic-w} with $s=0$,
		one has
		\begin{equation}\label{eq:St-w1-w}
			\|A B^\mu \mathcal{S}_{t} w_1\|_{L^\infty_{t,x}}
			\ls
			\sum_{|a|\le2}
			\|\langle x\rangle^{1+\mu}\nabla^a w_1\|_{L^2_x}.
		\end{equation}
		
		Next, we estimate $\mathcal{T}_t w_0$.
		At first, as in \eqref{eq:weak-weighted-Linfty-initial-displacement-final}, we can obtain
		\begin{equation}\label{eq:weighted-Linfty-initial-displacement}
			\begin{aligned}
				\left\|
				\langle x\rangle^{1+\mu}w_0
				\right\|_{L^\infty_x}
				&\ls
				\sum_{1\le|a|\le2}
				\left\|
				\langle x\rangle^{1+\mu}
				\nabla^a w_0
				\right\|_{L^2_x}.
			\end{aligned}
		\end{equation}
		Then, by the Cauchy-Schwarz inequality and Lemma~\ref{lem:geom-sphere}, one has
		\begin{equation}\label{eq:initial-displacement-spherical-average}
			\begin{aligned}
				&
				A(t,x)B(t,x)^\mu
				\big|
				\int_{\mathbb S^2}
				w_0(x+t\omega)\,d\omega
				\big|\\
				&\le
				A(t,x)B(t,x)^\mu
				\big\|
				\langle y\rangle^{1+\mu}w_0
				\big\|_{L^\infty_y}
				\int_{\mathbb S^2}
				\langle x+t\omega\rangle^{-1-\mu}\,d\omega
				\\
				&\ls
				A(t,x)B(t,x)^\mu
				\big\|
				\langle y\rangle^{1+\mu}w_0
				\big\|_{L^\infty_y}
				|\Sph|^{\f12}
				\left(\int_{\mathbb S^2}
				A^{-2-2\mu}(0,x+t\omega)\,d\omega\right)^{\f12}
				\\
				&\ls
				\big\|
				\langle y\rangle^{1+\mu}w_0
				\big\|_{L^\infty_y}
				\ls
				\sum_{1\le|a|\le2}
				\big\|
				\langle x\rangle^{1+\mu}
				\nabla^a w_0
				\big\|_{L^2_x}.
			\end{aligned}
		\end{equation}
		On the other hand, as in \eqref{eq:Srho-basic-w}, we have
		\begin{equation}\label{eq:initial-displacement-gradient-average}
			\begin{aligned}
				&
				\big\|
				A B^\mu
				t
				\int_{\mathbb S^2}
				\omega\cdot\nabla  w_0(x+t\omega)\,d\omega
				\big\|_{L^\infty_{t,x}}
				\ls
				\sum_{1\le|a|\le3}
				\left\|
				\langle x\rangle^{1+\mu}
				\nabla^a  w_0
				\right\|_{L^2_x}.
			\end{aligned}
		\end{equation}
		Combining \eqref{eq:initial-displacement-spherical-average} and
		\eqref{eq:initial-displacement-gradient-average} yields
		\begin{equation}\label{eq:dt-St-w0-w}
			\|A B^\mu \mathcal{T}_t  w_0\|_{L^\infty_{t,x}}
			\ls
			\sum_{1\le|a|\le3}
			\|\langle x\rangle^{1+\mu}\nabla^a  w_0\|_{L^2_x}.
		\end{equation}
		Collecting \eqref{eq:wF-zero-order-w}, \eqref{eq:St-w1-w} and
		\eqref{eq:dt-St-w0-w} gives \eqref{eq:weighted-Linfty-L2-w}.
		
		Next, we estimate \eqref{eq:weighted-Linfty-L2-dw}. Note that $\nabla w_F(t)
			=\int_0^t \mathcal{S}_{t-s}\nabla F(s)\,ds$
		and $\partial_t w_F(t)
			=\int_0^t \mathcal{T}_{t-s} F(s)\,ds$.
		Then it follows from \eqref{eq:Srho-basic-w} and \eqref{eq:dt-Srho-basic-w} that
		\begin{equation}\label{eq:Duhamel-dw-corrected}
			\|A B^\mu\partial w_F\|_{L^\infty_{t,x}}
			\ls
			\sum_{|a|\le3}
			\|A^{1+\mu}\nabla^a F\|_{L^1_tL^2_x}.
		\end{equation}
		
		In addition, we have
		$\nabla w_{\mathrm{hom}}
			=\mathcal{T}_t\nabla w_0+\mathcal{S}_{t}\nabla w_1$.
		Therefore, by \eqref{eq:Srho-basic-w} and
		\eqref{eq:dt-Srho-basic-w} with $s=0$, it holds that
		\begin{equation}\label{eq:homogeneous-space-dw-corrected}
			\begin{aligned}
				\|A B^\mu\nabla w_{\mathrm{hom}}\|_{L^\infty_{t,x}}
				&\ls
				\sum_{|a|\le3}
				\|\langle x\rangle^{1+\mu}\nabla^a\nabla w_0\|_{L^2_x}
				+
				\sum_{|a|\le3}
				\|\langle x\rangle^{1+\mu}\nabla^a w_1\|_{L^2_x}.
			\end{aligned}
		\end{equation}
		
		Due to
		\begin{equation}\label{eq:homogeneous-time-identity-dw-corrected}
			\begin{aligned}
				\partial_t w_{\mathrm{hom}}(t)
				=
				\partial_t^2\mathcal{S}_{t} w_0+\partial_t\mathcal{S}_{t} w_1
				=
				\Delta \mathcal{S}_{t} w_0+\mathcal{T}_t w_1
				=
				\mathcal{S}_{t}\Delta w_0+\mathcal{T}_t w_1,
			\end{aligned}
		\end{equation}
		one has
		\begin{equation}\label{eq:homogeneous-time-dw-corrected}
			\begin{aligned}
				\|A B^\mu\partial_t w_{\mathrm{hom}}\|_{L^\infty_{t,x}}
				&\ls
				\sum_{|a|\le2}
				\|\langle x\rangle^{1+\mu}\nabla^a\Delta w_0\|_{L^2_x}
				+
				\sum_{|a|\le3}
				\|\langle x\rangle^{1+\mu}\nabla^a w_1\|_{L^2_x}
				\\
				&\ls
				\sum_{|a|\le3}
				\|\langle x\rangle^{1+\mu}\nabla^a\nabla w_0\|_{L^2_x}
				+
				\sum_{|a|\le3}
				\|\langle x\rangle^{1+\mu}\nabla^a w_1\|_{L^2_x}.
			\end{aligned}
		\end{equation}
		Combining \eqref{eq:Duhamel-dw-corrected},
		\eqref{eq:homogeneous-space-dw-corrected} and
		\eqref{eq:homogeneous-time-dw-corrected}, we obtain
		\begin{equation*}
			\|A B^\mu\partial w\|_{L^\infty_{t,x}}
			\ls
			\sum_{|a|\le3}
			\left\|
			\langle x\rangle^{1+\mu}
			\nabla^a(\nabla w_0,w_1)
			\right\|_{L^2_x}
			+
			\sum_{|a|\le3}
			\|A^{1+\mu}\nabla^aF\|_{L^1_tL^2_x}.
		\end{equation*}
		Analogously, applying the same estimate to $\nabla^a w$
        yields \eqref{eq:weighted-Linfty-L2-dw} for $|a|\le m$.
	   \end{proof}

	\subsection{Weighted Strichartz estimates}
	\label{subsec:Weighted-linear-Strichartz-estimates}
	
	In this subsection, we recall the weighted Strichartz estimates established in
	\cite{GaoLiYin2026}. These estimates will be used in
	Section~\ref{sec:weighted-Strichartz-estimates} to derive the corresponding weighted
	Strichartz estimates for the solution $\phi$ of \eqref{eq:Chaplygin-Cauchy}.
	
	\begin{lemma}[{\bf Weighted Strichartz estimates}]\label{thm:Weighted Strichartz estimate}
		For $\beta_1\in(0,1)$, $\beta_2\in(\beta_1,\min\{\f32\beta_1,1\})$, $t\ge t_0\ge 0$ and integer $k\ge -1$,
		\begin{itemize}
			\item[(1)] when $p\in[2,\infty),\ r\in(2,\infty]$ with $\f1p+\f1r=\f12$, one has
			\begin{equation}\label{thm:Weighted Strichartz estimate 1}
				\|(1+s+|x|)^{\f{1}{p}\beta_1}P_{k}e^{\pm is|\nabla|}f\|_{L^p([t_0,t];L^r(\R^3))}
				\ls2^{\f2p(1+\beta_2)k }\|\w{x}^{\f2p \beta_2}P_kf\|_{L^2(\R^3)};
			\end{equation}
			\item[(2)] when $p\in[2,2+2\beta_2-\beta_1),\ r\in(2+\f{4}{2\beta_2-\beta_1},\infty]$ with $\f1p+\f1r=\f12$, one has
			\begin{equation}\label{thm:Weighted Strichartz estimate 2}
				\|(1+s+|x|)^{\f{1}{p}\beta_1}P_{k}|\nabla|^{-1}e^{\pm is|\nabla|}f\|_{L^p([t_0,t];L^r(\R^3))}
				\ls2^{(\f{2+2\beta_2}{p}-1)k }\|\w{x}^{\f2p \beta_2}P_kf\|_{L^2(\R^3)}.
			\end{equation}
		\end{itemize}
	\end{lemma}
	\begin{proof}
		See Theorem 3.1 of \cite{GaoLiYin2026}.
	\end{proof}

	\begin{lemma}\label{lem:weighted-L2-L2}
		For $\beta_1\in(0,\f32)$, $\beta_2>0$, $t\ge 0 $, integers $k\ge -1$, it holds that
		\begin{equation}\label{YHCCC-33}
			\begin{split}
				\|\w{x}^{\beta_1}P_k e^{\pm it|\nabla|}f\|_{L^2(\R^3)}
				\ls
				(1+t)^{\beta_1+\beta_2 }\|f\|_{L^2(\R^3)}
				+
				\|\w{x}^{\beta_1+\beta_2 } f\|_{L^2(\R^3)}.
			\end{split}
		\end{equation}
	\end{lemma}
	\begin{proof}
		See Corollary 3.20 of \cite{GaoLiYin2026}.
	\end{proof}

	\begin{lemma}[{\bf Weighted \(L^2_tL^\infty_x\)-\(L^2\) estimates}]
		\label{lem:weak-weighted-L2t-Linfty-L2}
		Assume \(0<\beta_1<\beta_2<1\) and $m\in\N_0$.
		Let \(w\) solve
		\begin{equation*}
			\Box w=F,\quad
			(w,\partial_t w)(0,x)=(w_0,w_1)(x),
		\end{equation*}
		where  \((t,x)\in [0,T]\times\mathbb R^3\). Then
		\begin{equation}
			\label{eq:weak-L2t-Linfty-L2-w}
			\begin{aligned}
				\big\|
				A^{\frac{\beta_1}{2}}w
				\big\|_{L^2([0,T];L^\infty_x)}
				&\lesssim
				\sum_{|a|\le1}
				\big[
				\big\|
				\langle x\rangle^{\beta_2}
				\nabla^a\partial w(0)
				\big\|_{L^2_x}
				+
				\big\|
				A^{\beta_2}\nabla^aF
				\big\|_{L^1([0,T];L^2_x)}
				\big].
			\end{aligned}
		\end{equation}
		Moreover,
		\begin{equation}
			\label{eq:weak-L2t-Linfty-L2-dw}
			\begin{aligned}
				\sum_{|a|\le m}
				\big\|
				A^{\frac{\beta_1}{2}}\nabla^a\partial w
				\big\|_{L^2([0,T];L^\infty_x)}
				&\lesssim
				\sum_{|a|\le m+2}
				\big[
				\big\|
				\langle x\rangle^{\beta_2}
				\nabla^a\partial w(0)
				\big\|_{L^2_x}
				+\big\|
				A^{\beta_2}\nabla^aF
				\big\|_{L^1([0,T];L^2_x)}
				\big].
			\end{aligned}
		\end{equation}
	\end{lemma}

	\begin{proof}
		Choosing $\widetilde\beta
		\in\big(\beta_1,
		\min\big\{\frac32\beta_1,\beta_2\big\}
		\big)$.
		Let \(0\le s\le T\). Writing
		\[
		\mathcal T_{t-s}f
		=
		\frac12 e^{i(t-s)|\nabla|}f
		+
		\frac12 e^{-i(t-s)|\nabla|}f.
		\]
		It follows from
		$e^{\pm i(t-s)|\nabla|}f
		=e^{\pm it|\nabla|}
		e^{\mp is|\nabla|}f,
		$
		Lemma~\ref{thm:Weighted Strichartz estimate} with parameters
		\((\beta_1,\widetilde\beta)\), \(p=2\) and \(r=\infty\) that  for
		each \(k\ge-1\),
		\begin{equation}
			\label{eq:dyadic-shifted-T-Strichartz}
			\begin{aligned}
				\big\|
				A^{\frac{\beta_1}{2}}(t,x)
				P_k\mathcal T_{t-s}f
				\big\|_{L^2([s,T];L^\infty_x)}
				\lesssim
				2^{(1+\widetilde\beta)k}
				\big\|
				\langle x\rangle^{\widetilde\beta}
				P_{k}e^{\mp is|\nabla|}f
				\big\|_{L^2_x}.
			\end{aligned}
		\end{equation}
	By Lemma~\ref{lem:weighted-L2-L2}, we have
		\begin{equation}
			\label{eq:weighted-L2-L2-used-for-shift}
			\begin{aligned}
				\big\|
				\langle x\rangle^{\widetilde\beta}
				P_{k}e^{\mp is|\nabla|}f
				\big\|_{L^2_x}
				&\lesssim_{\beta_1,\beta_2}
				(1+s)^{\beta_2}
				\|P_{[[k]]}f\|_{L^2_x}
				+\big\|
				\langle x\rangle^{\beta_2}
				P_{[[k]]}f
				\big\|_{L^2_x}\\
				&\lesssim
				\big\|
				A^{\beta_2}(s,\cdot)P_{[[k]]}f
				\big\|_{L^2_x}.
			\end{aligned}
		\end{equation}
		Combining \eqref{eq:dyadic-shifted-T-Strichartz}-\eqref{eq:weighted-L2-L2-used-for-shift} and summing over \(k\ge-1\) yield
		\begin{equation}
			\label{eq:shifted-T-L2t-Linfty}
			\begin{aligned}
				\big\|
				A^{\frac{\beta_1}{2}}(t,x)
				\mathcal T_{t-s}f
				\big\|_{L^2([s,T];L^\infty_x)}
				&\lesssim
				\sum_{k\ge-1}
				2^{(1+\widetilde\beta)k}
				\big\|
				A^{\beta_2}(s,\cdot)P_{[[k]]}f
				\big\|_{L^2_x}
				\\
				&\lesssim
				\sum_{|a|\le2}
				\big\|
				A^{\beta_2}(s,\cdot)\nabla^a f
				\big\|_{L^2_x}.
			\end{aligned}
		\end{equation}
		In the last inequality of \eqref{eq:shifted-T-L2t-Linfty} we have used \(1+\widetilde\beta<2\), Lemma~\ref{lem:weighted bernstein},
		\(A^{2\beta_2}(s,x) \sim (1+s)^{2\beta_2}+\w{x}^{2\beta_2}\) and $\w{x}^{2\beta_2} \in A_2(\R^3)$.
		
		Similarly, by
		$\mathcal S_{t-s}f
		=\frac{1}{2i}|\nabla|^{-1}
		\big(e^{i(t-s)|\nabla|}
		-e^{-i(t-s)|\nabla|}
		\big)f$, Lemma~\ref{thm:Weighted Strichartz estimate} with parameters
		\((\beta_1,\widetilde\beta)\), \(p=2\) and \(r=\infty\), we arrive at
		\[
		\big\|
		A^{\frac{\beta_1}{2}}(t,x)
		P_k\mathcal S_{t-s}f
		\big\|_{L^2([s,T];L^\infty_x)}
		\lesssim
		2^{\widetilde\beta k}
		\big\|
		\langle x\rangle^{\widetilde\beta}
		P_{k}e^{\mp is|\nabla|}f
		\big\|_{L^2_x}.
		\]
		Using \eqref{eq:weighted-L2-L2-used-for-shift} and summing over \(k\) yield
		\begin{equation}
			\label{eq:shifted-S-L2t-Linfty}
			\big\|
			A^{\frac{\beta_1}{2}}(t,x)
			\mathcal S_{t-s}f
			\big\|_{L^2([s,T];L^\infty_x)}
			\lesssim
			\sum_{|a|\le1}
			\big\|
			A^{\beta_2}(s,\cdot)\nabla^a f
			\big\|_{L^2_x}.
		\end{equation}
		
		As in \eqref{YHC-36}, decompose $w=w_{\mathrm{hom}}+w_F$.
		We first prove \eqref{eq:weak-L2t-Linfty-L2-w}. It follows from \eqref{eq:shifted-S-L2t-Linfty} and Minkowski's inequality
        that
		\begin{equation}
			\label{eq:weak-L2t-Duhamel-w}
			\begin{aligned}
				\big\|
				A^{\frac{\beta_1}{2}}w_F
				\big\|_{L^2_tL^\infty_x}
				&\le
				\int_0^T
				\big\|
				A^{\frac{\beta_1}{2}}(t,x)
				\mathcal S_{t-s}F(s,\cdot)
				\big\|_{L^2([s,T];L^\infty_x)}
				\,ds
				\\
				&\lesssim
				\sum_{|a|\le1}
				\int_0^T
				\big\|
				A^{\beta_2}(s,\cdot)\nabla^aF(s,\cdot)
				\big\|_{L^2_x}\,ds
				\\
				&=
				\sum_{|a|\le1}
				\big\|
				A^{\beta_2}\nabla^aF
				\big\|_{L^1_tL^2_x}.
			\end{aligned}
		\end{equation}
		In addition, applying \eqref{eq:shifted-S-L2t-Linfty} with \(s=0\) derives
		\begin{equation}
			\label{eq:weak-L2t-St-w1}
			\big\|
			A^{\frac{\beta_1}{2}}\mathcal S_t w_1
			\big\|_{L^2_tL^\infty_x}
			\lesssim
			\sum_{|a|\le1}
			\big\|
			\langle x\rangle^{\beta_2}\nabla^a w_1
			\big\|_{L^2_x}.
		\end{equation}
		In order to estimate \(\mathcal T_t w_0\), writing
		\[
		\mathcal T_t w_0
		=\frac12|\nabla|^{-1}e^{it|\nabla|}|\nabla|w_0
		+\frac12|\nabla|^{-1}e^{-it|\nabla|}|\nabla|w_0.
		\]
	By the same argument used in the proof of \eqref{eq:shifted-S-L2t-Linfty}, we have
		\[
		\big\|
		A^{\frac{\beta_1}{2}}\mathcal T_t w_0
		\big\|_{L^2_tL^\infty_x}
		\lesssim
		\sum_{|a|\le1}
		\big\|
		\langle x\rangle^{\beta_2}\nabla^a|\nabla|w_0
		\big\|_{L^2_x}.
		\]
		Due to \(\beta_2<1\), one has \(\langle x\rangle^{2\beta_2}\in A_2(\mathbb R^3)\).
		Therefore, by Lemma \ref{lem:riesz L2}, one arrives at
		\begin{equation}
			\label{eq:weak-L2t-Tt-w0}
			\big\|
			A^{\frac{\beta_1}{2}}\mathcal T_t w_0
			\big\|_{L^2_tL^\infty_x}
			\lesssim
			\sum_{|a|\le1}
			\big\|
			\langle x\rangle^{\beta_2}
			\nabla^a\nabla w_0
			\big\|_{L^2_x}.
		\end{equation}
		Combining \eqref{eq:weak-L2t-Duhamel-w},
		\eqref{eq:weak-L2t-St-w1} and
		\eqref{eq:weak-L2t-Tt-w0} yields
		\eqref{eq:weak-L2t-Linfty-L2-w}.
		
		We next prove \eqref{eq:weak-L2t-Linfty-L2-dw}. By
		$\nabla w_F(t)
		=\int_0^t
		\mathcal S_{t-s}\nabla F(s)\,ds$
		and
		$
		\partial_t w_F(t)
		=
		\int_0^t
		\mathcal T_{t-s}F(s)\,ds,
		$
		it follows from \eqref{eq:shifted-T-L2t-Linfty}, \eqref{eq:shifted-S-L2t-Linfty}
		 and Minkowski's inequality that for \(|a|\le m\),
		\begin{equation}
			\label{eq:weak-L2t-Duhamel-dw}
			\begin{aligned}
				\big\|
				A^{\frac{\beta_1}{2}}\nabla^a\partial w_F
				\big\|_{L^2_tL^\infty_x}
				\lesssim
				\sum_{|b|\le |a|+2}
				\big\|
				A^{\beta_2}\nabla^bF
				\big\|_{L^1_tL^2_x}.
			\end{aligned}
		\end{equation}
		
		In addition, it follows that $\nabla w_{\mathrm{hom}}
		=\mathcal T_t\nabla w_0+\mathcal S_t\nabla w_1$.
		Therefore, by \eqref{eq:shifted-T-L2t-Linfty} and
		\eqref{eq:shifted-S-L2t-Linfty}, we have that for \(|a|\le m\),
		\begin{equation}
			\label{eq:weak-L2t-homogeneous-space-dw}
			\begin{aligned}
				\big\|
				A^{\frac{\beta_1}{2}}
				\nabla^a\nabla w_{\mathrm{hom}}
				\big\|_{L^2_tL^\infty_x}
				&\lesssim
				\sum_{|b|\le |a|+2}
				\big\|
				\langle x\rangle^{\beta_2}
				\nabla^b\nabla w_0
				\big\|_{L^2_x}
				+
				\sum_{|b|\le |a|+1}
				\big\|
				\langle x\rangle^{\beta_2}
				\nabla^b\nabla w_1
				\big\|_{L^2_x}\\
				&\lesssim
				\sum_{|b|\le |a|+2}
				\big\|
				\langle x\rangle^{\beta_2}
				\nabla^b\partial w(0)
				\big\|_{L^2_x}.
			\end{aligned}
		\end{equation}
		
		On the other hand, one has $
			\partial_t w_{\mathrm{hom}}(t)
			=\mathcal S_t\Delta w_0+\mathcal T_t w_1$.
		Hence, by \eqref{eq:shifted-S-L2t-Linfty} and
		\eqref{eq:shifted-T-L2t-Linfty}, one has that for \(|a|\le m\),
		\begin{equation}\label{eq:weak-L2t-homogeneous-time-dw}
			\begin{aligned}
				\big\|
				A^{\frac{\beta_1}{2}}
				\nabla^a\partial_t w_{\mathrm{hom}}
				\big\|_{L^2_tL^\infty_x}
				&\lesssim
				\sum_{|b|\le |a|+1}
				\big\|
				\langle x\rangle^{\beta_2}
				\nabla^b\Delta w_0
				\big\|_{L^2_x}
				+
				\sum_{|b|\le |a|+2}
				\big\|
				\langle x\rangle^{\beta_2}
				\nabla^b w_1
				\big\|_{L^2_x}
				\\
				&\lesssim
				\sum_{|b|\le |a|+2}
				\big\|
				\langle x\rangle^{\beta_2}
				\nabla^b\partial w(0)
				\big\|_{L^2_x}.
			\end{aligned}
		\end{equation}
		
		Combining \eqref{eq:weak-L2t-Duhamel-dw},
		\eqref{eq:weak-L2t-homogeneous-space-dw} and
		\eqref{eq:weak-L2t-homogeneous-time-dw} yields that for \(|a|\le m\),
		\begin{equation}\label{YHC-15}
		\begin{aligned}
		\big\|
		A^{\frac{\beta_1}{2}}\nabla^a\partial w
		\big\|_{L^2_tL^\infty_x}
		\lesssim
		\sum_{|b|\le |a|+2}
		\big[
		\big\|
		\langle x\rangle^{\beta_2}
		\nabla^b\partial w(0)
		\big\|_{L^2_x}
		+\big\|
		A^{\beta_2}\nabla^bF
		\big\|_{L^1_tL^2_x}
		\big].
		\end{aligned}
		\end{equation}
		Summing over \(|a|\le m\) for \eqref{YHC-15} yields
		\eqref{eq:weak-L2t-Linfty-L2-dw}.
	\end{proof}

	\subsection{Null conditions and good derivatives}\label{subsec:Good-derivative}
	
	In this subsection, we give some important properties of null forms.

	\begin{lemma}\label{lem:null}
		Assume that the constants $N_1^{\alpha\beta}$, $N_2^{\alpha\beta\gamma}$ and
		$N_3^{\alpha\beta\gamma\delta}$ satisfy the null conditions
		\begin{equation}\label{eq:null-condition-tensors}
			\begin{aligned}
				\sum_{\alpha,\beta=0}^3
				N_1^{\alpha\beta}\xi_\alpha\xi_\beta\equiv 0,
				\quad
				\sum_{\alpha,\beta,\gamma=0}^3
				N_2^{\alpha\beta\gamma}\xi_\alpha\xi_\beta\xi_\gamma\equiv0,
				\quad
				\sum_{\alpha,\beta,\gamma,\delta=0}^3
				N_3^{\alpha\beta\gamma\delta}\xi_\alpha\xi_\beta\xi_\gamma\xi_\delta\equiv0,
			\end{aligned}
		\end{equation}
		where $\xi=(\xi_0,\xi_1,\xi_2,\xi_3)=(\pm1,\omega)$ with $\omega\in\mathbb S^2$.
		Then, for smooth functions $f,g,h,w$ on $\mathbb R^{1+3}$, one has that for
		$r=|x|>0$,
		\begin{equation}\label{eq:null-structure-estimates}
			\begin{aligned}
				\big|
				\sum_{\alpha,\beta=0}^3
				N_1^{\alpha\beta}\partial_{\alpha}f\,\partial_{\beta}g
				\big|
				&\lesssim
				|\bar\partial f|\,|\partial g|
				+
				|\partial f|\,|\bar\partial g|,
				\\
				\big|
				\sum_{\alpha,\beta,\gamma=0}^3
				N_2^{\alpha\beta\gamma}\partial^2_{\alpha\beta}f\,\partial_{\gamma}g
				\big|
				&\lesssim
				|\bar\partial\partial f|\,|\partial g|
				+
				|\partial^2 f|\,|\bar\partial g|
				+
				\frac1r|\partial f|\,|\partial g|,
				\\
				\big|
				\sum_{\alpha,\beta,\gamma=0}^3
				N_2^{\alpha\beta\gamma}
				\partial_{\alpha}f\,\partial_{\beta}g\,\partial_{\gamma}h
				\big|
				&\lesssim
				|\bar\partial f|\,|\partial g|\,|\partial h|
				+
				|\partial f|\,|\bar\partial g|\,|\partial h|
				+
				|\partial f|\,|\partial g|\,|\bar\partial h|,
				\\
				\big|
				\sum_{\alpha,\beta,\gamma,\delta=0}^3
				N_3^{\alpha\beta\gamma\delta}
				\partial^2_{\alpha\beta}f\,\partial_{\gamma}g\,\partial_{\delta}h
				\big|
				&\lesssim
				|\bar\partial\partial f|\,|\partial g|\,|\partial h|
				+
				|\partial^2 f|\,|\bar\partial g|\,|\partial h|
				\\
				&\quad+
				|\partial^2 f|\,|\partial g|\,|\bar\partial h|
				+
				\frac1r|\partial f|\,|\partial g|\,|\partial h|,
				\\
				\big|
				\sum_{\alpha,\beta,\gamma,\delta=0}^3
				N_3^{\alpha\beta\gamma\delta}
				\partial_{\alpha}f\,\partial_{\beta}g\,\partial_{\gamma}h\,\partial_{\delta}w
				\big|
				&\lesssim
				|\bar\partial f|\,|\partial g|\,|\partial h|\,|\partial w|
				+
				|\partial f|\,|\bar\partial g|\,|\partial h|\,|\partial w|
				\\
				&\quad+
				|\partial f|\,|\partial g|\,|\bar\partial h|\,|\partial w|
				+
				|\partial f|\,|\partial g|\,|\partial h|\,|\bar\partial w|,
			\end{aligned}
		\end{equation}
		where
		$\bar\partial\in\big\{\partial_t+\partial_r,\,
			\frac1r\Omega_{12},\,
			\frac1r\Omega_{13},\,
			\frac1r\Omega_{23}
			\big\}$.
		\end{lemma}
	\begin{remark}
		Although Lemma~\ref{lem:null} is similar to \cite[Lemma 2.5]{HouTaoYin2},
the good derivatives $\bar\partial$ in $\Bbb R^{1+3}$
are different from the ones $\{\p_j+\f{x_j}{r}\p_t: j=1,2\}$ in \cite{HouTaoYin2} for 2D case.
We still give the detailed proof on \eqref{eq:null-structure-estimates}.
	\end{remark}
	\begin{proof}
		Set $\zeta=(\zeta_0,\zeta_1,\zeta_2,\zeta_3)=(-1,\omega)$.
		Define the tangential derivatives
		\begin{equation*}
			T_0=\partial_t+\partial_r,
			\qquad
			T_k=\partial_k-\omega_k\partial_r,
			\qquad 1\le k\le3.
		\end{equation*}
Note that
\begin{equation}\label{eq:T-angular-identity-null}
			T_k
			=
			\partial_k-\omega_k\partial_r
			=
			\frac1r\sum_{j=1}^3\omega_j\Omega_{jk}.
		\end{equation}
		Then
		\begin{equation}\label{eq:first-derivative-decomposition-null}
			\partial_\alpha f
			=
			\zeta_\alpha\partial_r f+T_\alpha f,
			\qquad
			\alpha=0,1,2,3,
		\end{equation}
		and
		\begin{equation}\label{eq:T-good-control-null}
			|T_\alpha f|
			\lesssim
			|\bar\partial f|.
		\end{equation}
		In addition, applying
		\eqref{eq:first-derivative-decomposition-null} to $\partial_\beta f$ yields
		\begin{equation}\label{eq:second-derivative-start-null}
			\partial_{\alpha\beta}^2 f
			=
			\partial_\alpha(\partial_\beta f)
			=
			\zeta_\alpha\partial_r(\partial_\beta f)
			+
			T_\alpha(\partial_\beta f).
		\end{equation}
		Due to $\partial_\beta f=\zeta_\beta\partial_r f+T_\beta f$ and $\partial_r\zeta_\beta=0$,
		one has
		\begin{equation}\label{eq:second-derivative-expanded-null}
			\begin{aligned}
				\partial_{\alpha\beta}^2 f
				&=
				\zeta_\alpha\zeta_\beta\partial_r^2f
				+
				\zeta_\alpha\partial_rT_\beta f
				+
				T_\alpha\partial_\beta f.
			\end{aligned}
		\end{equation}
		Write
		\begin{equation}\label{eq:second-derivative-decomposition-null}
			\partial_{\alpha\beta}^2 f
			=\zeta_\alpha\zeta_\beta\partial_r^2f
			+R_{\alpha\beta}(f),
		\end{equation}
		where $R_{\alpha\beta}(f)
			=\zeta_\alpha\partial_rT_\beta f
			+T_\alpha\partial_\beta f$.

		We now estimate $R_{\alpha\beta}(f)$. At first, it holds that
		\begin{equation}\label{eq:Ta-pbeta-control-null}
			|T_\alpha\partial_\beta f|
			\lesssim
			|\bar\partial\partial f|.
		\end{equation}
		It remains to deal with $\partial_rT_\beta f$.
		
		If $\beta=0$, then $T_0=\partial_t+\partial_r$ and
		\begin{equation}\label{eq:prT0-control-null}
			|\partial_rT_0f|
			=
			|(\partial_t+\partial_r)\partial_rf|
			\le
			\sum_{i=1}^{3}|(\partial_t+\partial_r)(\omega_i\p_i)f|
			\lesssim
			|\bar\partial\partial f|.
		\end{equation}
		
		If $1\le k\le3$, due to $\partial_r\omega_j=0$ and $[\Omega_{jk},\partial_r]=0$, it follows from \eqref{eq:T-angular-identity-null}
        that
		\begin{equation}\label{eq:prTi-expanded-null}
			\begin{aligned}
				\partial_rT_k f
				&=
				-\frac1{r^2}\sum_{j=1}^3\omega_j\Omega_{jk}f
				+
				\frac1r\sum_{j=1}^3\omega_j\Omega_{jk}\partial_r f
				\\
				&=
				-\frac1rT_k f
				+
				\frac1r\sum_{j=1}^3\omega_j\Omega_{jk}\partial_r f.
			\end{aligned}
		\end{equation}
		Note that the first term in the second line of \eqref{eq:prTi-expanded-null} is bounded by
		\begin{equation*}
			\frac1r|T_k f|
			\lesssim
			\frac1r|\partial f|.
		\end{equation*}
		For the second term, due to $\partial_r f=\ds\sum_{k=1}^3\omega_k\partial_k f$, one has
		\begin{equation*}
			\begin{aligned}
				\frac1r\Omega_{jk}\partial_r f
				&=
				\sum_{l=1}^3
				\frac{\omega_l}{r}\Omega_{jk}\partial_l f
				+\frac1r\sum_{l=1}^3
				\Omega_{jk}(\omega_l)\partial_l f.
			\end{aligned}
		\end{equation*}
		In addition, $\Omega_{jk}(\omega_l)$ is bounded. Therefore,
		\begin{equation}\label{eq:prTi-control-null}
			|\partial_rT_k f|
			\lesssim
			|\bar\partial\partial f|
			+
			\frac1r|\partial f|.
		\end{equation}
		Combining
		\eqref{eq:Ta-pbeta-control-null}-\eqref{eq:prTi-control-null} yields
		\begin{equation}\label{eq:Rab-control-null}
			|R_{\alpha\beta}(f)|
			\lesssim
			|\bar\partial\partial f|
			+
			\frac1r|\partial f|.
		\end{equation}
		
	For the quadratic null form, by utilizing \eqref{eq:first-derivative-decomposition-null}, one has
		\begin{equation}\label{YHC-16}
			\begin{aligned}
				\sum_{\alpha,\beta=0}^3
				N_1^{\alpha\beta}\partial_\alpha f\,\partial_\beta g
				=
				\sum_{\alpha,\beta=0}^3
				N_1^{\alpha\beta}
				(\zeta_\alpha\partial_r f+T_\alpha f)
				(\zeta_\beta\partial_r g+T_\beta g).
			\end{aligned}
		\end{equation}
		The bad term in \eqref{YHC-16} is
		$\big(
		\ds\sum_{\alpha,\beta=0}^3
		N_1^{\alpha\beta}\zeta_\alpha\zeta_\beta
		\big)
		\partial_r f\,\partial_r g$,
		which vanishes due to the null condition. Hence, all remaining terms in \eqref{YHC-16} contain at least
		one $T$ derivative. This, together with \eqref{eq:T-good-control-null}, yields
		\begin{equation*}
			\big|
			\sum_{\alpha,\beta=0}^3
			N_1^{\alpha\beta}\partial_\alpha f\,\partial_\beta g
			\big|
			\lesssim
			|\bar\partial f|\,|\partial g|
			+
			|\partial f|\,|\bar\partial g|.
		\end{equation*}
		
		For the second estimate in \eqref{eq:null-structure-estimates}, by
		\eqref{eq:first-derivative-decomposition-null} and
		\eqref{eq:second-derivative-decomposition-null}, we arrive at
		\begin{equation}\label{YHC-17}
			\begin{aligned}
				&\sum_{\alpha,\beta,\gamma=0}^3
				N_2^{\alpha\beta\gamma}
				\partial_{\alpha\beta}^2f\,\partial_\gamma g
				=
				\sum_{\alpha,\beta,\gamma=0}^3
				N_2^{\alpha\beta\gamma}
				\left(
				\zeta_\alpha\zeta_\beta\partial_r^2f
				+
				R_{\alpha\beta}(f)
				\right)
				\left(
				\zeta_\gamma\partial_r g+T_\gamma g
				\right).
			\end{aligned}
		\end{equation}
		The bad term in \eqref{YHC-17} is
		$\big(
		\sum_{\alpha,\beta,\gamma=0}^3
		N_2^{\alpha\beta\gamma}
		\zeta_\alpha\zeta_\beta\zeta_\gamma
		\big)
		\partial_r^2f\,\partial_rg$,
		which vanishes by the null condition. This leads to that all remaining terms in \eqref{YHC-17} contain
		either $R_{\alpha\beta}(f)$ or $T_\gamma g$. Together with
		\eqref{eq:T-good-control-null} and \eqref{eq:Rab-control-null}, this yields
		\begin{equation*}
			\begin{aligned}
				\big|
				\sum_{\alpha,\beta,\gamma=0}^3
				N_2^{\alpha\beta\gamma}
				\partial_{\alpha\beta}^2f\,\partial_\gamma g
				\big|
				&\lesssim
				\big(
				|\bar\partial\partial f|
				+
				\frac1r|\partial f|
				\big)
				|\partial g|
				+
				|\partial^2 f|\,|\bar\partial g|.
			\end{aligned}
		\end{equation*}
		The proof of the other left estimates in \eqref{eq:null-structure-estimates} can be analogously done.
	\end{proof}

	Next, we give some estimates of good derivatives by the vector field
$\Gamma\in\{L_j,\Omega_{kl}: 1\le j\le3, 1\le k<l\le3\}$.

	\begin{lemma}\label{lem:good-derivative-one-Gamma}
		It holds that
		\begin{equation}\label{eq:good-derivative-Gamma-correct}
			|\bar\partial f|
			\lesssim
			A^{-1}\sum_{k=1}^3 |L_k f|
			+A^{-1}\sum_{1\le k<j\le3}|\Omega_{kj}f|
			+\frac{B}{A}|\partial f|.
		\end{equation}
		In addition,
		\begin{equation}\label{eq:good-derivative-applied}
			|\bar\partial\nabla^a \partial u|
			\lesssim
			A^{-1}\sum_{\Gamma}|\Gamma\nabla^a \partial u|
			+\frac{B}{A}|\nabla^a\partial^2 u|.
		\end{equation}
		
	\end{lemma}
	
	\begin{proof}
		We first estimate $(\partial_t+\partial_r)f$. Set
		$L_r=\ds\sum_{k=1}^3\omega_k L_k=t\partial_r+r\partial_t$.
		A direct computation gives
		\begin{equation}\label{eq:Tr-identity}
			(t+r)(\partial_t+\partial_r)
			=
			2L_r+(t-r)(\partial_t-\partial_r).
		\end{equation}
		Therefore, when $t+r>0$,
		\begin{equation}\label{eq:outgoing-good-bound}
			\begin{aligned}
				|(\partial_t+\partial_r)f|
				&\lesssim
				(t+r)^{-1}\sum_{k=1}^3|L_k f|
				+
				\frac{|t-r|}{t+r}|\partial f|.
			\end{aligned}
		\end{equation}
		Note that $|(\partial_t+\partial_r)f|$ can be trivially controlled by
			$(B/A)|\partial f|$ for $t+r\le1$. Then \eqref{eq:outgoing-good-bound} implies
		\begin{equation}\label{eq:outgoing-good-bound-A}
			|(\partial_t+\partial_r)f|
			\lesssim
			A^{-1}\sum_{k=1}^3|L_k f|
			+
			\frac{B}{A}|\partial f|.
		\end{equation}
		
		It remains to estimate $\Omega_{kj}f$.
		If	$r\ge \frac{1+t}{4}$,
		then $A(t,x)\lesssim r$,
		and hence
		\begin{equation}\label{eq:angular-large-r}
			\frac1r\sum_{1\le k<j\le3}|\Omega_{kj}f|
			\lesssim
			A^{-1}\sum_{1\le k<j\le3}|\Omega_{kj}f|.
		\end{equation}

		If	$r<\frac{1+t}{4}$,
		then $B(t,x)\gtrsim 1+t$ and $A(t,x)\lesssim 1+t$.
		Thus, $\frac{B(t,x)}{A(t,x)}\gtrsim1$.
		Due to $\frac1r\ds\sum_{1\le k<j\le3}|\Omega_{kj}f|\ls |\nabla f|\le |\partial f|$,
		 one has
		\begin{equation}\label{eq:angular-small-r}
			\frac1r\sum_{1\le k<j\le3}|\Omega_{kj}f|
			\lesssim
			\frac{B}{A}|\partial f|.
		\end{equation}

		Combining \eqref{eq:outgoing-good-bound-A}, \eqref{eq:angular-large-r}
		and \eqref{eq:angular-small-r} yields
		\eqref{eq:good-derivative-Gamma-correct}.
Applying \eqref{eq:good-derivative-Gamma-correct} to
		$f=\partial\nabla^a u$, we can obtain \eqref{eq:good-derivative-applied} directly.
	\end{proof}

	\section{Bootstrap assumptions}
	\label{sec:CVP-bootstrap-assumptions}

	Set
		\begin{equation}
		\label{eq:CVP-bootstrap-parameter-range}
		N \ge 15,
		\qquad
		0<\mu<\frac12,
		\qquad
		0\le\eta\le\frac{\mu}{100}.
	\end{equation}
	For problem \eqref{eq:Chaplygin-Cauchy} and any fixed $T>0$, we impose the following bootstrap assumptions:

(A)  Energy estimates
	\begin{equation}
		\label{eq:CVP-bootstrap-energy-ordinary}
		\sup_{0\le t\le T}
		(1+t)^{-\eta}
		\|\p\phi(t)\|_{H^N(\R^3)}
		\le
		\ve_1,
	\end{equation}
	\begin{equation}
		\label{eq:CVP-bootstrap-energy-Gamma}
		\sup_{0\le t\le T}
		(1+t)^{-\eta}
		\sum_{\Gamma}
		\|\p\Gamma\phi(t)\|_{H^{11}(\R^3)}
		\le
		\ve_1.
	\end{equation}

(B) Pointwise estimates:
	\begin{equation}\label{eq:CVP-bootstrap-pointwise-phi}
		\big\|
		AB^{\f\mu4}\phi
		\big\|_{L^\infty([0,T]\times\R^3)}
		+
		\sum_{|a|\le 7}
		\big\|
		AB^{\f\mu4}\nabla^a\p\phi
		\big\|_{L^\infty([0,T]\times\R^3)}
		\le
		\ve_1,
	\end{equation}
	\begin{equation}\label{eq:CVP-bootstrap-pointwise-Gamma}
		\sum_{\Gamma}
		\big\|
		A^{\frac{\mu}{2}}
		\Gamma\phi
		\big\|_{L^\infty([0,T]\times\R^3)}
		+\sum_{\Gamma}\sum_{|a|\le 6}\big\|A^{\frac{\mu}{2}}\p\nabla^a\Gamma\phi\big\|_{L^\infty([0,T]\times\R^3)}
		\le
		\ve_1.
	\end{equation}

(C) Weighted  Strichartz-type estimate:
	\begin{equation}
		\label{eq:CVP-bootstrap-Strichartz-Gamma}
		\sum_{\Gamma}
		\big\|
		A^{\frac{\mu}{3}}
		\Gamma\phi
		\big\|_{L^2([0,T];L^\infty(\mathbb R^3))}
		+
		\sum_{\Gamma}
		\sum_{|a|\le 6}
		\big\|
		A^{\frac{\mu}{3}}
		\partial\nabla^a\Gamma\phi
		\big\|_{L^2([0,T];L^\infty(\mathbb R^3))}
		\le
		\varepsilon_1.
	\end{equation}

	From the bootstrap assumptions above, we start to derive some related estimates.

	\begin{lemma}[{\bf Estimates for higher time derivatives}]
		\label{lem:CVP-higher-time-derivative-estimates}
		Let \(\phi\) be a solution to \eqref{eq:Chaplygin-Cauchy}
		on \([0,T]\times\mathbb R^3\).
		Assume that
		\eqref{eq:CVP-bootstrap-energy-ordinary}-\eqref{eq:CVP-bootstrap-Strichartz-Gamma} hold and
		$\ve_1>0$ is small. Then it holds that
		\begin{equation}
			\label{eq:CVP-ttphi-pointwise-zero}
			\sum_{|b|\le 6}\|AB^{\f\mu4}\nabla^b\p_t^2\phi\|_{L^\infty_{t,x}}
			\ls
			\ve_1,
		\end{equation}
		\begin{equation}
			\label{eq:CVP-ttphi-low-pointwise}
			\sum_{|b|\le 5}\|AB^{\f\mu4}\nabla^b\p_t^3\phi\|_{L^\infty_{t,x}}
			\ls
			\ve_1,
		\end{equation}
		\begin{equation}
			\label{eq:CVP-ttphi-low-pointwise-2}
			\sum_{|b|\le 4}\|AB^{\f\mu4}\nabla^b\p_t^4\phi\|_{L^\infty_{t,x}}
			\ls
			\ve_1,
		\end{equation}
		\begin{equation}
			\label{eq:CVP-ttphi-high-L2}
			\begin{aligned}
				\|\p_t^2\phi(t)\|_{H^{N-1}_x}
				+
				\|\p_t^3\phi(t)\|_{H^{N-2}_x}
				\ls
				\|\p\phi(t)\|_{H^N_x}.
			\end{aligned}
		\end{equation}
	where the implicit constants in \eqref{eq:CVP-ttphi-pointwise-zero}-\eqref{eq:CVP-ttphi-high-L2} are independent of $T$.
	\end{lemma}
	\begin{proof}
    Although the argument is straightforward, we still provide the details for completeness.
		\eqref{potentia1} can be written as
		\begin{equation}
			\label{eq:CVP-ttphi-coefficient-form}
			\p_t^2\phi
			=
			a^{kj}(\p\phi)\p_{kj}^2\phi
			+
			b^k(\p\phi)\p_{tk}^2\phi,
		\end{equation}
		where and below the repeated spatial indices are summed over $1,2,3$, and
		\begin{equation}
			\label{eq:CVP-coefficient-a}
			a^{kj}(\p\phi)
			=
			\left(
			1+2\p_t\phi+|\nabla\phi|^2
			\right)\delta_{kj}
			-\p_k\phi\p_j\phi,\quad
			b^k(\p\phi)
			=-2\p_k\phi.
		\end{equation}
		
		By \eqref{eq:CVP-bootstrap-pointwise-phi}, one has
		\begin{equation}
			\label{eq:CVP-basic-pointwise-from-bootstrap}
			|\nabla^c\p\phi|
			\ls
			\ve_1A^{-1}B^{-{\f\mu4}},
			\qquad
			|c|\le 7.
		\end{equation}

		We first prove \eqref{eq:CVP-ttphi-pointwise-zero}. Applying
		$\nabla^b$ to \eqref{eq:CVP-ttphi-coefficient-form} gives
		\begin{equation}\label{eq:CVP-ttphi-spatial-Leibniz}
			\begin{aligned}
				\nabla^b\p_t^2\phi
				&=
				\sum_{b_1+b_2=b}
				C_{b_1,b_2}
				\big[
				\nabla^{b_1}a^{kj}(\p\phi)
				\nabla^{b_2}\p_{kj}^2\phi
				+
				\nabla^{b_1}b^k(\p\phi)
				\nabla^{b_2}\p_{tk}^2\phi
				\big].
			\end{aligned}
		\end{equation}
		The terms with $b_1=0$ and $b_2=b$ in \eqref{eq:CVP-ttphi-spatial-Leibniz} are bounded by
		\begin{equation}
			\label{eq:CVP-ttphi-principal-pointwise}
			\begin{aligned}
				&
				|a^{kj}(\p\phi)|
				|\nabla^b\p_{kj}^2\phi|
				+
				|b^k(\p\phi)|
				|\nabla^b\p_{tk}^2\phi|
				\ls
				(1+|\p\phi|+|\p\phi|^2)\sum_{|c|\le |b|+1}
				|\nabla^c\p\phi|.
			\end{aligned}
		\end{equation}
		For $b_1\neq0$, such terms in \eqref{eq:CVP-ttphi-spatial-Leibniz} contain at least one order derivatives of $a^{kj}(\p\phi)$
        or $b^k(\p\phi)$. By \eqref{eq:CVP-coefficient-a}, these terms are bounded by the finite sums of
				$|\nabla^{c_1}\p\phi|
			|\nabla^{c_2}\p\phi|$
		or
		$|\nabla^{c_1}\p\phi||\nabla^{c_2}\p\phi||\nabla^{c_3}\p\phi|$
		with
		$
			|c_1|+|c_2|
			\le
			|b|+1
		$
		or
		$
			|c_1|+|c_2|+|c_3|
			\le
			|b|+1.
		$
		Using \eqref{eq:CVP-basic-pointwise-from-bootstrap} and
		$\ve_1\le1$, we have
		\begin{equation}
			\label{eq:CVP-ttphi-pointwise-intermediate}
			\begin{aligned}
				\sum_{|b|\le 6}|\nabla^b\p_t^2\phi|
				&\ls
				\ve_1A^{-1}B^{-\f\mu4}
				+
				\ve_1^2A^{-2}B^{-\f\mu2}
				+
				\ve_1^3A^{-3}B^{-\f{3\mu}4}
				\ls
				\ve_1A^{-1}B^{-\f\mu4},
			\end{aligned}
		\end{equation}
		which implies \eqref{eq:CVP-ttphi-pointwise-zero}.
		
		We next prove \eqref{eq:CVP-ttphi-low-pointwise}. Differentiating
		\eqref{eq:CVP-ttphi-coefficient-form} with respect to $t$ yields
		\begin{equation}
			\label{eq:CVP-tttphi-coefficient-form}
			\begin{aligned}
				\p_t^3\phi
				&=
				a^{kj}(\p\phi)\p_{tkj}^3\phi
				+
				b^k(\p\phi)\p_k\p_t^2\phi
				+
				\p_ta^{kj}(\p\phi)\p_{kj}^2\phi
				+
				\p_tb^k(\p\phi)\p_{tk}^2\phi.
			\end{aligned}
		\end{equation}
		Then it follows from
		\eqref{eq:CVP-basic-pointwise-from-bootstrap} and
		\eqref{eq:CVP-ttphi-pointwise-zero} that
		\begin{equation}
			\label{eq:CVP-tttphi-pointwise}
			\sum_{|b|\le 5}|\nabla^b\p_t^3\phi|
			\ls
			\ve_1A^{-1}B^{-\f\mu4},
		\end{equation}
		which proves
		\eqref{eq:CVP-ttphi-low-pointwise}.
		
		The estimate \eqref{eq:CVP-ttphi-low-pointwise-2} can be analogously obtained by \eqref{eq:CVP-bootstrap-pointwise-phi}, \eqref{eq:CVP-ttphi-pointwise-zero} and \eqref{eq:CVP-ttphi-low-pointwise}.
		
		Finally, we prove \eqref{eq:CVP-ttphi-high-L2}.
		By \eqref{eq:CVP-ttphi-coefficient-form} and
		$\|\p\phi(t)\|_{W^{1,\infty}_x}\ls\ve_1,$
		we arrive at
		\begin{equation}
			\label{eq:CVP-ttphi-L2-zero}
			\begin{aligned}
				\|\p_t^2\phi(t)\|_{H^{N-1}_x}
				&\ls
				(1+\|\p\phi\|_{L^\infty_x}+\|\p\phi\|_{L^\infty_x}^2)\|\nabla\p\phi\|_{H^{N-1}_x}
				\\
				&\quad+(1+\|\p\phi\|_{L^\infty_x})\|\p\phi\|_{H^{N-1}_x}\|\nabla\p\phi\|_{L^\infty_x}
				\\
				&\ls
				\|\p\phi(t)\|_{H^N_x}.
			\end{aligned}
		\end{equation}
		
		For $\|\p_t^3\phi(t)\|_{H^{N-2}_x}$,
		it follows from
		\eqref{eq:CVP-tttphi-coefficient-form}, \eqref{eq:CVP-bootstrap-pointwise-phi},
		\eqref{eq:CVP-ttphi-pointwise-zero},
		\eqref{eq:CVP-ttphi-low-pointwise} and \eqref{eq:CVP-ttphi-L2-zero} that
		\begin{equation}
			\label{eq:CVP-tttphi-L2}
			\begin{aligned}
				\|\p_t^3\phi(t)\|_{H^{N-2}_x}
				&\ls
				(1+\|\p\phi\|_{L^\infty_x}+\|\p\phi\|_{L^\infty_x}^2)\|\nabla\p^2\phi\|_{H^{N-2}_x}
				\\
				&\quad+(1+\|\p\phi\|_{L^\infty_x})\|\p\phi\|_{H^{N-2}_x}\|\nabla\p^2\phi\|_{L^\infty_x}
				\\
				&\quad+(\|\p^2\phi\|_{L^\infty_x}+\|\p\phi\|_{L^\infty_x}\|\p^2\phi\|_{L^\infty_x})\|\nabla\p\phi\|_{H^{N-2}_x}
				\\
				&\quad+(1+\|\p\phi\|_{L^\infty_x}+\|\p^2\phi\|_{L^\infty_x})(\|\p\phi\|_{H^{N-2}_x}+\|\p^2\phi\|_{H^{N-2}_x})\|\nabla\p\phi\|_{L^\infty_x}
				\\
				&\ls
				\|\p\phi(t)\|_{H^N_x}.
			\end{aligned}
		\end{equation}
		This proves
		\eqref{eq:CVP-ttphi-high-L2}.
	\end{proof}

		\begin{lemma}[{\bf Estimates for higher time derivatives with one vector field $\Gamma$}]
		\label{lem:CVP-higher-time-derivative-estimates-Gamma}
		Let \(\phi\) be a smooth solution of \eqref{eq:Chaplygin-Cauchy}
		on \([0,T]\times\mathbb R^3\). Assume that
		\eqref{eq:CVP-bootstrap-energy-ordinary}-\eqref{eq:CVP-bootstrap-Strichartz-Gamma} hold and
		\(\ve_1>0\) is sufficiently small. Then the following estimates hold:
		\begin{equation}
			\label{eq:CVP-Gamma-ttphi-pointwise-zero}
			\sum_{\Gamma}
			\sum_{|b|\le 5}
			\big\|
			A^{\frac{\mu}{2}}
			\nabla^b\p_t^2\Gamma\phi
			\big\|_{L^\infty_{t,x}}
			\ls
			\ve_1,
		\end{equation}
		\begin{equation}
			\label{eq:CVP-Gamma-ttphi-low-pointwise}
			\sum_{\Gamma}
			\sum_{|b|\le 4}
			\big\|
			A^{\frac{\mu}{2}}
			\nabla^b\p_t^3\Gamma\phi
			\big\|_{L^\infty_{t,x}}
			\ls
			\ve_1,
		\end{equation}
		\begin{equation}
			\label{eq:CVP-Gamma-ttphi-high-L2}
			\begin{aligned}
				&
				\sum_{\Gamma}
				\big(
				\|\p_t^2\Gamma\phi(t)\|_{H^{10}_x}
				+
				\|\p_t^3\Gamma\phi(t)\|_{H^{9}_x}
				\big)
				\ls
				\sum_{\Gamma}
				\|\p\Gamma\phi(t)\|_{H^{11}_x}
				+
				\|\p\phi(t)\|_{H^N_x},
			\end{aligned}
		\end{equation}
		where the implicit constants in \eqref{eq:CVP-Gamma-ttphi-pointwise-zero}-\eqref{eq:CVP-Gamma-ttphi-high-L2} are independent of \(T\).
	\end{lemma}

	\begin{proof}
		The proof of Lemma \ref{lem:CVP-higher-time-derivative-estimates-Gamma}
			is similar to that of Lemma
			\ref{lem:CVP-higher-time-derivative-estimates}.		
		Set $\psi=\Gamma\phi.$
		It follows from \eqref{eq:CVP-ttphi-coefficient-form} that
		\begin{equation}
			\label{eq:CVP-Gamma-coefficient-form}
			\p_t^2\psi
			=a^{kj}(\p\phi)\p_{kj}^2\psi
			+b^k(\p\phi)\p_{tk}^2\psi
			+\mathcal C_\Gamma,
		\end{equation}
		where \(a^{kj}\) and \(b^k\) are defined by
		\eqref{eq:CVP-coefficient-a}, and
		\begin{equation}
			\label{eq:CVP-CGamma-schematic}
			\begin{aligned}
				\mathcal C_\Gamma
				&=\Gamma\bigl(a^{kj}(\p\phi)\bigr)\p_{kj}^2\phi
				+\Gamma\bigl(b^k(\p\phi)\bigr)\p_{tk}^2\phi
				\\
				&\quad+a^{kj}(\p\phi)[\Gamma,\p_{kj}^2]\phi
				+b^k(\p\phi)[\Gamma,\p_{tk}^2]\phi
				+[\p_t^2,\Gamma]\phi.
			\end{aligned}
		\end{equation}
		Then one has
		\begin{equation}\label{eq:CVP-CGamma-estimate}
			\begin{aligned}
				|\mathcal C_\Gamma|
				\ls&\left(
				|\Gamma\p\phi|+|\Gamma\p\phi||\p\phi|
				\right)|\p^2\phi|\\
				&
				+\left(1+|\p\phi|+|\p\phi|^2
				\right)|\p^2\phi|\\
				\ls&\left(1+|\p\phi|+|\p\Gamma\phi|+|\p\phi|^2+|\p\Gamma\phi||\p\phi|
				\right)|\p^2\phi|.
			\end{aligned}
		\end{equation}
		
		By \eqref{eq:CVP-bootstrap-pointwise-phi} and
		Lemma~\ref{lem:CVP-higher-time-derivative-estimates}, we have
		\begin{equation}
			\label{eq:CVP-basic-ordinary-time-pointwise-for-Gamma}
			\sum_{|c|\le 7}
			AB^{\f\mu4}|\nabla^c\p\phi|
			+
			\sum_{|c|\le 6}
			AB^{\f\mu4}|\nabla^c\p^2\phi|
			+
			\sum_{|c|\le 5}
			AB^{\f\mu4}|\nabla^c\p^3\phi|
			\ls
			\ve_1.
		\end{equation}
		In addition, it follows from \eqref{eq:CVP-bootstrap-pointwise-Gamma} that
		\begin{equation}
			\label{eq:CVP-basic-Gamma-pointwise-from-bootstrap}
			\sum_{\Gamma}
			\big(
			A^{\frac{\mu}{2}}|\Gamma\phi|
			+
			\sum_{|c|\le 6}
			A^{\frac{\mu}{2}}
			|\p\nabla^c\Gamma\phi|
			\big)
			\ls
			\ve_1.
		\end{equation}
		
		By \eqref{eq:CVP-CGamma-schematic} and direct computation, similarly to \eqref{eq:CVP-CGamma-estimate}, one can arrive at
			\begin{equation}
			\label{eq:CVP-CGamma-pointwise-schematic}
			\begin{aligned}
				|\nabla^b\mathcal C_\Gamma|
				&\ls
				\sum_{|c|\le |b|}
				|\nabla^c\p^2\phi|
				\\
				&\quad+
				\sum_{|c_1|+|c_2|\le |b|}
				\left(
				|\nabla^{c_1}\p\phi|
				+
				|\nabla^{c_1}\p\Gamma\phi|
				\right)
				|\nabla^{c_2}\p^2\phi|
				\\
				&\quad+
				\sum_{|c_1|+|c_2|+|c_3|\le |b|}
				|\nabla^{c_1}\p\phi|
				\left(
				|\nabla^{c_2}\p\phi|
				+
				|\nabla^{c_2}\p\Gamma\phi|
				\right)
				|\nabla^{c_3}\p^2\phi|.
			\end{aligned}
		\end{equation}
		For \(|b|\le 5\), collecting
		\eqref{eq:CVP-basic-ordinary-time-pointwise-for-Gamma}-\eqref{eq:CVP-CGamma-pointwise-schematic}
		yields
		\begin{equation}
			\label{eq:CVP-CGamma-pointwise-final}
			A^{\frac{\mu}{2}}
			|\nabla^b\mathcal C_\Gamma|
			\ls
			\ve_1+\ve_1^2+\ve_1^3
			\ls
			\ve_1.
		\end{equation}
		
		We next prove \eqref{eq:CVP-Gamma-ttphi-pointwise-zero}. Acting
		\(\nabla^b\) on \eqref{eq:CVP-Gamma-coefficient-form} (\(|b|\le 5\)), one has
		\begin{equation}
			\label{eq:CVP-Gamma-ttphi-Leibniz}
			\begin{aligned}
				\nabla^b\p_t^2\psi
				=\sum_{b_1+b_2=b}
				C_{b_1,b_2}
				\big[
				\nabla^{b_1}a^{kj}(\p\phi)
				\nabla^{b_2}\p_{kj}^2\psi
				+
				\nabla^{b_1}b^k(\p\phi)
				\nabla^{b_2}\p_{tk}^2\psi
				\big]
				+\nabla^b\mathcal C_\Gamma.
			\end{aligned}
		\end{equation}
	Note that \(|b|+1\le 6\), \eqref{eq:CVP-basic-Gamma-pointwise-from-bootstrap} and
		\begin{equation}
			\label{eq:CVP-Gamma-principal-pointwise}
			\begin{aligned}
				&A^{\frac{\mu}{2}}\sum_{b_1+b_2=b}
				\big[
				|\nabla^{b_1}a^{kj}(\p\phi)|
				|\nabla^{b_2}\p_{kj}^2\psi|
				+
				|\nabla^{b_1}b^k(\p\phi)|
				|\nabla^{b_2}\p_{tk}^2\psi|
				\big]
				\\
				\ls&
				\big[
				1+\sum_{|d|\le |b|}|\nabla^d\p\phi|+\Big(\sum_{|d|\le |b|}|\nabla^d\p\phi|\Big)^2
				\big]
				\sum_{|c|\le |b|+1}	A^{\frac{\mu}{2}}
				|\p\nabla^c\Gamma\phi|
				\ls
				\ve_1.
			\end{aligned}
		\end{equation}
	Collecting 	\eqref{eq:CVP-CGamma-pointwise-final} and \eqref{eq:CVP-Gamma-principal-pointwise}
		yields  \eqref{eq:CVP-Gamma-ttphi-pointwise-zero}.
		
		We next prove \eqref{eq:CVP-Gamma-ttphi-low-pointwise}. Differentiating
		\eqref{eq:CVP-Gamma-coefficient-form} with respect to \(t\), we get
		\begin{equation}
			\label{eq:CVP-Gamma-tttphi-coefficient-form}
			\begin{aligned}
				\p_t^3\psi
				&=
				a^{kj}(\p\phi)\p_{tkj}^3\psi
				+
				b^k(\p\phi)\p_i\p_t^2\psi
				+
				\p_ta^{kj}(\p\phi)\p_{kj}^2\psi
				+
				\p_tb^k(\p\phi)\p_{tk}^2\psi
				+
				\p_t\mathcal C_\Gamma.
			\end{aligned}
		\end{equation}
		For \(|b|\le 4\),  it is easy to know
		\begin{equation}\label{YHC-18}
		\begin{aligned}
		|\nabla^b(a^{kj}(\p\phi)\p_{tkj}^3\psi)|\ls\sum_{|c|\le |b|+2}
		|\p\nabla^c\Gamma\phi|\ls\ve_1.
\end{aligned}
		\end{equation}		
   In addition, it holds that due to \(|b|+1\le 5\) and \eqref{eq:CVP-Gamma-ttphi-pointwise-zero},
   \begin{equation}\label{YHC-19}
		\begin{aligned}
		|\nabla^b(b^k(\p\phi)\p_k\p_t^2\psi)|\ls
		\sum_{|d|\le |b|}|\nabla^d\p\phi|\cdot
		\sum_{|c|\le |b|+1}
		|\nabla^c\p_t^2\Gamma\phi|\ls
			\ve_1.
		\end{aligned}
		\end{equation}

		Note that
		\begin{equation}
			\label{eq:CVP-Gamma-time-derivative-coefficients}
			\begin{aligned}
				|\p_ta^{kj}(\p\phi)|
				+|\p_tb^k(\p\phi)|
				\ls|\p^2\phi|
				+|\p\phi||\p^2\phi|.
			\end{aligned}
		\end{equation}
		Then it follows from \eqref{eq:CVP-basic-ordinary-time-pointwise-for-Gamma},
		\eqref{eq:CVP-basic-Gamma-pointwise-from-bootstrap} and
		\eqref{eq:CVP-Gamma-ttphi-pointwise-zero} that the third and fourth terms in
		\eqref{eq:CVP-Gamma-tttphi-coefficient-form} are estimated as
\begin{equation}\label{YHC-20}
			\begin{aligned}
			|\nabla^b(\p_ta^{kj}(\p\phi)\p_{kj}^2\psi
				+\p_tb^k(\p\phi)\p_{tk}^2\psi)|\ls
			\ve_1.
			\end{aligned}
		\end{equation}

		It remains to treat \(\nabla^b\p_t\mathcal C_\Gamma\). By
		\eqref{eq:CVP-CGamma-schematic},
		\eqref{eq:CVP-Gamma-ttphi-pointwise-zero} and
		\eqref{eq:CVP-basic-Gamma-pointwise-from-bootstrap},
		we obtain
		\begin{equation}\label{eq:CVP-dtCGamma-pointwise-final}
			A^{\frac{\mu}{2}}
			|\nabla^b\p_t\mathcal C_\Gamma|
			\ls
			\ve_1+\ve_1^2+\ve_1^3
			\ls
			\ve_1.
		\end{equation}
		
		Combining \eqref{YHC-18}-\eqref{eq:CVP-dtCGamma-pointwise-final} yields
			\begin{equation}\label{eq:CVP-Gamma-tttphi-pointwise}
			\sum_{\Gamma}
			\sum_{|b|\le 4}
			A^{\frac{\mu}{2}}
			|\nabla^b\p_t^3\Gamma\phi|
			\ls
			\ve_1,
		\end{equation}
		which implies \eqref{eq:CVP-Gamma-ttphi-low-pointwise}.
		
		Meanwhile,  \eqref{eq:CVP-Gamma-ttphi-high-L2} can be shown as for \eqref{eq:CVP-ttphi-high-L2}. We omit
the details here.
	\end{proof}

Finally, we establish the weighted estimates for the good derivatives of $\phi$.

	\begin{lemma}[{\bf Weighted estimates for the good derivatives}]
		\label{lem:CVP-derived-good-derivative-estimates}
		Let \(\phi\) be a solution to \eqref{eq:Chaplygin-Cauchy}
		on \([0,T]\times\mathbb R^3\).
		Assume that
		\eqref{eq:CVP-bootstrap-energy-ordinary}-\eqref{eq:CVP-bootstrap-Strichartz-Gamma} hold and
		$\ve_1>0$ is small. Then
		\begin{equation}
			\label{eq:CVP-derived-good-derivative-estimates}
			\begin{aligned}
				|\bar{\p}\phi(t,x)|
				+\sum_{|a|\le 6}|\bar{\p}\p\nabla^a\phi(t,x)|
				\ls
				\ve_1\big[A^{-1-\f\mu2}
				+A^{-2}B^{1-\f\mu4}\big](t,x)
			\end{aligned}
		\end{equation}
		and
		\begin{equation}
			\label{eq:CVP-ttphi-low-good}
			\sum_{|b|\le 5}
			|\bar{\p}\nabla^b\p_t^2\phi(t,x)|
			\ls
			\ve_1\big[A^{-1-\f\mu2}
			+A^{-2}
			B^{1-\f\mu4}\big](t,x).
		\end{equation}
	\end{lemma}
	
	\begin{proof}
		By
		\eqref{eq:CVP-bootstrap-pointwise-phi},
		\eqref{eq:CVP-bootstrap-pointwise-Gamma} and
		Lemma~\ref{lem:CVP-higher-time-derivative-estimates}, we have
		\begin{equation}
			\label{eq:CVP-derived-basic-second-order-bounds}
			\sum_{|c|\le 5}
			AB^{\f\mu4}
			|\nabla^c\p^{\le 3}\phi|
			\ls
			\ve_1,
		\end{equation}
		\begin{equation}
			\label{eq:CVP-derived-basic-Gamma-bounds}
			\sum_{\Gamma}
			A^{\frac{\mu}{2}}
			|\Gamma\phi|
			+
			\sum_{\Gamma}
			\sum_{|c|\le 6}
			A^{\frac{\mu}{2}}
			|\p\nabla^c\Gamma\phi|
			\ls
			\ve_1.
		\end{equation}
		By Lemma \ref{lem:good-derivative-one-Gamma}, it holds that
		\begin{equation}\label{eq:CVP-good-relation-zero-order-proof}
			|\bar{\p}\phi|
			\ls
			A^{-1}\sum_{\Gamma}
			|\Gamma\phi|
			+A^{-1}B|\p\phi|.
		\end{equation}
		By \eqref{eq:CVP-derived-basic-Gamma-bounds}, the first term on the right-hand side of \eqref{eq:CVP-good-relation-zero-order-proof}
is bounded by
\begin{equation}\label{YHC-22}
A^{-1}\sum_{\Gamma}|\Gamma\phi|\ls\ve_1A^{-1-\f\mu2}.
\end{equation}
		 For the second term on the right-hand side of \eqref{eq:CVP-good-relation-zero-order-proof}, due to
		\eqref{eq:CVP-derived-basic-second-order-bounds} and \(0<\mu<1\), we have
		\begin{equation}
			\label{eq:CVP-weighted-good-B-zero-order-proof}
			\begin{aligned}
				A^{-1}B|\p\phi|
				&=A^{-2}B^{1-\f\mu4}
				\big(AB^{\f\mu4}|\p\phi|\big)
				\ls\ve_1 A^{-2}	B^{1-\f\mu4}.
			\end{aligned}
		\end{equation}
		Thus, it holds that
		\begin{equation}\label{eq:CVP-derived-zero-order-good-final}
			|\bar{\p}\phi|
			\ls
			\ve_1\big[A^{-1-\f\mu2}
			+A^{-2}	B^{1-\f\mu4}\big].
		\end{equation}
		
		We next treat
		\(\bar{\p}\p\nabla^a\phi\) with
		\(|a|\le 6\). Applying Lemma  \ref{lem:good-derivative-one-Gamma} to
		each component of \(\p\nabla^a\phi\) yields
		\begin{equation}
			\label{eq:CVP-good-relation-differentiated-proof}
			\begin{aligned}
				|\bar{\p}\p\nabla^a\phi|
				&\ls
				A^{-1}\sum_{\Gamma}
				|\Gamma\p\nabla^a\phi|
				+
				A^{-1}B|\p^2\nabla^a\phi|.
			\end{aligned}
		\end{equation}
		Due to $[\Gamma,\p_\alpha]\in \mathrm{span}\{\p_0,\p_1,\p_2,\p_3\}$, one has
		\begin{equation}
			\label{eq:CVP-Gamma-partial-commutator-proof}
			\begin{aligned}
				|\Gamma\p\nabla^a\phi|
				&\ls
				|\p\nabla^a\Gamma\phi|
				+
				\sum_{|c|= |a|}
				|\p\nabla^c\phi|
				+
				\sum_{|c|= |a|-1}
				|\p^2\nabla^c\phi|,
			\end{aligned}
		\end{equation}
		where the last sum is absent when \(|a|=0\).

		Combining
		\eqref{eq:CVP-good-relation-differentiated-proof} and
		\eqref{eq:CVP-Gamma-partial-commutator-proof} yields
		\begin{equation}
			\label{eq:CVP-good-relation-differentiated-weighted-proof}
			\begin{aligned}
				|\bar{\p}\p\nabla^a\phi|
				&\ls
				\sum_{\Gamma}
				A^{-1}
				|\p\nabla^a\Gamma\phi|
				+
				A^{-1}
				\sum_{|c| = |a|}
				|\p\nabla^c\phi|
				+
				A^{-1}
				\sum_{|c| = |a|-1}
				|\p^2\nabla^c\phi|
				+
				A^{-1}B
				|\p^2\nabla^a\phi|.
			\end{aligned}
		\end{equation}
		Therefore, by  \eqref{eq:CVP-derived-basic-Gamma-bounds} and
		\eqref{eq:CVP-derived-basic-second-order-bounds}, we obtain
		\begin{equation}
			\label{eq:CVP-differentiated-good-final-proof}
			\sum_{|a|\le 6}
			|\bar{\p}\p\nabla^a\phi|
			\ls
			\ve_1\big[A^{-1-\f\mu2}
			+A^{-2}	B^{1-\f\mu4}\big].
		\end{equation}
		Combining
		\eqref{eq:CVP-derived-zero-order-good-final} and
		\eqref{eq:CVP-differentiated-good-final-proof} derives
		\eqref{eq:CVP-derived-good-derivative-estimates}.
		
		It remains to prove
		\eqref{eq:CVP-ttphi-low-good}. Let \(|b|\le 5\). Applying Lemma \ref{lem:good-derivative-one-Gamma} to \(\nabla^b\p_t^2\phi\)
yields
		\begin{equation}
			\label{eq:CVP-good-ttphi-vector-relation}
			\begin{aligned}
				|\bar{\p}\nabla^b\p_t^2\phi|
				\ls
				\sum_{\Gamma}
				A^{-1}|\Gamma\nabla^b\p_t^2\phi|
				+
				A^{-1}B|\p\nabla^b\p_t^2\phi|.
			\end{aligned}
		\end{equation}
		By \eqref{eq:CVP-derived-basic-second-order-bounds}, the second term on the right-hand side of \eqref{eq:CVP-good-ttphi-vector-relation}
		is estimated  as
		\begin{equation}
			\label{eq:CVP-good-ttphi-B-term}
			\begin{aligned}
				A^{-1}B|\p\nabla^b\p_t^2\phi|
				&=A^{-2}B^{1-\f\mu4}
				\big(AB^{\f\mu4}|\p\nabla^b\p_t^2\phi|\big)
				\ls
				\ve_1A^{-2}B^{1-\f\mu4}.
			\end{aligned}
		\end{equation}
		Next, we estimate the first term in
		\eqref{eq:CVP-good-ttphi-vector-relation}. By
		$[\Gamma,\p_\alpha]\in
		\operatorname{span}\{\p_0,\p_1,\p_2,\p_3\}$,
		one has
		\begin{equation}
			\label{eq:CVP-Gamma-tt-commutator}
			\begin{aligned}
				|\Gamma\nabla^b\p_t^2\phi|
				\ls
				|\nabla^b\p_t^2\Gamma\phi|
				+
				\sum_{|c|\le |b|+1}
				|\nabla^c\p\phi|
				+
				\sum_{|c|\le |b|}
				|\nabla^c\p_t^2\phi|
				+
				\sum_{|c|\le |b|-1}
				|\nabla^c\p_t^3\phi|,
			\end{aligned}
		\end{equation}
		where the last sum is absent when \(|b|=0\).
		
		Due to \(|b|\le 5\), Lemma~\ref{lem:CVP-higher-time-derivative-estimates-Gamma}
		gives
		\begin{equation}
			\label{eq:CVP-Gamma-main-tt-weighted}
			\sum_{\Gamma}
			|\nabla^b\p_t^2\Gamma\phi|
			\ls
			\ve_1A^{-\frac{\mu}{2}}.
		\end{equation}
		In addition, by \eqref{eq:CVP-derived-basic-second-order-bounds},
		\begin{equation}
			\label{eq:CVP-Gamma-commutator-ordinary-weighted}
			\begin{aligned}
				\sum_{|c|\le |b|+1}
				|\nabla^c\p\phi|
				+
				\sum_{|c|\le |b|}
				|\nabla^c\p_t^2\phi|
				+
				\sum_{|c|\le |b|-1}
				|\nabla^c\p_t^3\phi|
				\ls
				\ve_1A^{-1}B^{-\f\mu 4}.
			\end{aligned}
		\end{equation}
		Combining
		\eqref{eq:CVP-good-ttphi-vector-relation}-\eqref{eq:CVP-Gamma-commutator-ordinary-weighted} yields \eqref{eq:CVP-ttphi-low-good}.
		
	\end{proof}

	\section{Energy estimate}\label{sec:Chaplygin-energy-estimates}

	We rewrite 	\eqref{eq:Chaplygin-Cauchy}  as
	\begin{equation}\label{eq:Chaplygin-energy-reformulation}
		\Box\phi
		=\sum_{\alpha,\beta=0}^3
		\cQ^{\alpha\beta}(\partial\phi)
		\partial_{\alpha\beta}^2\phi,
	\end{equation}
	where $\cQ^{00}=0$, $\cQ^{0j}=\cQ^{j0}
		=-\partial_j\phi$ for $1\le j\le3$,
	and $\cQ^{kj}
		=\left(
		2\partial_t\phi+|\nabla\phi|^2
		\right)\delta_{kj}
		-\partial_k\phi\,\partial_j\phi$ for
		$1\le k,j\le3$.	
	Meanwhile, it holds that
	\begin{equation}\label{eq:Chaplygin-energy-symmetry}
		\cQ^{\alpha\beta}
		=
		\cQ^{\beta\alpha},
		\qquad
		0\le\alpha,\beta\le3.
	\end{equation}
	
	\begin{lemma}[{\bf Energy estimate}]\label{lem:Chaplygin-energy-estimate}
		Let \(\phi\) be a smooth solution to
		\eqref{eq:Chaplygin-Cauchy} on \([0,T]\times\mathbb R^3\).
		Assume that the initial data satisfy \eqref{YHC-4} and
		\eqref{eq:CVP-bootstrap-parameter-range}-\eqref{eq:CVP-bootstrap-Strichartz-Gamma}
		hold. Suppose that \(\ve_1>0\) is small. Then
		\begin{equation}\label{eq:Chaplygin-energy-growth-epsilon1}
			\sup_{0\le t\le T}(1+t)^{-\eta}\|\partial\phi(t)\|_{H^N(\mathbb R^3)}
			\ls
			\ve +\ve_1^2.
		\end{equation}
	\end{lemma}
	
	\begin{proof}
		Acting $\partial_x^a$ with $a\in\N_0^3$ on two sides of
		\eqref{eq:Chaplygin-energy-reformulation} yields
		\begin{equation}\label{eq:Chaplygin-commuted-equation}
			\begin{aligned}
				\Box\partial_x^a\phi
				&=
				I_1^a+I_2^a,
				\\
				I_1^a
				=
				\sum_{\alpha,\beta=0}^3
				\cQ^{\alpha\beta}
				\partial_x^a\partial_{\alpha\beta}^2\phi,\quad
				&I_2^a
				=
				\sum_{\alpha,\beta=0}^3
				\sum_{\substack{b+c=a,\\ |b|\le |a|-1}}
				C_{abc}
				\partial_x^c\cQ^{\alpha\beta}
				\partial_x^b\partial_{\alpha\beta}^2\phi,
			\end{aligned}
		\end{equation}
		where $I_2^a$ vanishes for $|a|=0$.
		
		Multiplying \eqref{eq:Chaplygin-commuted-equation} by
		$\partial_t\partial_x^a\phi$ and integrating the resulting equality over
		$[0,t]\times\mathbb R^3$, one has
		\begin{equation}\label{eq:Chaplygin-basic-energy}
			\begin{aligned}
				\|\partial_x^a\partial\phi(t)\|_{L^2(\mathbb R^3)}^2
				&\ls
				\|\partial_x^a\partial\phi(0)\|_{L^2(\mathbb R^3)}^2
				+
				\big|
				\int_0^t\int_{\mathbb R^3}
				\partial_t\partial_x^a\phi\,
				I_1^a\,dx\,d\tau
				\big|
				\\
				&\quad+
				\int_0^t
				\|\partial_t\partial_x^a\phi(\tau)\|_{L^2(\mathbb R^3)}
				\|I_2^a(\tau)\|_{L^2(\mathbb R^3)}
				d\tau.
			\end{aligned}
		\end{equation}
		At first, we deal with the last term in the first two lines of
		\eqref{eq:Chaplygin-basic-energy}. Note that
		\begin{equation}\label{eq:Chaplygin-principal-identity}
			\begin{aligned}
				\cQ^{\alpha\beta}
				\partial_t\partial_x^a\phi\,
				\partial_{\alpha\beta}^2\partial_x^a\phi
				&=
				\partial_\alpha
				\big(
				\cQ^{\alpha\beta}
				\partial_t\partial_x^a\phi\,
				\partial_\beta\partial_x^a\phi
				\big)
				-
				\partial_\alpha\cQ^{\alpha\beta}
				\partial_t\partial_x^a\phi\,
				\partial_\beta\partial_x^a\phi
				\\
				&\quad-
				\frac12\partial_t
				\big(
				\cQ^{\alpha\beta}
				\partial_\alpha\partial_x^a\phi\,
				\partial_\beta\partial_x^a\phi
				\big)
				+
				\frac12\partial_t\cQ^{\alpha\beta}
				\partial_\alpha\partial_x^a\phi\,
				\partial_\beta\partial_x^a\phi,
			\end{aligned}
		\end{equation}
		where the Einstein summation convention for
		$\alpha,\beta=0,1,2,3$ and the symmetric condition
		\eqref{eq:Chaplygin-energy-symmetry} are used in
		\eqref{eq:Chaplygin-principal-identity}.
		
		Due to
		\begin{equation}\label{eq:Chaplygin-Q-pointwise}
			|\cQ^{\alpha\beta}|
			\ls
			|\partial\phi|
			+
			|\partial\phi|^2,
		\end{equation}
		it follows from \eqref{eq:CVP-bootstrap-pointwise-phi} that
		\begin{equation}\label{eq:Chaplygin-Q-Linfty}
			\begin{aligned}
				\|\cQ^{\alpha\beta}(t)\|_{L^\infty}
				&\ls
				\|\partial\phi(t)\|_{L^\infty}
				+
				\|\partial\phi(t)\|_{L^\infty}^2
				\\
				&\ls
				\|\partial\phi(t)\|_{W^{2,\infty}}
				+
				\|\partial\phi(t)\|_{W^{2,\infty}}^2
				\\
				&\ls
				(1+t)^{-1}\varepsilon_1
				+
				(1+t)^{-2}\varepsilon_1^2
				\\
				&\ls
				\varepsilon_1.
			\end{aligned}
		\end{equation}
		
		Direct computation yields
		\begin{equation}\label{eq:Chaplygin-spatial-Q-derivative}
			\begin{aligned}
				\|\nabla\cQ^{\alpha\beta}(t)\|_{L^\infty}
				&\ls
				\|\nabla\partial\phi(t)\|_{L^\infty}
				+
				\|\partial\phi(t)\|_{L^\infty}
				\|\nabla\partial\phi(t)\|_{L^\infty}
				\\
				&\ls
				\|\partial\phi(t)\|_{B_{\infty,1}^{1+\delta}}
				+
				\|\partial\phi(t)\|_{B_{\infty,1}^{1+\delta}}^2.
			\end{aligned}
		\end{equation}
		Next we estimate $\partial_t\cQ^{\alpha\beta}$. Due to
		$\partial_t^2\phi
			=\Delta\phi	+2\ds\sum_{j=1}^3\cQ^{0j}\partial_{tj}^2\phi
			+\ds\sum_{k,j=1}^3\cQ^{kj}\partial_{kj}^2\phi$,
		by \eqref{eq:Chaplygin-Q-Linfty}, one has
		\begin{equation}\label{eq:Chaplygin-ttphi-Linfty}
			\begin{aligned}
				\|\partial_t^2\phi(t)\|_{L^\infty}
				&\ls
				\|\nabla\partial\phi(t)\|_{L^\infty}
				+
				\|\cQ(t)\|_{L^\infty}
				\|\nabla\partial\phi(t)\|_{L^\infty}
				\\
				&\ls
				\|\partial\phi(t)\|_{B_{\infty,1}^{1+\delta}}
				+
				\|\partial\phi(t)\|_{B_{\infty,1}^{1+\delta}}^2.
			\end{aligned}
		\end{equation}
		Consequently,
		\begin{equation}\label{eq:Chaplygin-time-Q-derivative}
			\begin{aligned}
				\|\partial_t\cQ^{\alpha\beta}(t)\|_{L^\infty}
				&\ls
				\|\partial^2\phi(t)\|_{L^\infty}
				+\|\partial\phi(t)\|_{L^\infty}
				\|\partial^2\phi(t)\|_{L^\infty}\\
				&\ls\|\partial\phi(t)\|_{B_{\infty,1}^{1+\delta}}
				+\|\partial\phi(t)\|_{B_{\infty,1}^{1+\delta}}^2.
			\end{aligned}
		\end{equation}
		Combining
		\eqref{eq:Chaplygin-spatial-Q-derivative} and
		\eqref{eq:Chaplygin-time-Q-derivative}, one has
		\begin{equation}\label{eq:Chaplygin-full-Q-derivative}
			\|\partial\cQ^{\alpha\beta}(t)\|_{L^\infty}
			\ls
			\|\partial\phi(t)\|_{B_{\infty,1}^{1+\delta}}
			+
			\|\partial\phi(t)\|_{B_{\infty,1}^{1+\delta}}^2.
		\end{equation}
		
		It follows from \eqref{eq:Chaplygin-Q-Linfty},
		\eqref{eq:Chaplygin-full-Q-derivative} and integration over
		$[0,t]\times\mathbb R^3$ for
		\eqref{eq:Chaplygin-principal-identity} that
		\begin{equation}\label{eq:Chaplygin-I1-estimate}
			\begin{aligned}
				\big|
				\int_0^t\int_{\mathbb R^3}
				\partial_t\partial_x^a\phi\,
				I_1^a\,dx\,d\tau
				\big|
				&\ls
				\varepsilon_1
				\|\partial_x^a\partial\phi(0)\|_{L^2}^2
				+
				\varepsilon_1
				\|\partial_x^a\partial\phi(t)\|_{L^2}^2
				\\
				&\quad+
				\int_0^t
				\big(
				\|\partial\phi(\tau)\|_{B_{\infty,1}^{1+\delta}}
				+
				\|\partial\phi(\tau)\|_{B_{\infty,1}^{1+\delta}}^2
				\big)
				\|\partial_x^a\partial\phi(\tau)\|_{L^2}^2
				d\tau.
			\end{aligned}
		\end{equation}

		Next, we treat $\|I_2^a\|_{L^2(\mathbb R^3)}$ with $|a|\ge1$ in
		\eqref{eq:Chaplygin-commuted-equation}. It is easy to find that
		\begin{equation}\label{eq:Chaplygin-I2-frequency}
			\begin{aligned}
				\|I_2^a\|_{L^2}
				&\ls
				\big\|
				\|P_kI_2^a\|_{L^2}
				\big\|_{\ell_k^2}\\
				&\ls
				\sum_{\alpha,\beta=0}^3
				\sum_{\substack{b+c=a,\\ |b|\le |a|-1}}
				\big\|
				\sum_{(k_1,k_2)\in\mathcal X_k}
				\big\|
				P_k
				\big(
				P_{k_1}\partial_{\alpha\beta}^2\partial_x^b\phi
				P_{k_2}\partial_x^c\cQ^{\alpha\beta}
				\big)
				\big\|_{L^2}
				\big\|_{\ell_k^2}.
			\end{aligned}
		\end{equation}
		Note that $\cQ^{00}=0$ and $|c|\ge1$ in \eqref{eq:Chaplygin-I2-frequency}. When
		$\max\{k_1,k_2\}=k_1$, one has that by Bernstein's inequality
		\begin{equation}\label{eq:Chaplygin-I2-first-frequency-case}
			\begin{aligned}
				&
				\big\|
				P_k
				\big(
				P_{k_1}\partial_{\alpha\beta}^2\partial_x^b\phi
				P_{k_2}\partial_x^c\cQ^{\alpha\beta}
				\big)
				\big\|_{L^2}
				\\
				&\ls
				\|P_{k_1}\partial\partial_x\partial_x^b\phi\|_{L^2}
				\|P_{k_2}\partial_x^c\cQ^{\alpha\beta}\|_{L^\infty}
				\\
				&\ls
				2^{k_1(|b|+1)}
				2^{k_2(|c|-1)}
				\|P_{k_1}\partial\phi\|_{L^2}
				\|P_{k_2}\partial_x\cQ^{\alpha\beta}\|_{L^\infty}
				\\
				&\ls
				2^{k_1|a|}
				\|P_{k_1}\partial\phi\|_{L^2}
				\|P_{k_2}\partial_x\cQ^{\alpha\beta}\|_{L^\infty}.
			\end{aligned}
		\end{equation}
		Therefore,
		\begin{equation}\label{eq:Chaplygin-I2-first-frequency-sum}
			\begin{aligned}
				&
				\big\|
				\sum_{\substack{(k_1,k_2)\in\mathcal X_k,\\
				\max\{k_1,k_2\}=k_1}}
				\big\|
				P_k
				\big(
				P_{k_1}\partial_{\alpha\beta}^2\partial_x^b\phi
				P_{k_2}\partial_x^c\cQ^{\alpha\beta}
				\big)
				\big\|_{L^2}
				\big\|_{\ell_k^2}
				\ls
				\|\partial\phi\|_{H^{|a|}}
				\|\cQ^{\alpha\beta}\|_{B_{\infty,1}^{1+\delta}}.
			\end{aligned}
		\end{equation}
		When $\max\{k_1,k_2\}=k_2$, we can analogously obtain
		\begin{equation}\label{eq:Chaplygin-I2-second-frequency-case}
			\begin{aligned}
				&
				\big\|
				\sum_{\substack{(k_1,k_2)\in\mathcal X_k,\\
				\max\{k_1,k_2\}=k_2}}
				\big\|
				P_k
				\big(
				P_{k_1}\partial_{\alpha\beta}^2\partial_x^b\phi
				P_{k_2}\partial_x^c\cQ^{\alpha\beta}
				\big)
				\big\|_{L^2}
				\big\|_{\ell_k^2}
				\ls
				\|\partial_x\cQ^{\alpha\beta}\|_{H^{|a|-1}}
				\|\partial\phi\|_{B_{\infty,1}^{1+\delta}}.
			\end{aligned}
		\end{equation}
		Note that
		$B_{p,r}^s(\mathbb R^3)\cap L^\infty(\mathbb R^3)$ is an algebra for
		$s>0$ and $p,r\in[1,\infty]$. Then, for $1\le |a|\le N$, one has
		\begin{equation}\label{eq:Chaplygin-Q-Besov}
			\begin{aligned}
				\|\cQ^{\alpha\beta}\|_{B_{\infty,1}^{1+\delta}}
				&\ls
				\|\partial\phi\|_{B_{\infty,1}^{1+\delta}}
				+
				\|\partial\phi\|_{B_{\infty,1}^{1+\delta}}
				\|\partial\phi\|_{L^\infty}
				\\
				&\ls
				\|\partial\phi\|_{B_{\infty,1}^{1+\delta}}
				+
				\|\partial\phi\|_{B_{\infty,1}^{1+\delta}}^2.
			\end{aligned}
		\end{equation}
		Moreover,
		\begin{equation}\label{eq:Chaplygin-Q-high-Sobolev}
			\begin{aligned}
				\|\partial_x\cQ^{\alpha\beta}\|_{H^{|a|-1}}
				\ls
				\|\cQ^{\alpha\beta}\|_{H^{|a|}}
				\ls
				\|\partial\phi\|_{H^{|a|}}
				\left(
				1+
				\|\partial\phi\|_{L^\infty}
				\right).
			\end{aligned}
		\end{equation}
		Hence,
		\begin{equation}\label{eq:Chaplygin-I2-final}
			\|I_2^a\|_{L^2}
			\ls
			\|\partial\phi\|_{H^{|a|}}
			\big(
			\|\partial\phi\|_{B_{\infty,1}^{1+\delta}}
			+
			\|\partial\phi\|_{B_{\infty,1}^{1+\delta}}^2
			\big).
		\end{equation}

		Collecting \eqref{eq:Chaplygin-basic-energy},
		\eqref{eq:Chaplygin-I1-estimate},
		\eqref{eq:Chaplygin-I2-final}, and using the smallness of
		$\varepsilon_1$, we arrive at
		\begin{equation}\label{eq:Chaplygin-energy-integral}
			\begin{aligned}
				\|\partial\phi(t)\|_{H^N(\mathbb R^3)}^2
				&\ls
				\|\partial\phi(0)\|_{H^N(\mathbb R^3)}^2
				+\int_0^t
				\big(
				\|\partial\phi(\tau)\|_{B_{\infty,1}^{1+\delta}}
				+\|\partial\phi(\tau)\|_{B_{\infty,1}^{1+\delta}}^2
				\big)
				\|\partial\phi(\tau)\|_{H^N(\mathbb R^3)}^2
				d\tau.
			\end{aligned}
		\end{equation}
		
		It follows from \eqref{eq:CVP-bootstrap-pointwise-phi} and $W^{2,\infty}(\R^3) \hookrightarrow B_{\infty,1}^{1+\delta}(\R^3)$ that
		\begin{equation}\label{eq:Chaplygin-time-coefficient}
			\begin{aligned}
				\|\partial\phi(\tau)\|_{B_{\infty,1}^{1+\delta}}
				+\|\partial\phi(\tau)\|_{B_{\infty,1}^{1+\delta}}^2
				\ls
				\varepsilon_1(1+\tau)^{-1}.
			\end{aligned}
		\end{equation}
		Substituting \eqref{eq:Chaplygin-time-coefficient} into
		\eqref{eq:Chaplygin-energy-integral} yields
		\begin{equation}\label{eq:Chaplygin-before-Gronwall}
			\begin{aligned}
				\|\partial\phi(t)\|_{H^N}^2
				&\ls
				\|\partial\phi(0)\|_{H^N}^2
				+
				\int_0^t
				\varepsilon_1(1+\tau)^{-1}
				\|\partial\phi(\tau)\|_{H^N}^2
				d\tau.
			\end{aligned}
		\end{equation}
		By Gronwall's inequality, \eqref{YHC-4} and \eqref{eq:CVP-bootstrap-energy-ordinary},
		\begin{equation}\label{eq:Chaplygin-after-Gronwall}
			\begin{aligned}
				\|\partial\phi(t)\|_{H^N}
				\ls
				\ve + \ve_1^2 \int_{0}^{t} (1+\tau)^{-1+\eta}d\tau
				\ls
				\ve + \ve_1^2 \eta^{-1}(1+t)^\eta.
			\end{aligned}
		\end{equation}
		Therefore, $\ds\sup_{0\le t\le T}(1+t)^{-\eta}\|\partial\phi(t)\|_{H^N(\mathbb R^3)}
			\ls \ve +\ve_1^2$ holds, which yields \eqref{eq:Chaplygin-energy-growth-epsilon1}.
	\end{proof}

	\section{Equation of the good unknown}
	\label{sec:equation-for-V}
	
	Introduce a good unknown as follows
	\begin{equation}
		\label{eq:V-definition}
		V=\phi-\phi\partial_t\phi.
	\end{equation}
	Note that	
	\begin{equation}
		\label{eq:quadratic-part-Q0-identity}
		2\partial_t\phi\Delta\phi
		-
		2\sum_{k=1}^3\partial_k\phi\,\partial_{tk}^2\phi
		=
		2Q_0(\partial_t\phi,\phi)
		-
		2\partial_t\phi\,\Box\phi .
	\end{equation}
	Then we have
	\begin{equation}\label{eq:Chaplygin-final-form-V-section}
		\Box\phi
		=
		2Q_0(\partial_t\phi,\phi)
		+\mathcal N(\phi),
		\end{equation}
where $\mathcal N(\phi)
		=\mathcal N_3(\phi)+\mathcal N_4(\phi)$ with
	\begin{equation}
		\label{eq:N3-definition-V-section}
		\begin{split}
			\mathcal N_3(\phi)
			&=
			-4(\partial_t\phi)^2\Delta\phi
			+
			4\sum_{k=1}^3
			\partial_t\phi\,\partial_k\phi\,\partial_{tk}^2\phi
			+
			|\nabla\phi|^2\Delta\phi
			-
			\sum_{k,j=1}^3
			\partial_k\phi\,\partial_j\phi\,\partial_{kj}^2\phi
		\end{split}
	\end{equation}
and
	\begin{equation}
		\label{eq:N4-definition-V-section}
		\begin{split}
			\mathcal N_4(\phi)
			&=
			-2\partial_t\phi\,|\nabla\phi|^2\Delta\phi
			+
			2\sum_{k,j=1}^3
			\partial_t\phi\,\partial_k\phi\,\partial_j\phi\,\partial_{kj}^2\phi.
		\end{split}
	\end{equation}
	By
	\begin{equation}
		\label{eq:box-product-identity-V-section}
		\Box(fg)
		=f\Box g+g\Box f+2Q_0(f,g),
	\end{equation}
	we then have
	\begin{equation}
		\label{eq:box-V-first-computation}
		\begin{split}
			\Box V
			&=\Box\phi
			-\Box(\phi\partial_t\phi)=
			\Box\phi
			-\phi\partial_t\Box\phi
			-(\partial_t\phi)\Box\phi
			-2Q_0(\phi,\partial_t\phi).
		\end{split}
	\end{equation}
	Substituting \eqref{eq:Chaplygin-final-form-V-section} into
	\eqref{eq:box-V-first-computation} yields
	
		\begin{equation}
		\label{eq:box-V-final-expanded}
		\begin{split}
			\Box V
			&=
			\mathcal N_3(\phi)
			+
			\mathcal N_4(\phi)
			-
			2(\partial_t\phi)Q_0(\partial_t\phi,\phi)
			-
			2\phi\,\partial_t Q_0(\partial_t\phi,\phi)\\
			&\quad
			-
			(\partial_t\phi)\mathcal N_3(\phi)
			-
			\phi\partial_t\mathcal N_3(\phi)
			-
			(\partial_t\phi)\mathcal N_4(\phi)
			-
			\phi\partial_t\mathcal N_4(\phi).
		\end{split}
	\end{equation}
	In addition,
	\begin{equation}
		\label{eq:dt-Q0-identity-V-section}
		\partial_tQ_0(\partial_t\phi,\phi)
		=
		Q_0(\partial_t^2\phi,\phi)
		+
		Q_0(\partial_t\phi,\partial_t\phi).
	\end{equation}
It follows from direct computation that

	\begin{equation}
		\label{eq:box-V-schematic-null}
		\Box V
		=\mathcal R(\phi)=
		\mathcal R_3(\phi)
		+\mathcal R_4(\phi)
		+\mathcal R_5(\phi)
	\end{equation}
	with
	\begin{equation}
		\label{eq:R3R4R5-definition}
		\begin{split}
			\mathcal R_3(\phi)
			&=
			\mathcal N_3(\phi)
			-
			2(\partial_t\phi)Q_0(\partial_t\phi,\phi)
			-
			2\phi\partial_tQ_0(\partial_t\phi,\phi),  \\
			\mathcal R_4(\phi)
			&=
			\mathcal N_4(\phi)
			-
			(\partial_t\phi)\mathcal N_3(\phi)
			-
			\phi\partial_t\mathcal N_3(\phi),  \\
			\mathcal R_5(\phi)
			&=
			-
			(\partial_t\phi)\mathcal N_4(\phi)
			-
			\phi\partial_t\mathcal N_4(\phi),
		\end{split}
	\end{equation}
	where \(\mathcal R_3\), \(\mathcal R_4\) and \(\mathcal R_5\) are cubic,
	quartic and quintic in \(\phi\), respectively. Moreover, \(\mathcal R_k\) for $3\le k\le 5$
	satisfy the null conditions.

	\section{Energy estimate for $\Gamma\phi$}
	\label{sec:CVP-one-Gamma-energy-closure}
	
	In this section, based on the equation \eqref{eq:box-V-schematic-null}, we derive the energy estimate for $\Gamma\phi$.

	\begin{lemma}[{\bf Energy estimate for $\Gamma\phi$}]
		\label{lem:CVP-one-Gamma-energy-closure}
	Let \(\phi\) be a smooth solution to
	\eqref{eq:Chaplygin-Cauchy} on \([0,T]\times\mathbb R^3\).
	Assume that the initial data satisfy \eqref{YHC-4} and
	\eqref{eq:CVP-bootstrap-parameter-range}-\eqref{eq:CVP-bootstrap-Strichartz-Gamma}
		hold. Suppose that \(\ve_1>0\) is small. Then
		\begin{equation}
			\label{eq:CVP-one-Gamma-energy-conclusion}
			\sup_{0\le t\le T}
			(1+t)^{-\eta}
			\sum_{\Gamma}
			\|\p\Gamma\phi(t)\|_{H^{11}_x}
			\ls
			\ve+\ve_1^2+\ve_1^3.
		\end{equation}
		\end{lemma}
		
	\begin{proof}
		Note that $\Box\Gamma V=\Gamma\mathcal R(\phi)$.
		By the standard energy estimate, $|\Gamma f|\ls A(t,x)|\p f|$ and $[\Gamma,\p_\alpha]\in\operatorname{span}\{\p_0,\p_1,\p_2,\p_3\}$,
		one has
		\begin{equation}
			\label{eq:CVP-Gamma-V-energy-estimate}
			\begin{aligned}
				\sum_{\Gamma}
				\|\p\Gamma V(t)\|_{H^{11}_x}
				\ls
				&
				\sum_{\Gamma}
				\|\p\Gamma V(0)\|_{H^{11}_x}
				+
				\int_0^t
				\sum_{\Gamma}
				\|\nabla^{\le 11}\Gamma\mathcal R(\phi)(\tau)\|_{L^2_x}
				\,d\tau
				\\
				\ls
				&
				\sum_{\Gamma}
				\|\p\Gamma V(0)\|_{H^{11}_x}
				+
				\int_0^t
				\sum_{\Gamma}
				\|\Gamma\nabla^{\le 11}\mathcal R(\phi)(\tau)\|_{L^2_x}
				\,d\tau
				\\
				&+
				\int_0^t
				\|\p\nabla^{\le 10}\mathcal R(\phi)(\tau)\|_{L^2_x}
				\,d\tau
				\\
				\ls
				&
				\sum_{\Gamma}
				\|\p\Gamma V(0)\|_{H^{11}_x}
				+
				\sum_{k=3}^{5}\int_0^t
				\|A(\tau,\cdot)\p\nabla^{\le 11}\mathcal R_k(\phi)(\tau)\|_{L^2_x}
				\,d\tau .
			\end{aligned}
		\end{equation}
		In addition, by \eqref{YHC-4}, one has
		\begin{equation}
			\label{eq:CVP-Gamma-V-initial}
			\begin{aligned}
				\sum_{\Gamma}
				\|\p\Gamma V(0)\|_{H^{11}_x}
				\ls
				\|\w{x}\nabla^{\le 11}\p^2(\phi-\phi\p_t\phi)(0)\|_{L^{2}_x} + \|\p(\phi-\phi\p_t\phi)(0)\|_{H^{11}_x}
				\ls
				\ve.
			\end{aligned}
		\end{equation}

		We now estimate $ \int_0^t\|A(\tau,\cdot)\p\nabla^{\le 11}\mathcal R_3(\phi)(\tau)\|_{L^2_x}
		\,d\tau$.
		
		By the Littlewood-Paley decomposition, $(1+\tau+|x|)\approx(1+\tau)+\w{x}$,
		$\w{x}^{2}\in A_2(\R^3)$ and Lemma \ref{lem:weighted bernstein}, one has
		\begin{equation}
			\label{eq:CVP-Apartial-R3-LP-start}
			\begin{aligned}
				&\|A(\tau,\cdot)\p\nabla^{\le 11}\mathcal R_3(\phi)(\tau)\|_{L^2_x}\\
				&\ls
				\sum_{k\ge -1}
				\|(1+\tau)P_k\p\nabla^{\le 11}\mathcal R_3(\phi)(\tau)\|_{L^2_x}
				+
				\sum_{k\ge -1}
				\|\w{x}P_k\p\nabla^{\le 11}\mathcal R_3(\phi)(\tau)\|_{L^2_x}\\
				&\ls
				\sum_{k\ge -1}
				2^{11k}\|(1+\tau)P_k\p\mathcal R_3(\phi)(\tau)\|_{L^2_x}
				+
				\sum_{k\ge -1}
				2^{11k}\|\w{x}P_k\p\mathcal R_3(\phi)(\tau)\|_{L^2_x}\\
				&\ls
				\sum_{k\ge -1}
				2^{11k}\|A(\tau,\cdot)P_k\p\mathcal R_3(\phi)(\tau)\|_{L^2_x}.
			\end{aligned}
		\end{equation}
		Since \(\p \mathcal R_3(\phi)\) is a cubic nonlinearity with such forms
		$\p^{\le3}\phi\cdot \p^{\le3}\phi\cdot \p^{\le3}\p\phi$,
		then it follows from
		$A(\tau,x)\le(1+\tau)^{-1}A(\tau,x)^2$
		and the weighted \(L^2\)-boundedness of \(P_k\) (see Lemma \ref{lem:A2}) that
		\begin{equation}
			\label{eq:CVP-R3-dyadic-product}
			\begin{aligned}
				&
				\|A(\tau,\cdot)P_k\p \mathcal R_3(\phi)(\tau)\|_{L^2_x}
				\\
				&\ls
				\sum_{(k_1,k_2,k_3)\in\cY_k}
				\|A(\tau,\cdot)P_{k_1}\p^{\le3}\phi
				\cdot
				P_{k_2}\p^{\le3}\phi
				\cdot
				P_{k_3}\p^{\le3}\p\phi\|_{L^2_x}
				\\
				&\ls
				(1+\tau)^{-1}
				\sum_{(k_1,k_2,k_3)\in\cY_k}
				\|A^2(\tau,\cdot)P_{k_1}\p^{\le3}\phi
				\cdot
				P_{k_2}\p^{\le3}\phi
				\cdot
				P_{k_3}\p^{\le3}\p\phi\|_{L^2_x}.
			\end{aligned}
		\end{equation}
		For \(-1\le k\le3\), by H\"older's inequality and Lemma \ref{lem:CVP-higher-time-derivative-estimates}, we have
		\begin{equation}
			\label{eq:CVP-R3-low-frequency}
			\begin{aligned}
				&
				2^{11k}
				\|A(\tau,\cdot)P_k\p \mathcal R_3(\phi)(\tau)\|_{L^2_x}
				\\
				&\ls
				(1+\tau)^{-1}2^{-\frac12 k}
				\big(
				\sum_{j\ge-1}
				\|A(\tau,\cdot)P_j\p^{\le3}\phi(\tau)\|_{L^\infty_x}
				\big)^2
				\|\p^{\le3}\p\phi(\tau)\|_{H^1_x}
				\\
				&\ls
				(1+\tau)^{-1}2^{-\frac12 k}
				\big(
				\sum_{j\ge-1}
				\|A(\tau,\cdot)P_j\p^{\le3}\phi(\tau)\|_{L^\infty_x}
				\big)^2
				\|\p\phi(\tau)\|_{H^4_x}.
			\end{aligned}
		\end{equation}
		For $k\ge 4$, recalling the definition of $\cY_k$, one has $\max_{i=1,2,3}\{k_i\}\ge k-4\ge 0$.
		Therefore, applying H\"older's inequality, Lemma \ref{lem:weighted bernstein} and Lemma \ref{lem:CVP-higher-time-derivative-estimates} yields
		\begin{equation}\label{eq:CVP-R3-high-frequency}
			\begin{aligned}
				&2^{11k}
				\|A(\tau,\cdot)P_k\p\mathcal R_3(\phi)(\tau)\|_{L^2_x}\\
				&\ls
				(1+\tau)^{-1}2^{-\frac12 k}
				\sum_{\substack{(k_1,k_2,k_3)\in\cY_k\\
						\max\{k_i\}=k_1}}
				2^{(11+\frac12)k_1}
				\|P_{k_1}\p^{\le3}\phi\|_{L^2_x}
				\|AP_{k_2}\p^{\le3}\phi\|_{L^\infty_x}
				\|AP_{k_3}\p^{\le3}\p\phi\|_{L^\infty_x}\\
				&\quad+
				(1+\tau)^{-1}2^{-\frac12 k}
				\sum_{\substack{(k_1,k_2,k_3)\in\cY_k\\
						\max\{k_i\}=k_3}}
				\|AP_{k_1}\p^{\le3}\phi\|_{L^\infty_x}
				\|AP_{k_2}\p^{\le3}\phi\|_{L^\infty_x}
				2^{(11+\frac12)k_3}
				\|P_{k_3}\p^{\le 3}\p\phi(\tau)\|_{L^2_x}\\
				&\ls
				(1+\tau)^{-1}2^{-\frac12 k}
				\sum_{\substack{(k_1,k_2,k_3)\in\cY_k\\
						\max\{k_i\}=k_1}}
				2^{(10+\frac12)k_1}
				\|P_{k_1}\p^{\le3}\p\phi(\tau)\|_{L^2_x}
				\|AP_{k_2}\p^{\le3}\phi\|_{L^\infty_x}
				\|AP_{k_3}\p^{\le4}\phi\|_{L^\infty_x}\\
				&\quad+
				(1+\tau)^{-1}2^{-\frac12 k}
				\sum_{\substack{(k_1,k_2,k_3)\in\cY_k\\
						\max\{k_i\}=k_3}}
				\|AP_{k_1}\p^{\le3}\phi\|_{L^\infty_x}
				\|AP_{k_2}\p^{\le3}\phi\|_{L^\infty_x}
				2^{(11+\frac12)k_3}
				\|P_{k_3}\p^{\le3}\p\phi(\tau)\|_{L^2_x}\\
				&\ls
				(1+\tau)^{-1}2^{-\frac12 k}
				\big(
				\sum_{j\ge-1}
				\|A(\tau,\cdot)P_j\p^{\le4}\phi(\tau)\|_{L^\infty_x}
				\big)^2
				\|\p^{\le 3}\p\phi(\tau)\|_{H^{12}_x}
				\\
				&\ls
				(1+\tau)^{-1}2^{-\frac12 k}
				\big(
				\sum_{j\ge-1}
				\|A(\tau,\cdot)P_j\p^{\le4}\phi(\tau)\|_{L^\infty_x}
				\big)^2
				\|\p\phi(\tau)\|_{H^{N}_x},
			\end{aligned}
		\end{equation}
		where the condition $N\ge 15$ is used in the last inequality of \eqref{eq:CVP-R3-high-frequency}.
		Summing \eqref{eq:CVP-R3-low-frequency} and
		\eqref{eq:CVP-R3-high-frequency} over \(k\), and using
		$\ds\sum_{k\ge -1}2^{-\f12 k}<\infty$,
		we arrive at
		\begin{equation}
			\label{eq:CVP-Apartial-R3-before-bootstrap}
			\begin{aligned}
				&
				\|A(\tau,\cdot)\p\nabla^{\le 11}
				\mathcal R_3(\phi)(\tau)\|_{L^2_x}
				\\
				&\ls
				(1+\tau)^{-1}
				\big(
				\sum_{j\ge-1}
				\|A(\tau,\cdot)P_j\p^{\le4}\phi(\tau)\|_{L^\infty_x}
				\big)^2
				\|\p\phi(\tau)\|_{H^N_x}.
			\end{aligned}
		\end{equation}
		By Lemma \ref{lem:weighted-dyadic-Linfty-sum}, \eqref{eq:CVP-bootstrap-pointwise-phi} and
		Lemma~\ref{lem:CVP-higher-time-derivative-estimates},
		\[
		\sum_{j\ge-1}
		\|A(\tau,\cdot)P_j\p^{\le4}\phi(\tau)\|_{L^\infty_x}
		\ls
		\|A(\tau,\cdot)\nabla^{\le 1}\p^{\le 4}\phi(\tau)\|_{L^\infty_x}
		\ls
		\ve_1.
		\]
		Hence
		\begin{equation}
			\label{eq:CVP-Apartial-R3-final}
			\|A(\tau,\cdot)\p\nabla^{\le 11}
			\mathcal R_3(\phi)(\tau)\|_{L^2_x}
			\ls
			\ve_1^2(1+\tau)^{-1}
			\|\p\phi(\tau)\|_{H^N_x}.
		\end{equation}
		More easily, one has
		\begin{equation}
			\label{eq:CVP-Apartial-R45-final}
			\|A(\tau,\cdot)\p\nabla^{\le 11}
			(\mathcal R_4+\mathcal R_5)(\phi)(\tau)\|_{L^2_x}
			\ls
			\ve_1^3(1+\tau)^{-1}
			\|\p\phi(\tau)\|_{H^N_x}.
		\end{equation}
		From \eqref{eq:CVP-Apartial-R3-final} and
		\eqref{eq:CVP-Apartial-R45-final}, we obtain
		\begin{equation}
			\label{eq:CVP-Apartial-R-final-source}
			\sum_{k=3}^{5}
			\|A(\tau,\cdot)\p\nabla^{\le 11}
			\mathcal R_k(\phi)(\tau)\|_{L^2_x}
			\ls
			\ve_1^2(1+\tau)^{-1}
			\|\p\phi(\tau)\|_{H^N_x}.
		\end{equation}
		Substituting
		\eqref{eq:CVP-Gamma-V-initial} and
		\eqref{eq:CVP-Apartial-R-final-source} into
		\eqref{eq:CVP-Gamma-V-energy-estimate}, and using
		$\|\p\phi(t)\|_{H^N_x}
		\le\ve_1(1+t)^\eta,
		$
		we obtain
		\begin{equation}
			\label{eq:CVP-Gamma-V-final}
			\begin{aligned}
				\sum_{\Gamma}
				\|\p\Gamma V(t)\|_{H^{11}_x}
				&\ls
				\ve
				+
				\ve_1^3
				\int_0^t
				(1+\tau)^{-1+\eta}\,d\tau
				\ls_\eta
				\ve
				+
				\ve_1^3(1+t)^\eta.
			\end{aligned}
		\end{equation}

		It remains to recover \(\Gamma\phi\).
        Note that
		\begin{equation}
			\label{eq:CVP-Gamma-recovery-identity}
			\Gamma\phi
			=
			\Gamma V
			+
			(\Gamma\phi)\p_t\phi
			+
			\phi\,\Gamma\p_t\phi.
		\end{equation}
		By $\|fg\|_{H^s_x}\ls\|f\|_{L^\infty_x}\|g\|_{H^s_x}
		+\|g\|_{L^\infty_x}\|f\|_{\dot H^s_x}
		$ for $s> 0 $ and Lemma \ref{lem:CVP-higher-time-derivative-estimates},
        one has
		\begin{equation}
			\label{eq:CVP-recovery-first-term}
			\begin{aligned}
				&
				\sum_{\Gamma}
				\|\p((\Gamma\phi)\p_t\phi)(t)\|_{H^{11}_x}
				\\
				&\ls
				\sum_{\Gamma}
				\|(\p\Gamma\phi)\p_t\phi(t)\|_{H^{11}_x}
				+
				\sum_{\Gamma}
				\|(\Gamma\phi)\p\p_t\phi(t)\|_{H^{11}_x}
				\\
				&\ls
				\|\p_t\phi(t)\|_{L^{\infty}_x}
				\sum_{\Gamma}
				\|\p\Gamma\phi(t)\|_{H^{11}_x}
				+
				\|\p_t\phi(t)\|_{H^{11}_x}
				\sum_{\Gamma}
				\|\p\Gamma\phi(t)\|_{L^{\infty}_x}
				\\
				&\quad+
				\sum_{\Gamma}
				\|\Gamma\phi(t)\|_{L^\infty_x}
				\|\p\p_t\phi(t)\|_{H^{11}_x}
				+
				\|\p\p_t\phi(t)\|_{L^{\infty}_x}
				\sum_{\Gamma}
				\|\Gamma\phi(t)\|_{\dot{H}^{11}_x}
				\\
				&\ls
				\ve_1
				\sum_{\Gamma}
				\|\p\Gamma\phi(t)\|_{H^{11}_x}
				+
				\ve_1
				\|\p\phi(t)\|_{H^N_x}
				+
				\ve_1
				\|\p\phi(t)\|_{H^N_x}
				+
				\ve_1
				\sum_{\Gamma}
				\|\nabla\Gamma\phi(t)\|_{H^{10}_x}
				\\
				&\ls
				\ve_1
				\sum_{\Gamma}
				\|\p\Gamma\phi(t)\|_{H^{11}_x}
				+
				\ve_1
				\|\p\phi(t)\|_{H^N_x}
			\end{aligned}
		\end{equation}
		and
		\begin{equation}
			\label{eq:CVP-recovery-second-term-1}
			\begin{aligned}
				&
				\sum_{\Gamma}
				\|\p(\phi\,\Gamma\p_t\phi)(t)\|_{H^{11}_x}
				\\
				&\ls
				\sum_{\Gamma}
				\|(\p\phi)\Gamma\p_t\phi(t)\|_{H^{11}_x}
				+
				\sum_{\Gamma}
				\|\phi\,\p\Gamma\p_t\phi(t)\|_{H^{11}_x}
				\\
				&\ls
				\|\p\phi(t)\|_{L^\infty_x}
				\sum_{\Gamma}
				\|\Gamma\p_t\phi(t)\|_{H^{11}_x}
				+
				\sum_{\Gamma}
				\|\Gamma\p_t\phi(t)\|_{L^\infty_x}
				\|\p\phi(t)\|_{H^{11}_x}
				+\sum_{\Gamma}
				\|\phi\,\p\Gamma\p_t\phi(t)\|_{H^{11}_x}.
			\end{aligned}
		\end{equation}
		We now treat $\ds\sum_{\Gamma}
		\|\phi\,\p\Gamma\p_t\phi(t)\|_{H^{11}_x}$ in \eqref{eq:CVP-recovery-second-term-1}. Due to
		\[
		\Gamma
		=
		\sum_{\beta=0}^3 Z^\beta_\Gamma(t,x)\p_\beta,
		\qquad
		|Z^\beta_\Gamma(t,x)|\ls A(t,x),
		\qquad
		|\nabla Z^\beta_\Gamma(t,x)|\ls1,
		\qquad
		\nabla^{\ge2}Z^\beta_\Gamma=0,
		\]
		one has
		$\p_\alpha\Gamma\p_t\phi
		=\ds\sum_{\beta=0}^3Z^\beta_\Gamma
		\p_{\alpha\beta t}^3\phi
		+\ds\sum_{\beta=0}^3
		c_{\alpha\Gamma}^{\beta}
		\p_{\beta t}^2\phi$,
		where \(c_{\alpha\Gamma}^{\beta}\) are constants. Therefore,
		\begin{equation}
			\label{eq:CVP-recovery-phi-pGamma-ptphi}
			\begin{aligned}
				&\sum_{\Gamma}
				\|\phi\,\p\Gamma\p_t\phi(t)\|_{H^{11}_x}\\
				&\ls
				\sum_{\Gamma}
				\sum_{\alpha,\beta=0}^3
				\|Z^\beta_\Gamma\phi\,\p_{\alpha\beta t}^3\phi(t)\|_{H^{11}_x}
				+\sum_{\Gamma}
				\sum_{\alpha,\beta=0}^3
				\|\phi\,\p_{\beta t}^2\phi(t)\|_{H^{11}_x}\\
				&\ls
				\|A(t,\cdot)\nabla^{\le 11}(\phi\,\p^3\phi)(t)\|_{L^{2}_x}
				+\|\phi\,\p^3\phi(t)\|_{H^{10}_x}
				+\|\phi\,\p^2\phi(t)\|_{H^{11}_x}.
			\end{aligned}
		\end{equation}
		As in the derivation of
		\eqref{eq:CVP-Apartial-R3-LP-start}-\eqref{eq:CVP-Apartial-R3-final}, we have
		\begin{equation}
			\label{eq:CVP-recovery-phi-2}
			\|A(t,\cdot)\nabla^{\le 11}(\phi\,\p^3\phi)(t)\|_{L^{2}_x}
			\ls
			\ve_1
			\|\p\phi(t)\|_{H^N_x}.
		\end{equation}
		Combining \eqref{eq:CVP-recovery-second-term-1}-\eqref{eq:CVP-recovery-phi-2}, together with Lemma \ref{lem:CVP-higher-time-derivative-estimates}, one arrives at
		\begin{equation}
			\label{eq:CVP-recovery-second-term}
			\begin{aligned}
				&
				\sum_{\Gamma}
				\|\p(\phi\,\Gamma\p_t\phi)(t)\|_{H^{11}_x}
				\\
				&\ls
				\|\p\phi(t)\|_{L^\infty_x}
				\sum_{\Gamma}
				\|\Gamma\p_t\phi(t)\|_{H^{11}_x}
				+
				\sum_{\Gamma}
				\|\Gamma\p_t\phi(t)\|_{L^\infty_x}
				\|\p\phi(t)\|_{H^{11}_x}
				\\
				&\quad+
				\ve_1
				\|\p\phi(t)\|_{H^N_x}
				+
				\|\phi\,\p^3\phi(t)\|_{H^{10}_x}
				+
				\|\phi\,\p^2\phi(t)\|_{H^{11}_x}
				\\
				&\ls
				\ve_1
				\sum_{\Gamma}
				\|\p\Gamma\phi(t)\|_{H^{11}_x}
				+
				\ve_1
				\|\p\phi(t)\|_{H^N_x}.
			\end{aligned}
		\end{equation}
		Collecting
		\eqref{eq:CVP-Gamma-recovery-identity},
		\eqref{eq:CVP-recovery-first-term} and
		\eqref{eq:CVP-recovery-second-term} yields
		\[
		\sum_{\Gamma}
		\|\p\Gamma\phi(t)\|_{H^{11}_x}
		\ls
		\sum_{\Gamma}
		\|\p\Gamma V(t)\|_{H^{11}_x}
		+
		\ve_1
		\sum_{\Gamma}
		\|\p\Gamma\phi(t)\|_{H^{11}_x}
		+
		\ve_1
		\|\p\phi(t)\|_{H^N_x}.
		\]
		Due to the smallness of \(\ve_1>0\), one has
		\begin{equation}
			\label{eq:CVP-Gamma-phi-recovery-final}
			\sum_{\Gamma}
			\|\p\Gamma\phi(t)\|_{H^{11}_x}
			\ls
			\sum_{\Gamma}
			\|\p\Gamma V(t)\|_{H^{11}_x}
			+
			\ve_1
			\|\p\phi(t)\|_{H^N_x}.
		\end{equation}
		It follows from
		\eqref{eq:CVP-Gamma-V-final} and
		$\|\p\phi(t)\|_{H^N_x}\le\ve_1(1+t)^\eta$
		that
		\[
		\sum_{\Gamma}
		\|\p\Gamma\phi(t)\|_{H^{11}_x}
		\ls
		\ve
		+
		\ve_1^3(1+t)^\eta
		+
		\ve_1^2(1+t)^\eta.
		\]
		Therefore, $\ds\sup_{0\le t\le T}
		(1+t)^{-\eta}
		\ds\sum_{\Gamma}
		\|\p\Gamma\phi(t)\|_{H^{11}_x}
		\ls
		\ve+\ve_1^2+\ve_1^3$ holds,
		which implies \eqref{eq:CVP-one-Gamma-energy-conclusion}.
	\end{proof}

	\section{Weighted pointwise estimates for $\phi$}
	\label{sec:weighted-pointwise-estimates-I}
	
	In this section, based on the equation \eqref{eq:box-V-schematic-null}, we derive the weighted pointwise estimates for $\phi$ and
	$\p\phi$.

	\begin{lemma}[{\bf Weighted pointwise estimates for $\phi$}]
		\label{lem:CVP-weighted-pointwise-estimate-new}
		Let \(\phi\) be a smooth solution to
		\eqref{eq:Chaplygin-Cauchy} on \([0,T]\times\mathbb R^3\).
		Assume that the initial data satisfy \eqref{YHC-4} and
		\eqref{eq:CVP-bootstrap-parameter-range}-\eqref{eq:CVP-bootstrap-Strichartz-Gamma}
		hold. Suppose that \(\ve_1>0\) is small. Then
		\begin{equation}
			\label{eq:CVP-weighted-conclusion-new}
			\begin{aligned}
				&
				\big\|
				AB^{\frac{\mu}{4}}\phi
				\big\|_{L^\infty_{t,x}}
				+
				\sum_{|a|\le 7}
				\big\|
				AB^{\frac{\mu}{4}}
				\nabla^a\p\phi
				\big\|_{L^\infty_{t,x}}
				\ls
				\ve+\ve_1^2+\ve_1^3.
			\end{aligned}
		\end{equation}
	\end{lemma}
	
	\begin{proof}
		Applying Lemma~\ref{lem:weighted-Linfty-L2} to \eqref{eq:box-V-schematic-null}, with \(\mu/4\)
		instead of \(\mu\), we have
		\begin{equation}
			\label{eq:CVP-linear-estimate-new}
			\begin{aligned}
				\|AB^{\frac{\mu}{4}}V\|_{L^\infty_{t,x}}
				+
				\sum_{|a|\le 7}
				\big\|
				AB^{\frac{\mu}{4}}
				\nabla^a\p V
				\big\|_{L^\infty_{t,x}}
				&\ls
				\sum_{|a|\le 10}
				\big\|
				\langle x\rangle^{1+\frac{\mu}{4}}
				\nabla^a
				\p V(0)
				\big\|_{L^2_x}
				\\
				&\quad+
				\sum_{|a|\le 10}
				\big\|
				A^{1+\frac{\mu}{4}}
				\nabla^a\mathcal R(\phi)
				\big\|_{L^1_tL^2_x}.
			\end{aligned}
		\end{equation}
	It follows from $\|f\|_{L^\infty(\R^3)}\ls\|f\|_{\dot{H}^1(\R^3)}+\|f\|_{\dot{H}^2(\R^3)}$ for $f\in \mathcal{S}'_h(\R^3)$
     and $N\ge 15$ that
		\begin{equation}
			\label{eq:CVP-initial-bound-new}
			\begin{aligned}
				\sum_{|a|\le 10}
				\big\|
				\langle x\rangle^{1+\frac{\mu}{4}}
				\nabla^a
				\p V(0)
				\big\|_{L^2_x}
				&\ls
				\big\|
				\langle x\rangle^{1+\frac{\mu}{4}}
				\nabla^{\le 10}
				\p \phi(0)
				\big\|_{L^2_x}
				(1+\big\|
				\p \phi(0)
				\big\|_{H^{11}_x})
				\\
				&\quad+\|\phi(0)\|_{L^\infty_x}
				\big\|
				\langle x\rangle^{1+\frac{\mu}{4}}
				\nabla^{\le 10}
				\p \p_t\phi(0)
				\big\|_{L^2_x}\\
				&\ls
				\big\|
				\langle x\rangle^{1+\frac{\mu}{4}}
				\nabla^{\le 11}
				\p \phi(0)
				\big\|_{L^2_x}
				(1+\big\|
				\p \phi(0)
				\big\|_{H^{11}_x})\\
				&\ls
				\ve.
			\end{aligned}
		\end{equation}
		
		We next treat $\ds\sum_{|a|\le 10}
				\big\|
				A^{1+\frac{\mu}{4}}
				\nabla^a\mathcal R(\phi)
				\big\|_{L^1_tL^2_x}$ in \eqref{eq:CVP-linear-estimate-new}. Define
		\begin{equation}
			\label{eq:CVP-exterior-and-interior-regions}
			\mathcal E_t=\big\{x\in\R^3:r\ge \f{1+t}{2}\big\}, \quad \mathcal I_t=\big\{x\in\R^3:r\le \f{1+t}{2}\big\}.
		\end{equation}
		In $\mathcal{E}_t$, one has
		\begin{equation}
			\label{eq:CVP-exterior-r}
			r^{-1}\ls A^{-1},
		\end{equation}
		while in $\mathcal{I}_t$, due to $r\le 1+t-r$, one can obtain
		\begin{equation}
			\label{eq:CVP-interior-B=A}
			B=1+|t-r|\le A = 1+t-r+2r\le 3(1+t-r)\le 3B.
		\end{equation}

		Set
		\begin{equation}
			\label{eq:CVP-R3-splitting-new}
			\mathcal R_3(\phi)
			=\mathcal R_{3,0}(\phi)
			+\mathcal R_{3,1}(\phi),
		\end{equation}
		where
		\begin{equation}
			\label{eq:CVP-R30-R31-definition-new}
			\begin{aligned}
				\mathcal R_{3,0}(\phi)
				&=
				\mathcal N_3(\phi)
				-
				2\p_t\phi Q_0(\p_t\phi,\phi)
				=
				N^{\alpha\beta\gamma\delta}\p_{\alpha\beta}\phi \p_\gamma\phi \p_\delta\phi,
				\\
				\mathcal R_{3,1}(\phi)
				&=
				-2\phi\,\p_tQ_0(\p_t\phi,\phi).
			\end{aligned}
		\end{equation}
		It is pointed out that \(\mathcal R_{3,0}\) is a cubic nonlinearity with null form (see \eqref{eq:null-condition-tensors}),
and \(\mathcal R_{3,1}\) contains the solution \(\phi\) itself.
		
		We first estimate \(\mathcal R_{3,0}\) and claim that
		\begin{equation}
			\label{eq:CVP-R30-time-estimate-new}
			\begin{aligned}
				&
				\sum_{|a|\le 10}
				\big\|
				A^{1+\frac{\mu}{4}}(t,\cdot)
				\nabla^a\mathcal R_{3,0}(\phi)(t)
				\big\|_{L^2_x}
			\ls
				\ve_1^2
				(1+t)^{-1-\frac{\mu}{4}}
				\big(
				\|\p\phi(t)\|_{H^N_x}
				+
				\sum_{\Gamma}
				\|\p\Gamma\phi(t)\|_{H^{11}_x}
				\big).
			\end{aligned}
		\end{equation}
		
		Indeed, for \(|a|\le 10\), one has
		\begin{equation}\label{eq:CVP-R30-Leibniz-new}
			\begin{aligned}
				\nabla^a\mathcal R_{3,0}(\phi)
				=\sum_{\substack{a_1+a_2+a_3=a}}
				C_{a_1,a_2,a_3}
				N^{\alpha\beta\gamma\delta}
				\p_{\alpha\beta}^2\nabla^{a_1}\phi\,
				\p_\gamma\nabla^{a_2}\phi\,
				\p_\delta\nabla^{a_3}\phi,
			\end{aligned}
		\end{equation}
		where \(|a_1|+|a_2|+|a_3|\le 10\), and at least two of
		\(|a_1|,|a_2|,|a_3|\) are not larger than $5$. Hence this implies
that for each term in \eqref{eq:CVP-R30-Leibniz-new}, its \(L^2_x\) estimate can be controlled by the \(L^\infty_x\)-norm
of 	two factors and the \(L^2_x\)-norm of the remaining factor.
		
		In $\mathcal E_t=\{x\in\R^3:r\ge (1+t)/2\}$,
		it follows from Lemma~\ref{lem:null} that
		\begin{equation}
			\label{eq:CVP-R30-exterior-null-Leibniz-new}
			\begin{aligned}
				|\nabla^a\mathcal R_{3,0}|
				\ls
				\sum_{a_1+a_2+a_3=a}
				\big[
				&
				|\bar\p\p\nabla^{a_1}\phi|\,
				|\p\nabla^{a_2}\phi|\,
				|\p\nabla^{a_3}\phi|
				\\
				&+
				|\p^2\nabla^{a_1}\phi|\,
				|\bar\p\nabla^{a_2}\phi|\,
				|\p\nabla^{a_3}\phi|
				\\
				&+
				|\p^2\nabla^{a_1}\phi|\,
				|\p\nabla^{a_2}\phi|\,
				|\bar\p\nabla^{a_3}\phi|
				\\
				&+
				r^{-1}
				|\p\nabla^{a_1}\phi|\,
				|\p\nabla^{a_2}\phi|\,
				|\p\nabla^{a_3}\phi|
				\big].
			\end{aligned}
		\end{equation}
		We next treat  the four terms on the right-hand side of \eqref{eq:CVP-R30-exterior-null-Leibniz-new}  separately.
		
		For the first term in \eqref{eq:CVP-R30-exterior-null-Leibniz-new}, define
		\[
		I_{1,a_1,a_2,a_3}
		=
		\big\|
		\mathbf 1_{\mathcal E_t}
		A^{1+\frac{\mu}{4}}
		|\bar\p\p\nabla^{a_1}\phi|\,
		|\p\nabla^{a_2}\phi|\,
		|\p\nabla^{a_3}\phi|
		\big\|_{L^2_x}.
		\]
		If \(|a_1|\le 5\), then at least one of \(|a_2|\) and \(|a_3|\) is less than
		\(5\). Without loss of generality, \(|a_2|\le 5\)
is assumed. Then by \eqref{eq:CVP-bootstrap-pointwise-phi}, Lemma \ref{lem:CVP-derived-good-derivative-estimates}, $B^{1-\f\mu2}\le A^{1-\f\mu2}$ for $\mu\in(0,\f12)$, one has
		\begin{equation}
			\label{eq:CVP-R30-I1-low-good-new}
			\begin{aligned}
				I_{1,a_1,a_2,a_3}
				&\ls
				\big\|
				\mathbf 1_{\mathcal E_t}
				A^{1+\frac{\mu}{4}}
				\ve_1
				\big(
				A^{-1-\frac{\mu}{2}}
				+
				A^{-2}B^{1-\frac{\mu}{4}}
				\big)
				\ve_1A^{-1}B^{-\frac{\mu}{4}}
				|\p\nabla^{a_3}\phi|
				\big\|_{L^2_x}
				\\
				&\ls
				\ve_1^2
				(1+t)^{-1-\frac{\mu}{4}}
				\|\p\nabla^{a_3}\phi(t)\|_{L^2_x}
				\\
				&\ls
				\ve_1^2
				(1+t)^{-1-\frac{\mu}{4}}
				\|\p\phi(t)\|_{H^N_x}.
			\end{aligned}
		\end{equation}

If \(|a_1|>5\), then \(|a_2|\le 5\) and \(|a_3|\le 5\). Using
		Lemma~\ref{lem:good-derivative-one-Gamma},
		\([\Gamma,\partial]\in{\rm span}\{\partial_0,\partial_1,\partial_2,\partial_3\}\),
		\eqref{eq:CVP-bootstrap-pointwise-phi}, Lemma \ref{lem:CVP-higher-time-derivative-estimates} and $B\le A$, we
arrive at
		\begin{equation}
			\label{eq:CVP-R30-I1-high-good-new}
			\begin{aligned}
				I_{1,a_1,a_2,a_3}
				&\ls
				\ve_1^2
				\big\|
				\mathbf 1_{\mathcal E_t}
				A^{1+\frac{\mu}{4}}
				A^{-2}B^{-\frac{\mu}{2}}
				\big(
				A^{-1}|\p\nabla^{a_1}\Gamma\phi|
				+
				A^{-1}B|\p^2\nabla^{\le a_1}\phi|
				\big)
				\big\|_{L^2_x}
				\\
				&\ls
				\ve_1^2
				(1+t)^{-2+\frac{\mu}{4}}
				\sum_{\Gamma}
				\|\p\Gamma\phi(t)\|_{H^{10}_x}
				+
				\ve_1^2
				(1+t)^{-1-\frac{\mu}{4}}
				\|\p\phi(t)\|_{H^N_x}.
			\end{aligned}
		\end{equation}

		The second and third terms in
		\eqref{eq:CVP-R30-exterior-null-Leibniz-new} can be estimated analogously. Namely, set
		\[
		I_{2,a_1,a_2,a_3}
		=
		\big\|
		\mathbf 1_{\mathcal E_t}
		A^{1+\frac{\mu}{4}}
		|\p^2\nabla^{a_1}\phi|\,
		|\bar\p\nabla^{a_2}\phi|\,
		|\p\nabla^{a_3}\phi|
		\big\|_{L^2_x}
		\]
		and
		\[
		I_{3,a_1,a_2,a_3}
		=
		\big\|
		\mathbf 1_{\mathcal E_t}
		A^{1+\frac{\mu}{4}}
		|\p^2\nabla^{a_1}\phi|\,
		|\p\nabla^{a_2}\phi|\,
		|\bar\p\nabla^{a_3}\phi|
		\big\|_{L^2_x}.
		\]
		Then
		\begin{equation}
			\label{eq:CVP-R30-I23-new}
			\begin{aligned}
				I_{2,a_1,a_2,a_3}
				+
				I_{3,a_1,a_2,a_3}
				\ls
				\ve_1^2
				(1+t)^{-1-\frac{\mu}{4}}
				\|\p\phi(t)\|_{H^N_x}
				+\ve_1^2
				(1+t)^{-2+\frac{\mu}{4}}
				\sum_{\Gamma}
				\|\p\Gamma\phi(t)\|_{H^{11}_x}.
			\end{aligned}
		\end{equation}
		
		For the fourth term in \eqref{eq:CVP-R30-exterior-null-Leibniz-new},  set
		\[
		I_{4,a_1,a_2,a_3}
		=
		\big\|
		\mathbf 1_{\mathcal E_t}
		A^{1+\frac{\mu}{4}}
		r^{-1}
		|\p\nabla^{a_1}\phi|\,
		|\p\nabla^{a_2}\phi|\,
		|\p\nabla^{a_3}\phi|
		\big\|_{L^2_x}.
		\]
		Due to \(r^{-1}\ls A^{-1}\) on
		\(\mathcal E_t\), one has
		\begin{equation}
			\label{eq:CVP-R30-rminus-new}
			\begin{aligned}
				I_{4,a_1,a_2,a_3}
				&\ls
				\ve_1^2
				\big\|
				\mathbf 1_{\mathcal E_t}
				A^{1+\frac{\mu}{4}}
				A^{-1}
				A^{-2}B^{-\frac{\mu}{2}}
				|\p\nabla^{\max \{a_i\}}\phi|
				\big\|_{L^2_x}
				\\
				&\ls
				\ve_1^2
				(1+t)^{-2+\frac{\mu}{4}}
				\|\p\phi(t)\|_{H^N_x}.
			\end{aligned}
		\end{equation}
		Combining \eqref{eq:CVP-R30-I1-low-good-new}-\eqref{eq:CVP-R30-rminus-new} and summing over
		\(a_1+a_2+a_3=a\), we then have that for \(|a|\le 10\),
		\begin{equation}
			\label{eq:CVP-R30-exterior-final-new}
			\begin{aligned}
				&
				\sum_{|a|\le 10}
				\big\|
				\mathbf 1_{\mathcal E_t}
				A^{1+\frac{\mu}{4}}
				\nabla^a\mathcal R_{3,0}(\phi)(t)
				\big\|_{L^2_x}\\
				&
				\ls
				\ve_1^2
				(1+t)^{-1-\frac{\mu}{4}}
				\|\p\phi(t)\|_{H^N_x}
				+
				\ve_1^2
				(1+t)^{-2+\frac{\mu}{4}}
				\sum_{\Gamma}
				\|\p\Gamma\phi(t)\|_{H^{11}_x}.
			\end{aligned}
		\end{equation}
		
		In $\mathcal I_t$,  \(B\sim A\) holds. Therefore,
		\begin{equation}
			\label{eq:CVP-R30-interior-new}
			\begin{aligned}
				\big\|
				\mathbf 1_{\mathcal I_t}
				A^{1+\frac{\mu}{4}}
				\p^2\nabla^{a_1}\phi\,
				\p\nabla^{a_2}\phi\,
				\p\nabla^{a_3}\phi
				\big\|_{L^2_x}
				&\ls
				\ve_1^2
				\big\|
				\mathbf 1_{\mathcal I_t}
				A^{1+\frac{\mu}{4}}
				A^{-2}B^{-\frac{\mu}{2}}
				|\p^{\le2}\nabla^{\max \{a_i\}}\phi|
				\big\|_{L^2_x}\\
				&
				\ls
				\ve_1^2
				(1+t)^{-1-\frac{\mu}{4}}
				\|\p\phi(t)\|_{H^N_x}.
			\end{aligned}
		\end{equation}
		Hence
		\begin{equation}
			\label{eq:CVP-R30-interior-final-new}
			\begin{aligned}
				\sum_{|a|\le 10}
				\big\|
				\mathbf 1_{\mathcal I_t}
				A^{1+\frac{\mu}{4}}
				\nabla^a\mathcal R_{3,0}(\phi)(t)
				\big\|_{L^2_x}
				\ls
				\ve_1^2
				(1+t)^{-1-\frac{\mu}{4}}
				\|\p\phi(t)\|_{H^N_x}.
			\end{aligned}
		\end{equation}
		Note that $\mu \in (0,1/2)$. Combining \eqref{eq:CVP-R30-exterior-final-new} and \eqref{eq:CVP-R30-interior-final-new} yields \eqref{eq:CVP-R30-time-estimate-new}.

		Secondly, we deal with \(\mathcal R_{3,1}\). It is claimed that
		\begin{equation}
			\label{eq:CVP-R31-time-estimate-new}
			\begin{aligned}
				&
				\sum_{|a|\le 10}
				\big\|
				A^{1+\frac{\mu}{4}}(t,\cdot)
				\nabla^a\mathcal R_{3,1}(\phi)(t)
				\big\|_{L^2_x}
				\ls
				\ve_1^2
				(1+t)^{-1-\frac{\mu}{4}}
				\big(
				\|\p\phi(t)\|_{H^N_x}
				+\sum_{\Gamma}
				\|\p\Gamma\phi(t)\|_{H^{11}_x}
				\big).
			\end{aligned}
		\end{equation}
		Indeed, note that
		\begin{equation}
			\label{eq:CVP-R31-expanded-new}
			\mathcal R_{3,1}(\phi)
			=-2\phi Q_0(\p_t^2\phi,\phi)
			-2\phi Q_0(\p_t\phi,\p_t\phi).
		\end{equation}
		Let
		\[
		(U_1,W_1)=(\p_t^2\phi,\phi),
		\qquad
		(U_2,W_2)=(\p_t\phi,\p_t\phi).
		\]
		For \(|a|\le 10\), one has
		\begin{equation}
			\label{eq:CVP-R31-Leibniz-new}
			\begin{aligned}
				\nabla^a\mathcal R_{3,1}(\phi)
				=-2\sum_{\nu=1}^{2}
				\sum_{a_0+a_1+a_2=a}
				C_{a_0,a_1,a_2}
				\nabla^{a_0}\phi\,
				Q_0(\nabla^{a_1}U_\nu,\nabla^{a_2}W_\nu).
			\end{aligned}
		\end{equation}
		
		In $\mathcal{E}_t$, it follows from Lemma \ref{lem:null} that
		\begin{equation}\label{eq:CVP-R31-exterior-null-new}
			\begin{aligned}
				|\nabla^a\mathcal R_{3,1}|
				\ls
				\sum_{\nu=1}^2
				\sum_{a_0+a_1+a_2=a}
				|\nabla^{a_0}\phi|
				\big(
				|\bar\p\nabla^{a_1}U_\nu|\,
				|\p\nabla^{a_2}W_\nu|
				+
				|\p\nabla^{a_1}U_\nu|\,
				|\bar\p\nabla^{a_2}W_\nu|
				\big).
			\end{aligned}
		\end{equation}
		It suffices to treat the first term in \eqref{eq:CVP-R31-exterior-null-new} since the second term can be analogously estimated.
		
		Set
		\[
		J_{a_0,a_1,a_2}^{\nu}
		=
		\big\|
		\mathbf 1_{\mathcal E_t}
		A^{1+\frac{\mu}{4}}
		|\nabla^{a_0}\phi|\,
		|\bar\p\nabla^{a_1}U_\nu|\,
		|\p\nabla^{a_2}W_\nu|
		\big\|_{L^2_x}.
		\]
		If \(|a_0|>5\), then \(|a_1|\le 5\) and \(|a_2|\le 5\). Therefore, it follows from \eqref{eq:CVP-bootstrap-pointwise-phi}, Lemma \ref{lem:CVP-derived-good-derivative-estimates} and $B^{1-\f\mu2}\le A^{1-\f\mu2}$ for $\mu\in(0,\f12)$ that
		\begin{equation}
			\label{eq:CVP-R31-high-phi-new}
			\begin{aligned}
				J_{a_0,a_1,a_2}^{\nu}
				&\ls
				\ve_1^2
				\big\|
				\mathbf 1_{\mathcal E_t}
				A^{1+\frac{\mu}{4}}
				A^{-1}B^{-\frac{\mu}{4}}
				\big(
				A^{-1-\frac{\mu}{2}}
				+
				A^{-2}B^{1-\frac{\mu}{4}}
				\big)
				|\nabla^{a_0}\phi|
				\big\|_{L^2_x}\\
				&\ls
				\ve_1^2
				(1+t)^{-1-\frac{\mu}{4}}
				\|\p\phi(t)\|_{H^N_x}.
			\end{aligned}
		\end{equation}
		If \(|a_0|\le 5\) and \(|a_1|\le 5\), then by \eqref{eq:CVP-bootstrap-pointwise-phi} and Lemma \ref{lem:CVP-derived-good-derivative-estimates}, we get
		\begin{equation}
			\label{eq:CVP-R31-low-good-new}
			\begin{aligned}
				J_{a_0,a_1,a_2}^{\nu}
				&\ls
				\ve_1^2
				\big\|
				\mathbf 1_{\mathcal E_t}
				A^{1+\frac{\mu}{4}}
				A^{-1}B^{-\frac{\mu}{4}}
				\big(
				A^{-1-\frac{\mu}{2}}
				+
				A^{-2}B^{1-\frac{\mu}{4}}
				\big)
				|\p\nabla^{a_2}W_\nu|
				\big\|_{L^2_x}
				\\
				&\ls
				\ve_1^2
				(1+t)^{-1-\frac{\mu}{4}}
				\|\p\phi(t)\|_{H^N_x}.
			\end{aligned}
		\end{equation}
		If \(|a_0|\le 5\) and \(|a_1|>5\), then
		\(|a_2|\le 5\) holds.
It follows from Lemma~\ref{lem:good-derivative-one-Gamma}, Lemma~\ref{lem:CVP-higher-time-derivative-estimates-Gamma}
and Lemma~\ref{lem:CVP-higher-time-derivative-estimates} that
		\begin{equation}
			\label{eq:CVP-R31-high-good-new}
			\begin{aligned}
				J_{a_0,a_1,a_2}^{\nu}
				&\ls
				\ve_1^2
				\big\|
				\mathbf 1_{\mathcal E_t}
				A^{1+\frac{\mu}{4}}
				A^{-2}B^{-\frac{\mu}{2}}
				\big(
				A^{-1}|\Gamma\nabla^{a_1}U_\nu|
				+
				A^{-1}B|\p\nabla^{a_1}U_\nu|
				\big)
				\big\|_{L^2_x}
				\\
				&\ls
				\ve_1^2
				(1+t)^{-2+\frac{\mu}{4}}
				\sum_{\Gamma}
				\|\p\Gamma\phi(t)\|_{H^{11}_x}
				+
				\ve_1^2
				(1+t)^{-1-\frac{\mu}{4}}
				\|\p\phi(t)\|_{H^N_x}.
			\end{aligned}
		\end{equation}

		Combining \eqref{eq:CVP-R31-high-phi-new}-\eqref{eq:CVP-R31-high-good-new} yields
		\begin{equation}
			\label{eq:CVP-R31-exterior-final-new}
			\begin{aligned}
				&
				\sum_{|a|\le 10}
				\big\|
				\mathbf 1_{\mathcal E_t}
				A^{1+\frac{\mu}{4}}
				\nabla^a\mathcal R_{3,1}(\phi)(t)
				\big\|_{L^2_x}
				\\
				&
				\ls
				\ve_1^2
				(1+t)^{-2+\frac{\mu}{4}}
				\sum_{\Gamma}
				\|\p\Gamma\phi(t)\|_{H^{11}_x}
				+
				\ve_1^2
				(1+t)^{-1-\frac{\mu}{4}}
				\|\p\phi(t)\|_{H^N_x}.
			\end{aligned}
		\end{equation}
		
		In  $\mathcal I_t$,  \(B\sim A\) holds. Then one has
		\begin{equation}
			\label{eq:CVP-R31-interior-new}
			\begin{aligned}
				\big\|
				\mathbf 1_{\mathcal I_t}
				A^{1+\frac{\mu}{4}}
				\nabla^{a_0}\phi\,
				\p\nabla^{a_1}U_\nu\,
				\p\nabla^{a_2}W_\nu
				\big\|_{L^2_x}
				&\ls
				\ve_1^2
				\big\|
				\mathbf 1_{\mathcal I_t}
				A^{1+\frac{\mu}{4}}
				A^{-2}B^{-\frac{\mu}{2}}
				\mathcal Z_{a_0,a_1,a_2}^{\nu}
				\big\|_{L^2_x}\\
				&\ls
				\ve_1^2
				(1+t)^{-1-\frac{\mu}{4}}
				\|\p\phi(t)\|_{H^N_x},
			\end{aligned}
		\end{equation}
		where \(\mathcal Z_{a_0,a_1,a_2}^{\nu}\) stands for the unique factor among
		\(\nabla^{a_0}\phi\), \(\p\nabla^{a_1}U_\nu\) and
		\(\p\nabla^{a_2}W_\nu\),  which is put in \(L^2_x\) (when $a_0=0$, $\nabla^{a_0}\phi=\phi$ is never  put into  \(L^2_x\)).
		Therefore,
		\begin{equation}
			\label{eq:CVP-R31-interior-final-new}
			\begin{aligned}
				\sum_{|a|\le 10}
				\big\|
				\mathbf 1_{\mathcal I_t}
				A^{1+\frac{\mu}{4}}
				\nabla^a\mathcal R_{3,1}(\phi)(t)
				\big\|_{L^2_x}
				\ls
				\ve_1^2
				(1+t)^{-1-\frac{\mu}{4}}
				\|\p\phi(t)\|_{H^N_x}.
			\end{aligned}
		\end{equation}
	Owing to $\mu\in(0,\f12)$, combining \eqref{eq:CVP-R31-exterior-final-new} and \eqref{eq:CVP-R31-interior-final-new}
		yields \eqref{eq:CVP-R31-time-estimate-new}.

		Collecting \eqref{eq:CVP-R30-time-estimate-new} and
		\eqref{eq:CVP-R31-time-estimate-new}, we arrive at
		\begin{equation}
			\label{eq:CVP-R3-time-estimate-new}
			\begin{aligned}
				\sum_{|a|\le 10}
				\big\|
				A^{1+\frac{\mu}{4}}(t,\cdot)
				\nabla^a\mathcal R_3(\phi)(t)
				\big\|_{L^2_x}
				\ls
				\ve_1^2
				(1+t)^{-1-\frac{\mu}{4}}
				\big(
				\|\p\phi(t)\|_{H^N_x}
				+
				\sum_{\Gamma}
				\|\p\Gamma\phi(t)\|_{H^{11}_x}
				\big).
			\end{aligned}
		\end{equation}
		Using the energy bootstrap assumptions
		\eqref{eq:CVP-bootstrap-energy-ordinary} and
		\eqref{eq:CVP-bootstrap-energy-Gamma}, together with
		\(\eta\le \mu/100\), we obtain
		\begin{equation}
			\label{eq:CVP-R3-time-estimate-integrable-new}
			\begin{aligned}
				\sum_{|a|\le 10}
				\big\|
				A^{1+\frac{\mu}{4}}(t,\cdot)
				\nabla^a\mathcal R_3(\phi)(t)
				\big\|_{L^2_x}
				\ls
				\ve_1^3
				(1+t)^{-1-\frac{\mu}{4}+\eta}
				\ls
				\ve_1^3
				(1+t)^{-1-\frac{\mu}{16}}.
			\end{aligned}
		\end{equation}
		
		Analogously,
		\begin{equation}
			\label{eq:CVP-R45-time-estimate-new}
			\begin{aligned}
				\sum_{|a|\le 10}
				\big\|
				A^{1+\frac{\mu}{4}}(t,\cdot)
				\nabla^a
				\big(
				\mathcal R_4(\phi)+\mathcal R_5(\phi)
				\big)(t)
				\big\|_{L^2_x}
				\ls
				(\ve_1^4+\ve_1^5)
				(1+t)^{-1-\frac{\mu}{16}}.
			\end{aligned}
		\end{equation}
		It follows from
		\eqref{eq:CVP-R3-time-estimate-integrable-new},
		\eqref{eq:CVP-R45-time-estimate-new} and the smallness of \(\ve_1\) that
		\begin{equation}
			\label{eq:CVP-R-time-estimate-new}
			\sum_{|a|\le 10}
			\big\|
			A^{1+\frac{\mu}{4}}(t,\cdot)
			\nabla^a\mathcal R(\phi)(t)
			\big\|_{L^2_x}
			\ls
			\ve_1^3
			(1+t)^{-1-\frac{\mu}{16}}.
		\end{equation}
		Therefore,
		\begin{equation}
			\label{eq:CVP-R-L1-estimate-new}
			\sum_{|a|\le 10}
			\big\|
			A^{1+\frac{\mu}{4}}
			\nabla^a\mathcal R(\phi)
			\big\|_{L^1_tL^2_x}
			\ls
			\ve_1^3.
		\end{equation}
		Substituting \eqref{eq:CVP-initial-bound-new} and
		\eqref{eq:CVP-R-L1-estimate-new} into
		\eqref{eq:CVP-linear-estimate-new} yields
		\begin{equation}
			\label{eq:CVP-V-final-bound-new}
			\begin{aligned}
				&
				\|AB^{\frac{\mu}{4}}V\|_{L^\infty_{t,x}}
				+
				\sum_{|a|\le 7}
				\big\|
				AB^{\frac{\mu}{4}}
				\nabla^a\p V
				\big\|_{L^\infty_{t,x}}
				\ls
				\ve+\ve_1^3.
			\end{aligned}
		\end{equation}
		It remains to recover \(\phi\) from \(V\). Due to $\phi=V+\phi\p_t\phi$
		and
		$\big\|AB^{\frac{\mu}{4}}\phi\p_t\phi
			\big\|_{L^\infty_{t,x}}\ls\ve_1^2,$
		one has
		\begin{equation}
			\label{eq:CVP-zero-order-phi-new}
			\big\|
			AB^{\frac{\mu}{4}}\phi
			\big\|_{L^\infty_{t,x}}
			\ls
			\ve+\ve_1^2+\ve_1^3.
		\end{equation}
		For \(|a|\le 7\), it holds that
		\begin{equation}\label{eq:CVP-recovery-Leibniz-new}
			\begin{aligned}
				\left|
				\nabla^a\p(\phi\p_t\phi)
				\right|
				\ls
				\sum_{a_1+a_2=a}
				\left(
				|\nabla^{a_1}\p\phi|
				|\nabla^{a_2}\p_t\phi|
				+
				|\nabla^{a_1}\phi|
				|\nabla^{a_2}\p\p_t\phi|
				\right).
			\end{aligned}
		\end{equation}
		Note that all terms in \eqref{eq:CVP-recovery-Leibniz-new} can be directly controlled by the bootstrap assumptions and
		Lemma~\ref{lem:CVP-higher-time-derivative-estimates} except the following top-order terms
		\begin{equation}\label{YHC-23-0}
		\begin{aligned}
		AB^{\frac{\mu}{4}}
		\phi\,\nabla^a\p_\alpha\p_t\phi
		\quad\text{for $|a|=7$ and $0\le\alpha\le3$.}
		\end{aligned}
		\end{equation}
		
Next, we treat the terms in \eqref{YHC-23-0}.
Set $G_\alpha=\p_\alpha\p_t\phi$. It follows from \eqref{eq:CVP-bootstrap-pointwise-phi} and
Lemma~\ref{lem:CVP-higher-time-derivative-estimates} that
		\begin{equation}
			\label{eq:CVP-interpolation-low}
			\|\nabla^{6}G_\alpha(t)\|_{L^\infty_x}
			\ls
			\ve_1(1+t)^{-1}.
		\end{equation}

		On the other hand, by \eqref{eq:CVP-bootstrap-energy-ordinary} and 	Lemma~\ref{lem:CVP-higher-time-derivative-estimates},
one has
		\begin{equation}
			\label{eq:CVP-interpolation-high}
			\|G_\alpha(t)\|_{H^{N-1}_x}
			\ls
			\|\p\phi(t)\|_{H^N_x}
			\ls
			\ve_1(1+t)^\eta.
		\end{equation}
		By the Gagliardo-Nirenberg inequality $\|\nabla f \|_{L^\infty(\R^3)} \ls \|f\|_{L^\infty(\R^3)}^{1-\theta}\|f\|_{\dot{H}^M(\R^3)}^\theta $
with  $1=\theta(M-\f32)$ and $\theta \in [1/M,1]$,
		we obtain that for $M=8$ and $\theta=\f{2}{13}$,
		\begin{equation}
			\label{eq:CVP-top-interpolation}
			\begin{aligned}
				\|\nabla^7G_\alpha(t)\|_{L^\infty_x}
				&\ls
				\|\nabla^{6}G_\alpha(t)\|_{L^\infty_x}^{1-\theta}
				\|G_\alpha(t)\|_{H^{14}_x}^{\theta}
				\\
				&\ls
				\left(
				\ve_1(1+t)^{-1}
				\right)^{1-\theta}
				\left(
				\ve_1(1+t)^\eta
				\right)^\theta
				\\
				&=
				\ve_1
				(1+t)^{-1+\f{15}{13}\eta}
				\ls
				\ve_1,
			\end{aligned}
		\end{equation}
		where $N\ge 15$ and $\eta\in(0,\f{1}{200})$ are used.
		Therefore, for \(|a|=7\),
		\begin{equation}
			\label{eq:CVP-top-recovery-product}
			\begin{aligned}
				\big\|
				AB^{\frac{\mu}{4}}
				\phi\,\nabla^a\p_\alpha\p_t\phi
				\big\|_{L^\infty_x}
				\ls
				\big\|
				AB^{\frac{\mu}{4}}\phi
				\big\|_{L^\infty_x}
				\|\nabla^7G_\alpha(t)\|_{L^\infty_x}
				\ls
				\ve_1^2.
			\end{aligned}
		\end{equation}
		This, together with the estimates on the related lower order derivative terms, yields
		\begin{equation}
			\label{eq:CVP-recovery-error-new}
			\sum_{|a|\le 7}
			\big\|
			AB^{\frac{\mu}{4}}
			\nabla^a\p
			\big(
			\phi\p_t\phi
			\big)
			\big\|_{L^\infty_{t,x}}
			\ls
			\ve_1^2.
		\end{equation}
		
		Combining
		\eqref{eq:CVP-V-final-bound-new},
		\eqref{eq:CVP-zero-order-phi-new} and
		\eqref{eq:CVP-recovery-error-new}, we arrive at
		\[
		\big\|
		AB^{\frac{\mu}{4}}\phi
		\big\|_{L^\infty_{t,x}}
		+
		\sum_{|a|\le 7}
		\big\|
		AB^{\frac{\mu}{4}}
		\nabla^a\p\phi
		\big\|_{L^\infty_{t,x}}
		\ls
		\ve+\ve_1^2+\ve_1^3,
		\]
		which proves \eqref{eq:CVP-weighted-conclusion-new}.
	\end{proof}

	\section{Weighted pointwise estimates for $\Gamma\phi$}
	\label{sec:weighted-pointwise-estimates-II}
	
	In this section, based on the equation \eqref{eq:box-V-schematic-null}, we establish the weighted pointwise estimates for $\Gamma\phi$
	and $\p\Gamma\phi$.

	\begin{lemma}[{\bf Weighted pointwise estimates  for $\Gamma\phi$}]
		\label{lem:CVP-weighted-pointwise-Gamma-closure}
		Assume that the initial data satisfy \eqref{YHC-4} and
		\eqref{eq:CVP-bootstrap-parameter-range}-\eqref{eq:CVP-bootstrap-Strichartz-Gamma}
		hold on \([0,T]\). Suppose that \(\ve_1>0\) is small. Then
		\begin{equation}
			\label{eq:CVP-weighted-Gamma-conclusion}
			\begin{aligned}
				\sum_{\Gamma}
				\big\|
				A^{\frac{\mu}{2}}
				\Gamma\phi
				\big\|_{L^\infty([0,T]\times\R^3)}
				+
				\sum_{\Gamma}
				\sum_{|a|\le 6}
				\big\|
				A^{\frac{\mu}{2}}
				\p\nabla^a\Gamma\phi
				\big\|_{L^\infty([0,T]\times\R^3)}
				\ls
				\ve+\ve_1^2+\ve_1^3.
			\end{aligned}
		\end{equation}
		\end{lemma}
	
	\begin{proof}
		It follows from  \eqref{eq:box-V-schematic-null} that
		\begin{equation}
			\label{eq:CVP-Gamma-V-equation-pointwise-II}
			\Box\Gamma V
			=\Gamma\mathcal R(\phi).
		\end{equation}
		Applying Lemma~\ref{lem:weak-weighted-Linfty-L2} to \(\Gamma V\), with
		\(\mu/2\) in place of \(\mu\), we have
		\begin{equation}\label{eq:CVP-Gamma-V-linear-pointwise-II}
			\begin{aligned}
				&\sum_{\Gamma}
				\big\|
				A^{\frac{\mu}{2}}\Gamma V
				\big\|_{L^\infty_{t,x}}
				+
				\sum_{\Gamma}
				\sum_{|a|\le 6}
				\big\|
				A^{\frac{\mu}{2}}
				\p\nabla^a\Gamma V
				\big\|_{L^\infty_{t,x}}
				\\
				&\ls
				\sum_{\Gamma}
				\sum_{|a|\le 9}
				\big\|
				\langle x\rangle^{\frac{\mu}{2}}
				\nabla^a\p\Gamma V(0)
				\big\|_{L^2_x}
				+
				\sum_{\Gamma}
				\sum_{|a|\le 9}
				\big\|
				A^{\frac{\mu}{2}}
				\nabla^a\Gamma\mathcal R(\phi)
				\big\|_{L^1_tL^2_x}.
			\end{aligned}
		\end{equation}
		In addition, direct computation yields
		\begin{equation}
			\label{eq:CVP-Gamma-V-initial-pointwise-II}
			\sum_{\Gamma}
			\sum_{|a|\le 9}
			\big\|
			\langle x\rangle^{\frac{\mu}{2}}
			\nabla^a\p\Gamma V(0)
			\big\|_{L^2_x}
			\ls
			\sum_{|a|\le 9}
			\big\|
			\langle x\rangle^{1+\frac{\mu}{2}}
			\nabla^a\p^{\le 1}\p V(0)
			\big\|_{L^2_x}
			\ls\ve.
		\end{equation}
		
		We next estimate $\ds\sum_{\Gamma}
		\sum_{|a|\le 9}\big\|A^{\frac{\mu}{2}}
		\nabla^a\Gamma\mathcal R_3(\phi)
		\big\|_{L^1_tL^2_x}$ in \eqref{eq:CVP-Gamma-V-linear-pointwise-II}.
		Note that
		\begin{equation}
			\label{eq:CVP-R3-schematic-structure-II}
			\mathcal R_3(\phi)
			=\sum
			\p^{\le2}\phi\cdot
			\p^{\le2}\phi\cdot
			\p^{\le2}\p\phi.
		\end{equation}
		It follows from \eqref{eq:CVP-R3-schematic-structure-II} that  for \(|a|\le 9\),
		\begin{equation}
			\label{eq:CVP-Gamma-R3-schematic-Leibniz-II}
			\begin{aligned}
				|\nabla^a\Gamma\mathcal R_3(\phi)|
				\ls
				&\sum_{|a_1|+|a_2|+|a_3|\le |a|}
				\Big(
				|\p^{\le2}\nabla^{a_1}\Gamma\phi|\,
				|\p^{\le2}\nabla^{a_2}\phi|\,
				|\p^{\le2}\nabla^{a_3}\p\phi|
				\\
				&\qquad\qquad\qquad +
				|\p^{\le2}\nabla^{a_1}\phi|\,
				|\p^{\le2}\nabla^{a_2}\Gamma\phi|\,
				|\p^{\le2}\nabla^{a_3}\p\phi|
				\\
				&\qquad\qquad\qquad +
				|\p^{\le2}\nabla^{a_1}\phi|\,
				|\p^{\le2}\nabla^{a_2}\phi|\,
				|\p^{\le2}\nabla^{a_3}\p\Gamma\phi|
				\Big)
				\\
				&+
				\sum_{|b_1|+|b_2|+|b_3|\le |a|}
				|\p^{\le2}\nabla^{b_1}\phi|\,
				|\p^{\le2}\nabla^{b_2}\phi|\,
				|\p^{\le2}\nabla^{b_3}\p\phi|.
			\end{aligned}
		\end{equation}

		We now treat each term on the right-hand side of \eqref{eq:CVP-Gamma-R3-schematic-Leibniz-II}.

If the derivative order of  $\p^{\le2}\nabla^{a_1}\Gamma\phi$ for $\Gamma\phi$ is top, that is $|a_1|>4$, then $|a_2|\le 4$ and $|a_3| \le 4$. By  \eqref{eq:CVP-bootstrap-energy-Gamma}, \eqref{eq:CVP-bootstrap-pointwise-phi},
Lemma \ref{lem:CVP-higher-time-derivative-estimates} and Lemma \ref{lem:CVP-higher-time-derivative-estimates-Gamma}, we have
		\begin{equation}
			\label{eq:CVP-Gamma-R3-high-Gamma-II}
			\begin{aligned}
				&
				\big\|
				A^{\frac{\mu}{2}}
				\p^{\le2}\nabla^{a_1}\Gamma\phi\,
				\p^{\le2}\nabla^{a_2}\phi\,
				\p^{\le2}\nabla^{a_3}\p\phi
				\big\|_{L^2_x}
				\\
				&\ls
				\ve_1^2
				\left\|
				A^{\frac{\mu}{2}}A^{-2}B^{-\frac{\mu}{2}}
				\p^{\le2}\p\nabla^{a_1}\Gamma\phi
				\right\|_{L^2_x}
				\\
				&\ls
				\ve_1^2
				(1+t)^{-2+\frac{\mu}{2}}
				\sum_{\Gamma}
				\|\p^{\le2}\p\Gamma\phi(t)\|_{H^{9}_x}
				\\
				&\ls
				\ve_1^3
				(1+t)^{-2+\frac{\mu}{2}+\eta}.
			\end{aligned}
		\end{equation}
		The same estimate holds when the derivative order of \(\p^{\le2}\nabla^{a_2}\Gamma\phi\) for $\Gamma\phi$ or
		\(\p^{\le2}\nabla^{a_3}\p\Gamma\phi\) for $\p\Gamma\phi$ is top. For \(0<\mu<1/2\) and \(\eta\le\mu/100\), it holds that
		\[
		\int_0^T
		(1+t)^{-2+\frac{\mu}{2}+\eta}\,dt
		\ls 1.
		\]
		Thus,
\begin{equation}\label{YHC-23}
\begin{aligned}
&\big\|A^{\frac{\mu}{2}}
\p^{\le2}\nabla^{a_1}\Gamma\phi\,
\p^{\le2}\nabla^{a_2}\phi\,
\p^{\le2}\nabla^{a_3}\p\phi
\big\|_{L_t^1L^2_x}
+\big\|A^{\frac{\mu}{2}}\p^{\le2}\nabla^{a_1}\phi\,
\p^{\le2}\nabla^{a_2}\Gamma\phi\,
\p^{\le2}\nabla^{a_3}\p\phi\big\|_{L_t^1L^2_x}\\
&\quad +\big\|A^{\frac{\mu}{2}}\p^{\le2}\nabla^{a_1}\phi\,
\p^{\le2}\nabla^{a_2}\phi\,
\p^{\le2}\nabla^{a_3}\p\Gamma\phi\|_{L_t^1L^2_x}\ls\ve_1^3.
\end{aligned}
		\end{equation}
		
		When the derivative order for \(\Gamma\phi\) or \(\p\Gamma\phi\) of the related terms in the first summation
of \eqref{eq:CVP-Gamma-R3-schematic-Leibniz-II} is low, that is $|a_3|\le 4$ in \eqref{eq:CVP-Gamma-R3-low-Gamma-II}, it follows from
		\eqref{eq:CVP-bootstrap-energy-ordinary}, \eqref{eq:CVP-bootstrap-pointwise-phi}, \eqref{eq:CVP-bootstrap-Strichartz-Gamma}, Lemma \ref{lem:CVP-higher-time-derivative-estimates} 	and Lemma \ref{lem:CVP-higher-time-derivative-estimates-Gamma} that
		\begin{equation}
			\label{eq:CVP-Gamma-R3-low-Gamma-II}
			\begin{aligned}
				&\big\|A^{\frac{\mu}{2}}
				\p^{\le3}\nabla^{a_1}\phi\,
				\p^{\le3}\nabla^{a_2}\phi\,
				\p^{\le3}\nabla^{a_3}\Gamma\phi
				\big\|_{L^1_tL^2_x}\\
				&\ls
				\ve_1
				\int_0^T
				(1+t)^{-1+\frac{\mu}{6}}
				\|\p^{\le2}\p\phi(t)\|_{H^{9}_x}
				\big\|
				A^{\frac{\mu}{3}}
				\p^{\le3}\nabla^{a_3}\Gamma\phi(t)
				\big\|_{L^\infty_x}
				\,dt
				\\
				&\ls
				\ve_1^2
				\big\|
				(1+t)^{-1+\frac{\mu}{6}+\eta}
				\big\|_{L^2(0,T)}
				\sum_{\Gamma}
				\big(
				\big\|
				A^{\frac{\mu}{3}}\Gamma\phi
				\big\|_{L^2_tL^\infty_x}
				+\sum_{|c|\le 6}
				\big\|
				A^{\frac{\mu}{3}}
				\p\nabla^c\Gamma\phi
				\big\|_{L^2_tL^\infty_x}
				\big)
				\\
				&\ls
				\ve_1^3,
			\end{aligned}
		\end{equation}
		where $2\left(-1+\frac{\mu}{6}+\eta\right)<-1$ is used.
		
		On the other hand, the terms in the second summation of \eqref{eq:CVP-Gamma-R3-schematic-Leibniz-II}
are easier to be estimated. Therefore,
		\begin{equation}
			\label{eq:CVP-R3-source-final-II}
			\sum_{\Gamma}
			\sum_{|a|\le 9}
			\big\|
			A^{\frac{\mu}{2}}
			\nabla^a\Gamma\mathcal R_3(\phi)
			\big\|_{L^1_tL^2_x}
			\ls
			\ve_1^3.
		\end{equation}
		
	Analogously, it holds that	
		\begin{equation}
			\label{eq:CVP-R4-R5-source-final-II}
			\sum_{\Gamma}
			\sum_{|a|\le 9}
			\big\|
			A^{\frac{\mu}{2}}
			\nabla^a\Gamma\mathcal R_4(\phi)
			\big\|_{L^1_tL^2_x}
			+\sum_{\Gamma}
			\sum_{|a|\le 9}
			\big\|
			A^{\frac{\mu}{2}}
			\nabla^a\Gamma\mathcal R_5(\phi)
			\big\|_{L^1_tL^2_x}
			\ls
			\ve_1^4+\ve_1^5.
		\end{equation}
		Combining \eqref{eq:CVP-R3-source-final-II} and
		\eqref{eq:CVP-R4-R5-source-final-II} yields
		\begin{equation}
			\label{eq:CVP-Gamma-source-main-claim-II}
			\sum_{\Gamma}
			\sum_{|a|\le 9}
			\big\|
			A^{\frac{\mu}{2}}
			\nabla^a\Gamma\mathcal R(\phi)
			\big\|_{L^1_tL^2_x}\ls
			\ve_1^3.
		\end{equation}
		Substituting \eqref{eq:CVP-Gamma-V-initial-pointwise-II} and
		\eqref{eq:CVP-Gamma-source-main-claim-II} into
		\eqref{eq:CVP-Gamma-V-linear-pointwise-II}, we obtain
		\begin{equation}
			\label{eq:CVP-Gamma-V-pointwise-final-II}
			\begin{aligned}
				\sum_{\Gamma}
				\big\|
				A^{\frac{\mu}{2}}\Gamma V
				\big\|_{L^\infty_{t,x}}
				+\sum_{\Gamma}
				\sum_{|a|\le 6}
				\big\|
				A^{\frac{\mu}{2}}
				\p\nabla^a\Gamma V
				\big\|_{L^\infty_{t,x}}
				\ls
				\ve+\ve_1^3.
			\end{aligned}
		\end{equation}
		
		We next estimate \(\Gamma\phi\) from \(\Gamma V\). Note that
			\begin{equation}
			\label{eq:CVP-Gamma-zero-recovery-II}
			\Gamma\phi
			=\Gamma V+\Gamma(\phi\p_t\phi)
		\end{equation}
		and
		\begin{equation}
			\label{eq:CVP-Gamma-zero-recovery-error-II}
			\begin{aligned}
				\sum_{\Gamma}
				A^{\frac{\mu}{2}}|\Gamma(\phi\p_t\phi)|
				\ls
				\sum_{\Gamma}
				A^{\frac{\mu}{2}}
				\left(
				|\Gamma\phi|\,|\p_t\phi|
				+
				|\phi|\,|\p\Gamma\phi|
				+
				|\phi|\,|\p\phi|
				\right)
				\ls
				\ve_1^2.
			\end{aligned}
		\end{equation}
		For \(|a|\le 6\), one has
		\begin{equation}
			\label{eq:CVP-Gamma-recovery-Leibniz-II}
			\begin{aligned}
				|\nabla^a\p\Gamma(\phi\p_t\phi)|
				&\ls
				\sum_{|a_1|+|a_2|\le |a|}
				\Big(
				|\nabla^{a_1}\p\Gamma\phi|\,
				|\nabla^{a_2}\p\phi|
				+
				|\nabla^{a_1}\Gamma\phi|\,
				|\nabla^{a_2}\p^2\phi|
				\\
				&\qquad\qquad\qquad+
				|\nabla^{a_1}\p\phi|\,
				|\nabla^{a_2}\p\Gamma\phi|
				+
				|\nabla^{a_1}\phi|\,
				|\nabla^{a_2}\p^2\Gamma\phi|
				\Big)
				\\
				&+
				\sum_{|b|\le |a|+1}
				|\nabla^b\p\phi|
				\sum_{|c|\le |a|}
				|\nabla^c\p^2\phi|.
			\end{aligned}
		\end{equation}
		We point out that all terms in \eqref{eq:CVP-Gamma-recovery-Leibniz-II} can be
		controlled directly by the bootstrap assumptions and the higher order time
		derivative estimates except the top-order terms of the form
		\begin{equation}\label{YHC-26}
			\begin{aligned}
		A^{\frac{\mu}{2}}
		\phi\,\nabla^a\p_\alpha\p_t\Gamma\phi
		\quad\text{for $|a|=6$}.
		\end{aligned}
		\end{equation}
		We next handle the terms in \eqref{YHC-26}. As in
		Section~\ref{sec:weighted-pointwise-estimates-I}, set
		$G_{\alpha,\Gamma}
		=\p_\alpha\p_t\Gamma\phi$.
		It follows from \eqref{eq:CVP-bootstrap-pointwise-Gamma} and
		Lemma~\ref{lem:CVP-higher-time-derivative-estimates-Gamma} that
		\begin{equation}
			\label{eq:CVP-Gamma-interpolation-low-II}
			\|\nabla^{5}G_{\alpha,\Gamma}(t)\|_{L^\infty_x}
			\ls
			\ve_1(1+t)^{-\frac{\mu}{2}}.
		\end{equation}
		On the other hand, by \eqref{eq:CVP-bootstrap-energy-Gamma} and Lemma \ref{lem:CVP-higher-time-derivative-estimates-Gamma},
		\begin{equation}
			\label{eq:CVP-Gamma-interpolation-high-II}
			\|G_{\alpha,\Gamma}(t)\|_{H^{10}_x}
			\ls
			\sum_{\Gamma}
			\|\p\Gamma\phi(t)\|_{H^{11}_x}
			\ls
			\ve_1(1+t)^\eta.
		\end{equation}
		Interpolating between
		\eqref{eq:CVP-Gamma-interpolation-low-II} and
		\eqref{eq:CVP-Gamma-interpolation-high-II}, we have
		\begin{equation}
			\label{eq:CVP-Gamma-interpolation-top-II}
			\|\nabla^{6}G_{\alpha,\Gamma}(t)\|_{L^\infty_x}
			\ls
			\ve_1(1+t)^{-\frac{\mu}{3}}.
		\end{equation}
		Due to $|\phi(t,x)|
		\ls
		\ve_1 A^{-1}(t,x)B^{-\frac{\mu}{4}}(t,x)$,
		one has
		\begin{equation}
			\label{eq:CVP-Gamma-recovery-top-error-II}
			\begin{aligned}
				A^{\frac{\mu}{2}}
				|\phi\,\nabla^{6}\p_\alpha\p_t\Gamma\phi|
				\ls
				\ve_1^2
				A^{-1+\frac{\mu}{2}}
				B^{-\frac{\mu}{4}}
				(1+t)^{-\frac{\mu}{3}}
				\ls
				\ve_1^2.
			\end{aligned}
		\end{equation}
		Consequently,
		\begin{equation}
			\label{eq:CVP-Gamma-recovery-error-final-II}
			\begin{aligned}
				\sum_{\Gamma}
				\sum_{|a|\le 6}
				\big\|
				A^{\frac{\mu}{2}}
				\nabla^a\p\Gamma(\phi\p_t\phi)
				\big\|_{L^\infty_{t,x}}
				\ls
				\ve_1^2.
			\end{aligned}
		\end{equation}
		Combining \eqref{eq:CVP-Gamma-zero-recovery-II},
		\eqref{eq:CVP-Gamma-zero-recovery-error-II},
		\eqref{eq:CVP-Gamma-V-pointwise-final-II} and
		\eqref{eq:CVP-Gamma-recovery-error-final-II}, we arrive at
		\begin{equation}
			\label{eq:CVP-Gamma-final-before-absorb-II}
			\begin{aligned}
				\sum_{\Gamma}
				\big\|
				A^{\frac{\mu}{2}}\Gamma\phi
				\big\|_{L^\infty_{t,x}}
				+
				\sum_{\Gamma}
				\sum_{|a|\le 6}
				\big\|
				A^{\frac{\mu}{2}}
				\p\nabla^a\Gamma\phi
				\big\|_{L^\infty_{t,x}}
				\ls
				\ve+\ve_1^2+\ve_1^3,
				\end{aligned}
		\end{equation}
		which yields \eqref{eq:CVP-weighted-Gamma-conclusion}.
	\end{proof}

	\section{Weighted Strichartz estimates for $\Gamma\phi$}
	\label{sec:weighted-Strichartz-estimates}
	
	In this section, based on  the equation \eqref{eq:box-V-schematic-null}, we derive  the weighted
Strichartz estimates for $\Gamma\phi$ and $\p\Gamma\phi$.

		\begin{lemma}[{\bf Weighted Strichartz estimates for $\Gamma\phi$ and
	$\p\Gamma\phi$}]
		\label{lem:CVP-weighted-Strichartz-Gamma-closure}
		Assume that the initial data satisfy \eqref{YHC-4} and
		\eqref{eq:CVP-bootstrap-parameter-range}-\eqref{eq:CVP-bootstrap-Strichartz-Gamma}
		hold on \([0,T]\). Suppose that \(\ve_1>0\) is small. Then
		\begin{equation}
			\label{eq:CVP-weighted-Strichartz-Gamma-conclusion}
			\begin{aligned}
			\sum_{\Gamma}
			\big\|
			A^{\frac{\mu}{3}}
			\Gamma\phi
			\big\|_{L^2([0,T];L^\infty(\mathbb R^3))}
			+
			\sum_{\Gamma}
			\sum_{|a|\le 6}
			\big\|
			A^{\frac{\mu}{3}}
			\partial\nabla^a\Gamma\phi
			\big\|_{L^2([0,T];L^\infty(\mathbb R^3))}
				\ls
				\ve+\ve_1^2+\ve_1^3.
			\end{aligned}
		\end{equation}
		\end{lemma}
	
	\begin{proof}
		Note that
		\begin{equation}\label{eq:CVP-Gamma-V-equation-Strichartz}
			\Box\Gamma V
			=\Gamma\mathcal R(\phi).
			\end{equation}
		Applying Lemma~\ref{lem:weak-weighted-L2t-Linfty-L2} to \(\Gamma V\), with
		\(\mu/3\) in place of \(\f{\beta_1}2\), we have
		\begin{equation}
			\label{eq:CVP-Gamma-V-linear-Strichartz}
			\begin{aligned}
				&\sum_{\Gamma}
				\big\|
				A^{\frac{\mu}{3}}\Gamma V
				\big\|_{L^2_{t}L^\infty_x}
				+\sum_{\Gamma}
				\sum_{|a|\le 6}
				\big\|
				A^{\frac{\mu}{3}}
				\p\nabla^a\Gamma V
				\big\|_{L^2_{t}L^\infty_x}
				\\
				&\ls
				\sum_{\Gamma}
				\sum_{|a|\le 8}
				\big\|
				\langle x\rangle^{\mu}
				\nabla^a\p\Gamma V(0)
				\big\|_{L^2_x}
				+
				\sum_{\Gamma}
				\sum_{|a|\le 8}
				\big\|
				A^{\mu}
				\nabla^a\Gamma\mathcal R(\phi)
				\big\|_{L^1_tL^2_x}.
			\end{aligned}
		\end{equation}
		Although the proof procedure for \eqref{eq:CVP-weighted-Strichartz-Gamma-conclusion} is almost identical to that in Section~\ref{sec:weighted-pointwise-estimates-II}, we still present the full details for the reader's convenience.
		
		It follows from direct computation that
		\begin{equation}\label{eq:CVP-Gamma-V-initial-Strichartz}
			\sum_{\Gamma}
			\sum_{|a|\le 8}
			\left\|
			\langle x\rangle^{\mu}
			\nabla^a\p\Gamma V(0)
			\right\|_{L^2_x}
			\ls
			\sum_{|a|\le 8}
			\left\|
			\langle x\rangle^{1+\mu}
			\nabla^a\p^{\le 1}\p V(0)
			\right\|_{L^2_x}
			\ls
			\ve.
		\end{equation}
		
		We now treat  $\ds\sum_{\Gamma}\sum_{|a|\le 8}
		\left\|
		A^{\mu}
		\nabla^a\Gamma\mathcal R(\phi)
		\right\|_{L^1_tL^2_x}$.
		Note that for \(|a|\le 8\),
		\begin{equation}
			\label{eq:CVP-Gamma-R3-schematic-Leibniz-III}
			\begin{aligned}
				|\nabla^a\Gamma\mathcal R_3(\phi)|
				\ls
				&\sum_{|a_1|+|a_2|+|a_3|\le |a|}
				\Big(
				|\p^{\le2}\nabla^{a_1}\Gamma\phi|\,
				|\p^{\le2}\nabla^{a_2}\phi|\,
				|\p^{\le2}\nabla^{a_3}\p\phi|
				\\
				&\qquad\qquad\qquad +
				|\p^{\le2}\nabla^{a_1}\phi|\,
				|\p^{\le2}\nabla^{a_2}\Gamma\phi|\,
				|\p^{\le2}\nabla^{a_3}\p\phi|
				\\
				&\qquad\qquad\qquad +
				|\p^{\le2}\nabla^{a_1}\phi|\,
				|\p^{\le2}\nabla^{a_2}\phi|\,
				|\p^{\le2}\nabla^{a_3}\p\Gamma\phi|
				\Big)
				\\
				&+
				\sum_{|b_1|+|b_2|+|b_3|\le |a|}
				|\p^{\le2}\nabla^{b_1}\phi|\,
				|\p^{\le2}\nabla^{b_2}\phi|\,
				|\p^{\le2}\nabla^{b_3}\p\phi|.
			\end{aligned}
		\end{equation}

		If the derivative order of  $\p^{\le2}\nabla^{a_1}\Gamma\phi$ for $\Gamma\phi$ is top, that is $|a_1|>4$, then $|a_2|\le 4$ and $|a_3| \le 4$.
		By  \eqref{eq:CVP-bootstrap-energy-Gamma}, \eqref{eq:CVP-bootstrap-pointwise-phi},
Lemma \ref{lem:CVP-higher-time-derivative-estimates}
		and Lemma \ref{lem:CVP-higher-time-derivative-estimates-Gamma}, one has
		\begin{equation}
			\label{eq:CVP-Gamma-R3-high-Gamma-Strichartz}
			\begin{aligned}
				&
				\left\|
				A^{\mu}
				\p^{\le2}\nabla^{a_1}\Gamma\phi\,
				\p^{\le2}\nabla^{a_2}\phi\,
				\p^{\le2}\nabla^{a_3}\p\phi
				\right\|_{L^2_x}
				\\
				&\ls
				\ve_1^2
				\left\|
				A^{\mu}A^{-2}B^{-\frac{\mu}{2}}
				\p^{\le2}\p\nabla^{a_1}\Gamma\phi
				\right\|_{L^2_x}
				\\
				&\ls
				\ve_1^2
				(1+t)^{-2+\mu}
				\sum_{\Gamma}
				\|\p^{\le2}\p\Gamma\phi(t)\|_{H^{8}_x}
				\\
				&\ls
				\ve_1^3
				(1+t)^{-2+\mu+\eta}.
			\end{aligned}
		\end{equation}
The same estimate holds when the derivative order of \(\p^{\le2}\nabla^{a_2}\Gamma\phi\) for $\Gamma\phi$ or
		\(\p^{\le2}\nabla^{a_3}\p\Gamma\phi\) for $\p\Gamma\phi$  is top. For \(0<\mu<1/2\) and \(\eta\le\mu/100\), it holds that
		\[
		\int_0^T
		(1+t)^{-2+\mu+\eta}\,dt
		\ls 1.
		\]
		Thus,
\begin{equation}\label{YHC-27}
\begin{aligned}
&
				\big\|
				A^{\mu}
				\p^{\le2}\nabla^{a_1}\Gamma\phi\,
				\p^{\le2}\nabla^{a_2}\phi\,
				\p^{\le2}\nabla^{a_3}\p\phi
				\big\|_{L_t^1L^2_x}
+\big\|A^{\mu}\p^{\le2}\nabla^{a_1}\phi\,
\p^{\le2}\nabla^{a_2}\Gamma\phi\,
\p^{\le2}\nabla^{a_3}\p\phi\big\|_{L_t^1L^2_x}\\
&\quad +\big\|A^{\mu}\p^{\le2}\nabla^{a_1}\phi\,
				\p^{\le2}\nabla^{a_2}\phi\,
			\p^{\le2}\nabla^{a_3}\p\Gamma\phi\|_{L_t^1L^2_x}
			\ls\ve_1^3.
\end{aligned}
		\end{equation}

When the derivative order for \(\Gamma\phi\) or \(\p\Gamma\phi\) of the related terms in the first summation
of \eqref{eq:CVP-Gamma-R3-schematic-Leibniz-III} is low, that is $|a_3|\le 4$ in \eqref{eq:CVP-Gamma-R3-low-Gamma-Strichartz},
we also get
		\begin{equation}
			\label{eq:CVP-Gamma-R3-low-Gamma-Strichartz}
			\begin{aligned}
				&
				\left\|
				A^{\mu}
				\p^{\le3}\nabla^{a_1}\phi\,
				\p^{\le3}\nabla^{a_2}\phi\,
				\p^{\le3}\nabla^{a_3}\Gamma\phi
				\right\|_{L^1_tL^2_x}
				\\
				&\ls
				\ve_1
				\int_0^T
				(1+t)^{-1+\frac{2\mu}{3}}
				\|\p^{\le2}\p\phi(t)\|_{H^{8}_x}
				\big\|
				A^{\frac{\mu}{3}}
				\p^{\le3}\nabla^{a_3}\Gamma\phi(t)
				\big\|_{L^\infty_x}
				\,dt
				\\
				&\ls
				\ve_1^2
				\big\|
				(1+t)^{-1+\frac{2\mu}{3}+\eta}
				\big\|_{L^2(0,T)}
				\sum_{\Gamma}
				\big(
				\big\|
				A^{\frac{\mu}{3}}\Gamma\phi
				\big\|_{L^2_tL^\infty_x}
				+\sum_{|c|\le 6}
				\big\|
				A^{\frac{\mu}{3}}
				\p\nabla^c\Gamma\phi
				\big\|_{L^2_tL^\infty_x}
				\big)
				\\
				&\ls
				\ve_1^3,
			\end{aligned}
		\end{equation}
		where $2(-1+\frac{2\mu}{3}+\eta)<-1$ is used.
		
		Therefore, it holds that
		\begin{equation}
			\label{eq:CVP-R3-source-final-Strichartz}
			\sum_{\Gamma}
			\sum_{|a|\le 8}
			\left\|
			A^{\mu}
			\nabla^a\Gamma\mathcal R_3(\phi)
			\right\|_{L^1_tL^2_x}
			\ls
			\ve_1^3
		\end{equation}
		and analogously,
		\begin{equation}
			\label{eq:CVP-R4-R5-source-final-Strichartz}
			\sum_{\Gamma}
			\sum_{|a|\le 8}
			\left\|
			A^{\mu}
			\nabla^a\Gamma\mathcal R_4(\phi)
			\right\|_{L^1_tL^2_x}
			+\sum_{\Gamma}
			\sum_{|a|\le 8}
			\left\|
			A^{\mu}
			\nabla^a\Gamma\mathcal R_5(\phi)
			\right\|_{L^1_tL^2_x}
			\ls
			\ve_1^4+\ve_1^5.
		\end{equation}
		Combining \eqref{eq:CVP-R3-source-final-Strichartz} and
		\eqref{eq:CVP-R4-R5-source-final-Strichartz} yields
		\begin{equation}
			\label{eq:CVP-Gamma-source-main-claim-Strichartz}
			\sum_{\Gamma}
			\sum_{|a|\le 8}
			\left\|
			A^{\mu}
			\nabla^a\Gamma\mathcal R(\phi)
			\right\|_{L^1_tL^2_x}\ls
			\ve_1^3.
		\end{equation}
		Substituting \eqref{eq:CVP-Gamma-V-initial-Strichartz} and
		\eqref{eq:CVP-Gamma-source-main-claim-Strichartz} into
		\eqref{eq:CVP-Gamma-V-linear-Strichartz}, one has
		\begin{equation}
			\label{eq:CVP-Gamma-V-pointwise-final-Strichartz}
			\begin{aligned}
				\sum_{\Gamma}
				\big\|
				A^{\frac{\mu}{3}}\Gamma V
				\big\|_{L^2_{t}L^\infty_x}
				+
				\sum_{\Gamma}
				\sum_{|a|\le 6}
				\big\|
				A^{\frac{\mu}{3}}
				\p\nabla^a\Gamma V
				\big\|_{L^2_{t}L^\infty_x}
				\ls
				\ve+\ve_1^3.
			\end{aligned}
		\end{equation}
		
		It remains to recover \(\Gamma\phi\) from \(\Gamma V\). Note that
		$\Gamma\phi=\Gamma V+\Gamma(\phi\p_t\phi)$
		and
		\begin{equation}\label{eq:CVP-Gamma-zero-recovery-error-Strichartz}
			\begin{aligned}
				\sum_{\Gamma}
				A^{\frac{\mu}{3}}|\Gamma(\phi\p_t\phi)|
				&\ls\sum_{\Gamma}
				A^{\frac{\mu}{3}}
				\left(|\Gamma\phi|\,|\p_t\phi|
				+|\phi|\,|\p\Gamma\phi|
				+|\phi|\,|\p\phi|
				\right)\\
				&\ls\ve_1
				\sum_{\Gamma}
				A^{\frac{\mu}{3}}|\Gamma\phi|
				+\ve_1\sum_{\Gamma}
				A^{\frac{\mu}{3}}|\p\Gamma\phi|
				+\ve_1^2(1+t)^{-2}.
			\end{aligned}
		\end{equation}
		Therefore, by \eqref{eq:CVP-bootstrap-Strichartz-Gamma},
		\begin{equation}
			\label{GM-1}
			\sum_{\Gamma}
			\big\|
			A^{\frac{\mu}{3}}\Gamma(\phi\p_t\phi)
			\big\|_{L^2_{t}L^\infty_x}
			\ls
			\ve_1^2.
		\end{equation}
		For \(|a|\le 6\), it holds that
		\begin{equation}
			\label{eq:CVP-Gamma-recovery-Leibniz-Strichartz}
			\begin{aligned}
				|\nabla^a\p\Gamma(\phi\p_t\phi)|
				&\ls
				\sum_{|a_1|+|a_2|\le |a|}
				\big(|\nabla^{a_1}\p\Gamma\phi|\,
				|\nabla^{a_2}\p\phi|
				+|\nabla^{a_1}\Gamma\phi|\,
				|\nabla^{a_2}\p^2\phi|\\
				&\qquad\qquad\qquad+
				|\nabla^{a_1}\p\phi|\,
				|\nabla^{a_2}\p\Gamma\phi|
				+|\nabla^{a_1}\phi|\,
				|\nabla^{a_2}\p^2\Gamma\phi|
				\big)\\
				&+\sum_{|b|\le |a|+1}
				|\nabla^b\p\phi|
				\sum_{|c|\le |a|}
				|\nabla^c\p^2\phi|.
			\end{aligned}
		\end{equation}
		It is pointed out that all terms in \eqref{eq:CVP-Gamma-recovery-Leibniz-Strichartz} can be
		controlled directly by the bootstrap assumptions and the higher order time
		derivative estimates except the top-order derivative terms of the form
		\begin{equation}\label{YHC-29}
			\begin{aligned}
		A^{\frac{\mu}{3}}
		\phi\,\nabla^a\p_\alpha\p_t\Gamma\phi
		\quad\text{with $|a|=6$.}
		\end{aligned}
		\end{equation}
		We handle the terms in \eqref{YHC-29} by the same interpolation argument as in
		Section~\ref{sec:weighted-pointwise-estimates-II}. Recall that by \eqref{eq:CVP-Gamma-interpolation-top-II},
		\begin{equation}
			\label{eq:CVP-Gamma-interpolation-top-Strichartz}
			\|\nabla^{6}\p\p_t\Gamma\phi(t)\|_{L^\infty_x}
			\ls
			\ve_1(1+t)^{-\frac{\mu}{3}}.
		\end{equation}
		Since
		$|\phi(t,x)|
		\ls
		\ve_1 A^{-1}(t,x)B^{-\frac{\mu}{4}}(t,x)$,
		we get
		\begin{equation}
			\label{eq:CVP-Gamma-recovery-top-error-Strichartz}
			\begin{aligned}
				\|A^{\frac{\mu}{3}}
				\phi\,\nabla^{6}\p_\alpha\p_t\Gamma\phi\|_{L^2_tL^\infty_x}
				\ls
				\ve_1^2
				\|
				(1+t)^{-1}\|_{L^2_t}
				\ls
				\ve_1^2 .
			\end{aligned}
		\end{equation}
		Consequently,
		\begin{equation}
			\label{eq:CVP-Gamma-recovery-error-final-Strichartz}
			\begin{aligned}
				\sum_{\Gamma}
				\sum_{|a|\le 6}
				\big\|
				A^{\frac{\mu}{3}}
				\nabla^a\p\Gamma(\phi\p_t\phi)
				\big\|_{L^2_tL^\infty_x}
				\ls
				\ve_1^2.
			\end{aligned}
		\end{equation}
		Combining
		\eqref{eq:CVP-Gamma-V-pointwise-final-Strichartz}, \eqref{GM-1} and
		\eqref{eq:CVP-Gamma-recovery-error-final-Strichartz} yields
		\begin{equation}
			\label{eq:CVP-Gamma-final-before-absorb-Strichartz}
			\begin{aligned}
				\sum_{\Gamma}
				\big\|
				A^{\frac{\mu}{3}}\Gamma\phi
				\big\|_{L^2_tL^\infty_x}
				+\sum_{\Gamma}
				\sum_{|a|\le 6}
				\big\|
				A^{\frac{\mu}{3}}
				\p\nabla^a\Gamma\phi
				\big\|_{L^2_tL^\infty_x}
				\ls
				\ve+\ve_1^2+\ve_1^3,
			\end{aligned}
		\end{equation}
		which yields
		\eqref{eq:CVP-weighted-Strichartz-Gamma-conclusion}.
	\end{proof}

	\section{Proofs of Theorems \ref{thm:main-global}-\ref{thm:main-global-1}}\label{sec:Proof-Theorem}
	
	At first, based on Sections \ref{sec:CVP-bootstrap-assumptions}-\ref{sec:weighted-Strichartz-estimates}, we start to prove Theorem \ref{thm:main-global}.

	\subsection{Proof of Theorem \ref{thm:main-global}}\label{YHC-31}

	We shall improve the constant \(\varepsilon_1\) in
	\eqref{eq:CVP-bootstrap-energy-ordinary}-\eqref{eq:CVP-bootstrap-Strichartz-Gamma}
	to \(\varepsilon_1/2\).
	From Lemmas~\ref{lem:Chaplygin-energy-estimate}, \ref{lem:CVP-one-Gamma-energy-closure}, \ref{lem:CVP-weighted-pointwise-estimate-new}, \ref{lem:CVP-weighted-pointwise-Gamma-closure} and \ref{lem:CVP-weighted-Strichartz-Gamma-closure}, there exists a constant \(C_1\ge 1\) depending on $\mu$ such that
	\begin{equation}
		\label{eq:improve-energy-ordinary}
		\sup_{0\le t\le T}(1+t)^{-\eta}\|\partial\phi(t)\|_{H^N(\mathbb R^3)}
		\le C_1(\varepsilon+\varepsilon_1^2),
	\end{equation}
	\begin{equation}
		\label{eq:improve-energy-Gamma}
		\sup_{0\le t\le T}(1+t)^{-\eta}\sum_{\Gamma}\|\partial\Gamma\phi(t)\|_{H^{11}(\mathbb R^3)}
		\le C_1(\varepsilon+\varepsilon_1^2+\varepsilon_1^3),
	\end{equation}
	\begin{equation}
		\label{eq:improve-pointwise-phi}
		\bigl\|AB^{\frac{\mu}{4}}\phi\bigr\|_{L^\infty([0,T]\times\mathbb R^3)}
		+\sum_{|a|\le 7}\bigl\|AB^{\frac{\mu}{4}}\nabla^a\partial\phi\bigr\|_{L^\infty([0,T]\times\mathbb R^3)}
		\le C_1(\varepsilon+\varepsilon_1^2+\varepsilon_1^3),
	\end{equation}
	\begin{equation}
		\label{eq:improve-pointwise-Gamma}
		\begin{aligned}
			\sum_{\Gamma}\bigl\|A^{\frac{\mu}{2}}\Gamma\phi\bigr\|_{L^\infty([0,T]\times\mathbb R^3)}
			+\sum_{\Gamma}\sum_{|a|\le 6}\bigl\|A^{\frac{\mu}{2}}\partial\nabla^a\Gamma\phi\bigr\|_{L^\infty([0,T]\times\mathbb R^3)}
			\le C_1(\varepsilon+\varepsilon_1^2+\varepsilon_1^3),
		\end{aligned}
	\end{equation}
	\begin{equation}
		\label{eq:improve-Strichartz-Gamma}
		\begin{aligned}
			\sum_{\Gamma}\bigl\|A^{\frac{\mu}{3}}\Gamma\phi\bigr\|_{L^2([0,T];L^\infty(\mathbb R^3))}
			+\sum_{\Gamma}\sum_{|a|\le 6}\bigl\|A^{\frac{\mu}{3}}\partial\nabla^a\Gamma\phi\bigr\|_{L^2([0,T];L^\infty(\mathbb R^3))}
			\le C_1(\varepsilon+\varepsilon_1^2+\varepsilon_1^3).
		\end{aligned}
	\end{equation}
	Choosing \(\varepsilon_1 = 4C_1\varepsilon_0\) with
	\(\varepsilon_0 = (32C_1^2)^{-1}\) and taking
	\(\varepsilon\le\varepsilon_0\), one has
	\[
	\sup_{0\le t\le T}(1+t)^{-\eta}\|\partial\phi(t)\|_{H^N(\mathbb R^3)}
	\le \frac{\varepsilon_1}{2},
	\]
	and analogously the left-hand sides of
	\eqref{eq:improve-energy-Gamma}-\eqref{eq:improve-Strichartz-Gamma}
	can be bounded by \(\varepsilon_1/2\).
	This, together with the local existence of classical solution to
	\eqref{eq:Chaplygin-Cauchy} and the continuity argument, ensures that
	\eqref{eq:Chaplygin-Cauchy} admits a unique solution
	\(\phi\in C([0,\infty);H^{N+1}(\mathbb R^3))\cap C^1([0,\infty);H^{N}(\mathbb R^3))\).

	From the bootstrap assumptions
	\eqref{eq:CVP-bootstrap-energy-ordinary}-\eqref{eq:CVP-bootstrap-Strichartz-Gamma}, we obtain the following global bounds:
	
	\begin{equation}
		\sup_{t\ge 0}(1+t)^{-\eta}\|\partial\phi(t)\|_{H^N(\mathbb R^3)}
		\le C\varepsilon,
		\label{eq:global-energy-ordinary}
	\end{equation}
	
	\begin{equation}
		\sup_{t\ge 0}(1+t)^{-\eta}\sum_{\Gamma}\|\partial\Gamma\phi(t)\|_{H^{11}(\mathbb R^3)}
		\le C\varepsilon,
		\label{eq:global-energy-Gamma}
	\end{equation}
	
	\begin{equation}
		\bigl\|AB^{\frac{\mu}{4}}\phi\bigr\|_{L^\infty([0,\infty)\times\mathbb R^3)}
		+\sum_{|a|\le 7}\bigl\|AB^{\frac{\mu}{4}}\nabla^a\partial\phi\bigr\|_{L^\infty([0,\infty)\times\mathbb R^3)}
		\le C\varepsilon,
		\label{eq:global-pointwise-phi}
	\end{equation}
	
	\begin{equation}
		\sum_{\Gamma}\bigl\|A^{\frac{\mu}{2}}\Gamma\phi\bigr\|_{L^\infty([0,\infty)\times\mathbb R^3)}
		+\sum_{\Gamma}\sum_{|a|\le 6}\bigl\|A^{\frac{\mu}{2}}\partial\nabla^a\Gamma\phi\bigr\|_{L^\infty([0,\infty)\times\mathbb R^3)}
		\le C\varepsilon,
		\label{eq:global-pointwise-Gamma}
	\end{equation}
	
	\begin{equation}
		\sum_{\Gamma}\bigl\|A^{\frac{\mu}{3}}\Gamma\phi\bigr\|_{L^2([0,\infty);L^\infty(\mathbb R^3))}
		+\sum_{\Gamma}\sum_{|a|\le 6}\bigl\|A^{\frac{\mu}{3}}\partial\nabla^a\Gamma\phi\bigr\|_{L^2([0,\infty);L^\infty(\mathbb R^3))}
		\le C\varepsilon.
		\label{eq:global-Strichartz-Gamma}
	\end{equation}
	Then the proof of
		Theorem \ref{thm:main-global} has been finished.

\subsection{Proof of Theorem \ref{thm:main-global-1}}\label{YHC-33}

Based on Theorem \ref{thm:main-global}, we next prove Theorem \ref{thm:main-global-1}.

\noindent{\bf Proof of Theorem \ref{thm:main-global-1}.} By \eqref{YHC-2}, one has

\begin{equation}\label{YHC-34}
\rho-\bar\rho
=
-\bar\rho
\big(
\partial_t\phi+\frac12|\nabla\phi|^2
\big)
+
O\big(
\left|2\partial_t\phi+|\nabla\phi|^2\right|^2
\big),
\end{equation}	
and
\begin{equation}\label{YHC-35}
v=\nabla\phi.	
\end{equation}	
It follows from \eqref{YHC-34}-\eqref{YHC-35}, \eqref{YHCCC-1} and \eqref{YHCCC-3} that \eqref{YHCCC-01} and \eqref{YHCCC-03} are easily obtained.
In addition, it follows from \eqref{YHCCC-03} that $|\rho-\bar\rho|
\le C\varepsilon$ and further $\rho>\f{\bar\rho}{2}$ for small $\ve$.
Therefore, Theorem \ref{thm:main-global-1} is proved.

	\appendix
	
	\section{Introduction to $A_2$ weights}\label{section a}
	For the reader's convenience, in this section we give a brief introduction to some $A_2$ weight inequalities.
	For further properties of $A_p$ weights, we refer the reader to \cite[Chapter 7]{Grafakos} and \cite[Chapter V]{Stein}.
	\begin{definition}
		A non-negative weight function $w$ is said to be of class $A_2$ if
		\begin{equation}\label{def:A_2 weight}
			[w]_{A_2}=\sup_{\text{all cubes $Q$ in } \mathbb{R}^3}\big(\frac{1}{|Q|} \int_Q w(x) dx\big)\big(\frac{1}{|Q|} \int_Q w(x)^{-1} dx\big)<\infty. 	
		\end{equation}
		
	\end{definition}
	\begin{lemma}\label{lem:w{x} in A2}
		$|x|^\alpha,\ \w{x}^{\alpha} \in A_2(\R^3)$ if and only if $\alpha\in(-3,3)$.
	\end{lemma}
	\begin{proof}
		See \cite[Lemma A.2.]{GaoLiYin2026}
	\end{proof}

	\begin{definition}\label{def:7.4.1}
		Let $0 < \delta$ and $A < \infty$. A function $K(x,y)$ defined for $x,y \in \mathbb{R}^n$ with $x \neq y$ is called a standard kernel (with constants $\delta$ and $A$) if
		\begin{equation}\label{eq:7.4.1}
			|K(x,y)| \leq \frac{A}{|x-y|^n}, \quad x \neq y,
		\end{equation}
		and whenever $|x-x'| \leq \frac12\max\left(|x-y|,|x'-y|\right)$, it holds
		\begin{equation}\label{eq:7.4.2}
			|K(x,y) - K(x',y)| \leq \frac{A|x-x'|^\delta}{\left(|x-y|+|x'-y|\right)^{n+\delta}};
		\end{equation}
		when $|y-y'| \leq \frac12\max\left(|x-y|,|x-y'|\right)$, it holds
		\begin{equation}\label{eq:7.4.3}
			|K(x,y) - K(x,y')| \leq \frac{A|y-y'|^\delta}{\left(|x-y|+|x-y'|\right)^{n+\delta}}.
		\end{equation}
		The class of all kernels that satisfy \eqref{eq:7.4.1}-\eqref{eq:7.4.3} is denoted by $SK(\delta,A)$.
	\end{definition}

	\begin{definition}\label{def:7.4.2}
		Let $0 < \delta$, $A < \infty$ and $K\in SK(\delta,A)$. A Calder\'{o}n-Zygmund operator $T$ associated with $K$ is defined as
		\begin{equation}\label{eq:7.4.5}
			T(f)(x) = \int_{\mathbb{R}^n} K(x,y)f(y)\,dy,
		\end{equation}
		which fulfills
		\begin{equation}\label{eq:7.4.4}
			\|T(f)\|_{L^2} \leq B\|f\|_{L^2}.
		\end{equation}
		Hence, all Calder\'{o}n-Zygmund operators $T$ form the space $CZO(\delta,A,B)$.
	\end{definition}
	\begin{lemma}\label{lem:A2}
		Let $A,B,\beta > 0$ and $T\in CZO(\beta,A,B)$. Then there is a constant $C = C(\beta,[w]_{A_2})$ such that
		for all $w \in A_2$ and $f \in L^2(w)$,
		\[
		\|T(f)\|_{L^2(w)} \leq C\,(A+B)\|f\|_{L^2(w)}.
		\]
	\end{lemma}
	\begin{proof}
		See \cite[Theorem 7.4.6]{Grafakos}.
	\end{proof}
	\begin{lemma}\label{lem:riesz L2}
		Let $w \in A_2$ and $\Rj{j}$ be the $j$-th Riesz transformation for $j = 1, 2, 3$. There exists a constant $C = C (\Atwo{w})> 0$ such that
		for any $f \in L^2(w)$,
		\begin{equation}\label{YHCCC-42}
			\begin{aligned}
				\Ltw{\Rj{j}f} \leq C (\Atwo{w}) \Ltw{f}.
			\end{aligned}
		\end{equation}
	\end{lemma}
	\begin{proof}
		Since the Riesz transformation is a Calder\'on-Zygmund operator, \eqref{YHCCC-42} is shown by Lemma \ref{lem:A2}.
	\end{proof}
	
	\begin{lemma}\label{lem:weighted bernstein}
		Let $w \in A_2$, and $k\in \Z$. Then there exists a constant $C = C (\Atwo{w})> 0$ such that for any $f \in L^2(w)$,
		\[
		\Ltw{\pk f} \leq C (\Atwo{w}) 2^{-k}\Ltw{\pk \nabla f}
		\]
		and
		\begin{equation}\label{YHCCC-43}
			\begin{split}
				\Ltw{\pk \nabla f} \leq C (\Atwo{w}) 2^{k}\Ltw{\pk f}.
			\end{split}
		\end{equation}
	\end{lemma}
	\begin{proof}
		When $k=0$, it follows from Bernstein's inequality that
		\begin{equation*}
			\|\dot{P}_0 f\|_{L^2(\R^3)}=\|\dot{P}_0 |\nabla|^{-1}R\cdot\nabla f\|_{L^2(\R^3)}\ls  \|\dot{P}_0\nabla f\|_{L^2(\R^3)}.
		\end{equation*}
		Since $\dot{P}_0|\nabla|^{-1}R$ is a Calder\'{o}n-Zygmund operator, then by Lemma \ref{lem:A2}, one has
		
		\[
		\Ltw{\dot{P}_0 f} \leq C_1(\Atwo{w}) \cdot \Ltw{\dot{P}_0 \nabla f}.
		\]
		In addition, due to $\Atwo{w(2^{-k}\cdot)}=\Atwo{w}$, then we arrive at
		\begin{equation}\label{YHCCC-44}
			\begin{split}
				\Ltw{\pk f}  = &2^{-\f32 k}\|\dot{P}_0(f(2^{-k}\cdot))\|_{L^2(w(2^{-k}\cdot))}\\
				\le& 2^{-\f32 k}C_1(\Atwo{w(2^{-k}\cdot)})\|\dot{P}_0\nabla(f(2^{-k}\cdot))\|_{L^2(w(2^{-k}\cdot))}\\
				\le& C (\Atwo{w}) 2^{-k}\Ltw{\pk \nabla f}.
			\end{split}
		\end{equation}
		Analogously, the proof of \eqref{YHCCC-43} can be completed as for \eqref{YHCCC-44}.
	\end{proof}

		\begin{lemma}[{\bf Weighted dyadic \(L^\infty\) estimates}]
		\label{lem:weighted-dyadic-Linfty-sum}
		Let
		\[
		A_\tau(x)=1+\tau+|x|,\qquad \tau\ge0.
		\]
		Then, for every \(k\ge-1\),
		\begin{equation}
			\label{eq:weighted-single-dyadic-Linfty}
			\|A_\tau P_k f\|_{L^\infty_x}
			\ls
			\|A_\tau f\|_{L^\infty_x}.
		\end{equation}
		Moreover, if $k \ge 0$, we have
		\begin{equation}
			\label{eq:weighted-dyadic-Linfty-sum}
			\|A_\tau P_k f\|_{L^\infty_x}
			\ls
			2^{-k}
			\|A_\tau\nabla P_k f\|_{L^\infty_x},
		\end{equation}
		where the implicit constants in \eqref{eq:weighted-single-dyadic-Linfty} and \eqref{eq:weighted-dyadic-Linfty-sum} are independent of \(\tau\) and \(f\).
	\end{lemma}
	
	\begin{proof}
		Note that
		\begin{equation}
			\label{eq:A-moderate-weight}
			A_\tau(x)
			=
			1+\tau+|x|
			\le
			(1+|x-y|)(1+\tau+|y|)
			=
			(1+|x-y|)A_\tau(y).
		\end{equation}
		Let \(K_{-1}\) be the kernel of \(P_{-1}\). Then, by
		\eqref{eq:A-moderate-weight},
		\[
		\begin{aligned}
			A_\tau(x)|P_{-1}f(x)|
			&\le
			\int_{\mathbb R^3}
			|K_{-1}(x-y)|A_\tau(x)|f(y)|\,dy
			\\
			&\ls
			\int_{\mathbb R^3}
			|K_{-1}(x-y)|(1+|x-y|)A_\tau(y)|f(y)|\,dy
			\\
			&\ls
			\|A_\tau f\|_{L^\infty_x}.
		\end{aligned}
		\]
		Hence,
		\[
		\|A_\tau P_{-1}f\|_{L^\infty_x}
		\ls
		\|A_\tau f\|_{L^\infty_x}.
		\]
		
		For \(k\ge0\), write the kernel of \(P_k\) as
		$K_k(x)=2^{3k}K(2^kx)$,
		where \(K\in\mathcal S(\mathbb R^3)\). Then the same argument gives
		\[
		\begin{aligned}
			A_\tau(x)|P_kf(x)|
			&\ls
			\int_{\mathbb R^3}
			|K_k(x-y)|(1+|x-y|)A_\tau(y)|f(y)|\,dy
			\\
			&\le
			\|A_\tau f\|_{L^\infty_x}
			\int_{\mathbb R^3}
			|K_k(z)|(1+|z|)\,dz .
		\end{aligned}
		\]
		Since
		\[
		\begin{aligned}
			\int_{\mathbb R^3}|K_k(z)|(1+|z|)\,dz
			&=
			\int_{\mathbb R^3}
			2^{3k}|K(2^kz)|(1+|z|)\,dz
			\\
			&=
			\int_{\mathbb R^3}
			|K(w)|(1+2^{-k}|w|)\,dw
			\ls 1
		\end{aligned}
		\]
		holds uniformly for \(k\ge0\), we obtain
		\[
		\|A_\tau P_kf\|_{L^\infty_x}
		\ls
		\|A_\tau f\|_{L^\infty_x},
		\qquad k\ge0.
		\]
		This proves \eqref{eq:weighted-single-dyadic-Linfty}.
		
		It remains to prove \eqref{eq:weighted-dyadic-Linfty-sum}.
		For \(k\ge 0 \), since \(P_k\) is supported away from the origin, there
		exist smooth Fourier multipliers $2^{k}P_k R|\nabla|^{-1}$, whose
		kernels are of the form
		\[
		L_{k}(x)=2^{3k}L(2^kx),
		\qquad L\in\mathcal S(\mathbb R^3),
		\]
		such that
		\[
		P_k f
		=
		-2^{-k}\cdot2^{k}P_k R|\nabla|^{-1}\cdot \nabla f .
		\]
		Applying the preceding kernel argument to \(2^{k}P_k R|\nabla|^{-1}\), we have that uniformly for \(k\ge 0\),
		\[
		\begin{aligned}
			\|A_\tau P_k f\|_{L^\infty_x}
			&\le
			2^{-k}
			\|A_\tau 2^{k}P_{[[k]]} R|\nabla|^{-1}\cdot P_k\nabla f\|_{L^\infty_x}
			\\
			&\ls
			2^{-k}
			\|A_\tau P_k\nabla f\|_{L^\infty_x}.
		\end{aligned}
		\]	
		\end{proof}

	\section{Technical lemmas}\label{sec:b}
	
	\begin{lemma}[{\bf A trace estimate on spheres}]\label{lem:sphere-trace-H32}
		Let $\kappa>0$. For any $g\in H^{3/2+\kappa}(\mathbb R^3)$, one has
		\begin{equation}\label{eq:trace-sphere-H32}
			\sup_{x\in\mathbb R^3,\ \rho\ge0}
			(1+\rho)
			\big(
			\int_{\mathbb S^2}|g(x+\rho\omega)|^2\,d\omega
			\big)^{1/2}
			\lesssim
			\|g\|_{H^{3/2+\kappa}(\mathbb R^3)} .
		\end{equation}
	\end{lemma}
	
	\begin{proof}
		By translation invariance, it suffices to prove that for each
		$h\in H^{3/2+\kappa}(\mathbb R^3)$,
		\begin{equation}\label{eq:trace-sphere-origin-H32}
			\sup_{\rho\ge0}
			(1+\rho)
			\big(
			\int_{\mathbb S^2}|h(\rho\omega)|^2\,d\omega
			\big)^{1/2}
			\lesssim
			\|h\|_{H^{3/2+\kappa}(\mathbb R^3)} .
		\end{equation}

		We first prove \eqref{eq:trace-sphere-origin-H32} for
		$h\in C_0^\infty(\mathbb R^3)$.
		
		For $0\le\rho\le1$, by the Sobolev embedding
		$H^{3/2+\kappa}(\mathbb R^3)\hookrightarrow L^\infty(\mathbb R^3)$,
		one has
		\begin{equation*}
			\begin{aligned}
				(1+\rho)
				\big(
				\int_{\mathbb S^2}|h(\rho\omega)|^2\,d\omega
				\big)^{1/2}
				\lesssim
				\|h\|_{L^\infty(\mathbb R^3)}
				\lesssim
				\|h\|_{H^{3/2+\kappa}(\mathbb R^3)} .
			\end{aligned}
		\end{equation*}
		
		It remains to consider the case of $\rho\ge1$. Define
		\begin{equation*}
			F(r)
			=
			\int_{|y|=r}|h(y)|^2\,dS_y
			=
			r^2\int_{\mathbb S^2}|h(r\omega)|^2\,d\omega,
			\qquad r>0.
		\end{equation*}
		Due to $h\in C_0^\infty(\mathbb R^3)$, one has
		\begin{equation*}
			F(r)\to0
			\qquad \text{as } r\to\infty .
		\end{equation*}
		Therefore, for $\rho\ge1$,
		\begin{equation}\label{eq:F-rho-integral-H32}
			F(\rho)
			=-\int_\rho^\infty F'(s)\,ds.
		\end{equation}
		Note that
		\begin{equation}\label{eq:F-derivative-H32}
			\begin{aligned}
				F'(s)
				&=
				\frac{d}{ds}
				\left(
				s^2\int_{\mathbb S^2}|h(s\omega)|^2\,d\omega
				\right)
				\\
				&=
				2s\int_{\mathbb S^2}|h(s\omega)|^2\,d\omega
				+
				2s^2
				\int_{\mathbb S^2}
				h(s\omega)\partial_r h(s\omega)\,d\omega,
			\end{aligned}
		\end{equation}
		where $\partial_r h(s\omega)=\omega\cdot\nabla h(s\omega)$.

		It follows from  $\rho\ge1$ and  $s\ge\rho$ that $s\ge1$ and $s\le s^2$.
		Thus,
		\begin{equation}\label{eq:Fprime-first-term-H32}
			\begin{aligned}
				\int_\rho^\infty&
				s\int_{\mathbb S^2}|h(s\omega)|^2\,d\omega\,ds
				\le
				\int_\rho^\infty
				s^2\int_{\mathbb S^2}|h(s\omega)|^2\,d\omega\,ds\\
				&\le
				\int_0^\infty
				s^2\int_{\mathbb S^2}|h(s\omega)|^2\,d\omega\,ds=\|h\|_{L^2(\mathbb R^3)}^2.
			\end{aligned}
		\end{equation}
		In addition,
		\begin{equation}\label{eq:Fprime-second-term-H32}
			\begin{aligned}
				&\int_\rho^\infty
				s^2
				\big(
				\int_{\mathbb S^2}|h(s\omega)|^2\,d\omega
				\big)^{1/2}
				\big(
				\int_{\mathbb S^2}|\partial_r h(s\omega)|^2\,d\omega
				\big)^{1/2}
				ds
				\\
				&\qquad\le
				\big(
				\int_\rho^\infty
				s^2\int_{\mathbb S^2}|h(s\omega)|^2\,d\omega\,ds
				\big)^{1/2}
				\big(
				\int_\rho^\infty
				s^2\int_{\mathbb S^2}|\partial_r h(s\omega)|^2\,d\omega\,ds
				\big)^{1/2}
				\\
				&\qquad\le
				\|h\|_{L^2(\mathbb R^3)}
				\|\nabla h\|_{L^2(\mathbb R^3)} .
			\end{aligned}
		\end{equation}
		Combining \eqref{eq:F-rho-integral-H32}-\eqref{eq:Fprime-second-term-H32} yields
		\begin{equation}\label{eq:F-rho-bound-H32}
			\begin{aligned}
				F(\rho)
				&\le
				\int_\rho^\infty |F'(s)|\,ds
				\\
				&\lesssim
				\|h\|_{L^2(\mathbb R^3)}^2
				+
				\|h\|_{L^2(\mathbb R^3)}
				\|\nabla h\|_{L^2(\mathbb R^3)}
				\\
				&\lesssim
				\|h\|_{H^1(\mathbb R^3)}^2
				\\
				&\lesssim
				\|h\|_{H^{3/2+\kappa}(\mathbb R^3)}^2.
			\end{aligned}
		\end{equation}
		Because of $F(\rho)=\rho^2\int_{\mathbb S^2}|h(\rho\omega)|^2\,d\omega$,
		we conclude that for $\rho\ge1$,
		\begin{equation*}
			\begin{aligned}
				(1+\rho)
				\big(
				\int_{\mathbb S^2}|h(\rho\omega)|^2\,d\omega
				\big)^{1/2}
				&=\frac{1+\rho}{\rho}F(\rho)^{1/2}\lesssim
				\|h\|_{H^{3/2+\kappa}(\mathbb R^3)}.
			\end{aligned}
		\end{equation*}
		Together with the estimate for $0\le\rho\le1$, this proves
		\eqref{eq:trace-sphere-origin-H32} for $h\in C_0^\infty(\mathbb R^3)$.
		
		We now show \eqref{eq:trace-sphere-origin-H32} for $h\in H^{3/2+\kappa}(\mathbb R^3)$.
		Let $\{h_n\}_{n=1}^\infty\subset C_0^\infty(\mathbb R^3)$ satisfy
		\begin{equation*}
			h_n\to h
			\qquad\text{in }H^{3/2+\kappa}(\mathbb R^3).
		\end{equation*}
		This implies
		\begin{equation*}
			h_n\to h
			\qquad\text{in }L^\infty(\mathbb R^3).
		\end{equation*}
		Hence, for each fixed $\rho\ge0$, it holds that
		\begin{equation*}
			\begin{aligned}
				\big(
				\int_{\mathbb S^2}|h_n(\rho\omega)-h(\rho\omega)|^2\,d\omega
				\big)^{1/2}
				&\le
				|\mathbb S^2|^{1/2}
				\|h_n-h\|_{L^\infty(\mathbb R^3)}
				\to0,
			\end{aligned}
		\end{equation*}
		which means
		\begin{equation*}
			\big(
			\int_{\mathbb S^2}|h_n(\rho\omega)|^2\,d\omega
			\big)^{1/2}
			\to
			\big(
			\int_{\mathbb S^2}|h(\rho\omega)|^2\,d\omega
			\big)^{1/2}.
		\end{equation*}

		Due to
		\begin{equation*}
			(1+\rho)
			\big(
			\int_{\mathbb S^2}|h_n(\rho\omega)|^2\,d\omega
			\big)^{1/2}
			\lesssim
			\|h_n\|_{H^{3/2+\kappa}(\mathbb R^3)},
		\end{equation*}
		we have that as $n\to\infty$,
		\begin{equation*}
			(1+\rho)
			\big(
			\int_{\mathbb S^2}|h(\rho\omega)|^2\,d\omega
			\big)^{1/2}
			\lesssim
			\|h\|_{H^{3/2+\kappa}(\mathbb R^3)},
		\end{equation*}
		which yields \eqref{eq:trace-sphere-origin-H32}.
	\end{proof}

	\begin{lemma}[{\bf Weighted Hardy inequality}]
		\label{lem:weighted-Hardy-initial-displacement}
		Let \(\mu>-\frac32\). Suppose $f\in L^2(\mathbb R^3)\cap H^1_{\mathrm{loc}}(\mathbb R^3)$ and
		$\langle x\rangle^{1+\mu}\nabla f\in L^2(\mathbb R^3)$.
		Then
		\begin{equation}
			\label{eq:weighted-Hardy-for-initial-displacement}
			\left\|
			\langle x\rangle^\mu f
			\right\|_{L^2_x}
			\le
			\frac{2}{\min\{3,3+2\mu\}}
			\left\|
			\langle x\rangle^{1+\mu}\nabla f
			\right\|_{L^2_x}.
		\end{equation}
	\end{lemma}
	
	\begin{proof}
		Set
		$
		\mu_+=\max\{\mu,0\}.
		$
		For \(\varepsilon>0\), define
		\begin{equation*}
			W_\varepsilon(x)
			=
			\langle x\rangle^{2\mu}
			\left(
			1+\varepsilon\langle x\rangle
			\right)^{-2\mu_+}
		\quad\text{and} \quad
		X_\varepsilon(x)
			=xW_\varepsilon(x).
		\end{equation*}
		Then \(W_\varepsilon\in L^\infty(\mathbb R^3)\) for each fixed
		\(\varepsilon>0\), meanwhile,
		\begin{equation}
			\label{eq:weighted-Hardy-regularized-divergence}
			\begin{aligned}
				\operatorname{div}X_\varepsilon
				&=\big(3+2\mu\frac{|x|^2}{\langle x\rangle^2}
				-2\mu_+\frac{\varepsilon |x|^2}
				{\langle x\rangle(1+\varepsilon\langle x\rangle)}
				\big)W_\varepsilon.
			\end{aligned}
		\end{equation}
		We claim that
		\begin{equation}
			\label{eq:weighted-Hardy-divergence-lower}
			\operatorname{div}X_\varepsilon
			\ge
			c_\mu W_\varepsilon,
			\qquad
			c_\mu=\min\{3,3+2\mu\}>0.
		\end{equation}
		Indeed, if \(\mu\ge0\), then \(\mu_+=\mu\), and the coefficient in
		\eqref{eq:weighted-Hardy-regularized-divergence} fulfills
		\[
		3
		+
		2\mu
		\frac{|x|^2}
		{\langle x\rangle^2(1+\varepsilon\langle x\rangle)}
		\ge 3.
		\]
		If \(-\frac32<\mu<0\), then \(\mu_+=0\), and
		\[
		3
		+
		2\mu
		\frac{|x|^2}{\langle x\rangle^2}
		\ge
		3+2\mu.
		\]
		Then \eqref{eq:weighted-Hardy-divergence-lower} is shown.
		
		Let \(\chi\in C_0^\infty(\mathbb R^3)\) satisfy
		\begin{equation*}
			0\le\chi\le1,
			\qquad
			\chi=1
			\quad\text{on }\{|x|\le1\},
			\qquad
			\chi=0
			\quad\text{on }\{|x|\ge2\}.
		\end{equation*}
		Set $\chi_R(x)=\chi\left(\frac{x}{R}\right)$.
		Due to \(f\in H^1_{\mathrm{loc}}(\mathbb R^3)\), one has
		$|f|^2\in W^{1,1}_{\mathrm{loc}}(\mathbb R^3)$ and
		$\nabla |f|^2
		=2f\,\nabla f$.
		In addition,
		\[
		\chi_R^2X_\varepsilon
		\in C_0^1(\mathbb R^3;\mathbb R^3).
		\]
		Due to
		\begin{equation}
			\label{eq:weighted-Hardy-weak-integration-by-parts}
			\begin{aligned}
				\int_{\mathbb R^3}
				\operatorname{div}
				\left(
				\chi_R^2X_\varepsilon
				\right)
				|f|^2\,dx
				=-2\int_{\mathbb R^3}
				\chi_R^2X_\varepsilon
				\cdot\nabla f\,
				f\,dx
			\end{aligned}
		\end{equation}
		and \eqref{eq:weighted-Hardy-divergence-lower}, we have
		\begin{equation}
			\label{eq:weighted-Hardy-cutoff-identity}
			\begin{aligned}
				c_\mu
				\int_{\mathbb R^3}
				W_\varepsilon\chi_R^2|f|^2\,dx
				&\le
				-2\int_{\mathbb R^3}
				\chi_R^2X_\varepsilon
				\cdot\nabla f\,
				 f\,dx
				-2\int_{\mathbb R^3}
				\chi_RX_\varepsilon
				\cdot\nabla\chi_R
				|f|^2\,dx.
			\end{aligned}
		\end{equation}
		
		Note that \(W_\varepsilon\in L^\infty(\mathbb R^3)\) for fixed
		\(\varepsilon>0\) and
		$|\nabla\chi_R|\lesssim	R^{-1}\boldsymbol 1_{\{R\le |x|\le2R\}}$.
		Then one has that as \(R\to\infty\),
		\begin{equation}
			\label{eq:weighted-Hardy-cutoff-error-vanishes}
			\begin{aligned}
				\big|
				\int_{\mathbb R^3}
				\chi_RX_\varepsilon
				\cdot\nabla\chi_R
				|f|^2\,dx
				\big|
				&\lesssim
				\int_{\{R\le |x|\le2R\}}
				|f|^2\,dx
				\longrightarrow0.
			\end{aligned}
		\end{equation}
		On the other hand, by
		$|x|^2W_\varepsilon(x)\le
		\langle x\rangle^{2+2\mu}$, it holds that
		\begin{equation}
			\label{eq:weighted-Hardy-main-term}
			\begin{aligned}
				\big|
				\int_{\mathbb R^3}
				\chi_R^2X_\varepsilon
				\cdot\nabla f\,
				f\,dx
				\big|
				&\le
				\big(
				\int_{\mathbb R^3}
				W_\varepsilon\chi_R^4|f|^2\,dx
				\big)^{1/2}
				\big(
				\int_{\mathbb R^3}
				|x|^2W_\varepsilon
				|\nabla f|^2\,dx
				\big)^{1/2}\\
				&\le
				\big(
				\int_{\mathbb R^3}
				W_\varepsilon\chi_R^2|f|^2\,dx
				\big)^{1/2}
				\big\|
				\langle x\rangle^{1+\mu}
				\nabla f
				\big\|_{L^2_x}.
			\end{aligned}
		\end{equation}
		
		Letting \(R\to\infty\) in
		\eqref{eq:weighted-Hardy-cutoff-identity}, and using
		\eqref{eq:weighted-Hardy-cutoff-error-vanishes} and
		\eqref{eq:weighted-Hardy-main-term}, we get
		\[
		c_\mu
		\big\|
		W_\varepsilon^{1/2}f
		\big\|_{L^2_x}^2
		\le
		2
		\big\|
		W_\varepsilon^{1/2}f
		\big\|_{L^2_x}
		\big\|
		\langle x\rangle^{1+\mu}
		\nabla f
		\big\|_{L^2_x}.
		\]
		Hence,
		\begin{equation}
			\label{eq:weighted-Hardy-epsilon-estimate}
			\big\|
			W_\varepsilon^{1/2}f
			\big\|_{L^2_x}
			\le
			\frac{2}{c_\mu}
			\big\|
			\langle x\rangle^{1+\mu}
			\nabla f
			\big\|_{L^2_x}.
		\end{equation}
		It follows from
		$W_\varepsilon(x)\to\langle x\rangle^{2\mu}$
		as 	$\varepsilon\to0^+$,
		Fatou's lemma and \eqref{eq:weighted-Hardy-epsilon-estimate} that
		\[
		\left\|
		\langle x\rangle^\mu f
		\right\|_{L^2_x}
		\le
		\frac{2}{c_\mu}
		\left\|
		\langle x\rangle^{1+\mu}
		\nabla f
		\right\|_{L^2_x},
		\]
		which yields \eqref{eq:weighted-Hardy-for-initial-displacement}.
	\end{proof}

	\begin{lemma}
		\label{lem:weighted-data-in-Sh}
		Let $\mu>-\f32$. Suppose that
		$f\in L^2(\R^3)\cap H^1_{\mathrm{loc}}(\R^3)$
		and $\langle x\rangle^{1+\mu}\nabla f\in L^2(\R^3)$.
		Then $\langle x\rangle^{1+\mu}f\in
			L^6(\R^3)\subset\mathcal S'_h(\R^3)$
		and
		\begin{equation}
			\label{eq:weighted-data-L6-estimate}
			\left\|
			\langle x\rangle^{1+\mu}f
			\right\|_{L^6_x}
			+
			\left\|
			\langle x\rangle^{1+\mu}f
			\right\|_{\dot{H}^1_x}
			\ls
			\left\|
			\langle x\rangle^{1+\mu}\nabla f
			\right\|_{L^2_x}.
		\end{equation}
	\end{lemma}
	
	\begin{proof}
		Set $G=\langle x\rangle^{1+\mu}f$.
		By Lemma~\ref{lem:weighted-Hardy-initial-displacement},
		\begin{equation}
			\label{eq:weighted-data-Hardy}
			\left\|
			\langle x\rangle^\mu f
			\right\|_{L^2_x}
			\ls
			\left\|
			\langle x\rangle^{1+\mu}\nabla f
			\right\|_{L^2_x}.
		\end{equation}
		Due to $\nabla G=\langle x\rangle^{1+\mu}\nabla f
			+(1+\mu)x\langle x\rangle^{\mu-1}f$,
		one has
		\begin{equation}
			\label{eq:weighted-data-gradient}
			\begin{aligned}
				\|\nabla G\|_{L^2_x}
				&\ls
				\left\|
				\langle x\rangle^{1+\mu}\nabla f
				\right\|_{L^2_x}
				+
				\left\|
				\langle x\rangle^\mu f
				\right\|_{L^2_x}
				\ls
				\left\|
				\langle x\rangle^{1+\mu}\nabla f
				\right\|_{L^2_x}.
			\end{aligned}
		\end{equation}
		Moreover,
		\begin{equation}
			\label{eq:weighted-data-over-x}
			\big\|
			\frac{G}{\langle x\rangle}
			\big\|_{L^2_x}
			=
			\left\|
			\langle x\rangle^\mu f
			\right\|_{L^2_x}
			\ls
			\left\|
			\langle x\rangle^{1+\mu}\nabla f
			\right\|_{L^2_x}.
		\end{equation}
		
		Let $\chi\in C_0^\infty(\R^3)$ satisfy
		\begin{equation*}
			0\le\chi\le1,
			\qquad
			\chi(x)=1
			\quad\text{for }|x|\le1,
			\qquad
			\chi(x)=0
			\quad\text{for }|x|\ge2,
		\end{equation*}
		and set $\chi_R(x)
			=\chi\left(\frac{x}{R}\right)$ for
			$R\ge1$.
		
		By the Sobolev inequality,
		\begin{equation*}
			\label{eq:weighted-data-truncated-Sobolev}
			\begin{aligned}
				\|\chi_R G\|_{L^6_x}
				&\ls
				\|\nabla(\chi_R G)\|_{L^2_x}
				\ls
				\|\nabla G\|_{L^2_x}
				+
				\|G\nabla\chi_R\|_{L^2_x}.
			\end{aligned}
		\end{equation*}
		By $|\nabla\chi_R(x)|\ls R^{-1}
			\boldsymbol 1_{\{R\le|x|\le2R\}}$,
		it follows that
		\begin{equation*}
			\label{eq:weighted-data-cutoff-error}
			\begin{aligned}
				\|G\nabla\chi_R\|_{L^2_x}
				&\ls
				\big\|
				\boldsymbol 1_{\{R\le|x|\le2R\}}
				\frac{G}{\langle x\rangle}
				\big\|_{L^2_x}
				\longrightarrow0
				\qquad
				\text{as }R\to\infty,
			\end{aligned}
		\end{equation*}
		where \eqref{eq:weighted-data-over-x} is used.
		Therefore, by Fatou's lemma, it holds that
		\begin{equation}
			\label{eq:weighted-data-L6}
			\begin{aligned}
				\|G\|_{L^6_x}
				&\le
				\liminf_{R\to\infty}
				\|\chi_R G\|_{L^6_x}
				\ls
				\|\nabla G\|_{L^2_x}
				\ls
				\left\|
				\langle x\rangle^{1+\mu}\nabla f
				\right\|_{L^2_x}.
			\end{aligned}
		\end{equation}
		Combining \eqref{eq:weighted-data-gradient} and \eqref{eq:weighted-data-L6} yields \eqref{eq:weighted-data-L6-estimate}.
		
	\end{proof}

	\vskip 0.2 true cm
	{\bf \color{blue}{Conflict of interest}}
	\vskip 0.2 true cm
	
	{\bf On behalf of all authors, the corresponding author states that there is no conflict of interest.}
	
	\vskip 0.2 true cm

	{\bf \color{blue}{Data availability}}
	
	\vskip 0.2 true cm
	
	{\bf Data sharing is not applicable to this article as no new data were created.}

	\end{document}